\documentclass{amsart}       

\usepackage{graphicx}

\usepackage{gensymb,epic,eepic,float,fullpage}
\usepackage{longtable}
\usepackage{amsmath,amsfonts,amssymb}
\usepackage{mathtools}
\usepackage[colorinlistoftodos]{todonotes}
\usepackage{cancel}
\usepackage{stmaryrd}
\usepackage{url}
\usepackage{tikz}
\usepackage[outline]{contour}
\usetikzlibrary{graphs,patterns,decorations.markings,arrows.meta,matrix}
\usetikzlibrary{calc,decorations.pathmorphing,decorations.pathreplacing,shapes}
\usepackage{graphicx}

\usepackage{hyperref}

\usepackage{bm}
\usepackage{bbm}
\usepackage{mathrsfs}

\colorlet{lgray}{white!85!black}
 \colorlet{dgray}{white!45!black}
\colorlet{lred}{white!85!red}
\colorlet{dred}{white!35!red}
\colorlet{lgreen}{white!60!green}
\colorlet{dgreen}{black!30!green}
\colorlet{lpurple}{white!60!purple}
\colorlet{lblue}{white!60!blue}
\definecolor{green}{rgb}{0.1,0.8,0.1}
\definecolor{yellow}{rgb}{1.0,0.85,0.25}
\definecolor{purple}{rgb}{1.0, 0, 1.0}
\definecolor{blue}{rgb}{0, 0, 1.0}

\tikzstyle{unfused}=[line width=1.5pt, draw=lgray, arrows={-Stealth[scale=0.6, length=10, width=10, dgray]}]
\tikzstyle{unfused*}=[lgray, line width=1.5pt]
\tikzstyle{fused}=[lgray, line width=4pt, arrows={-Stealth[scale=1.1,length=10, width=10,dgray]}]
\tikzstyle{fused*}=[lgray, line width=4pt]
\tikzstyle{cont}=[lred, line width=4pt, arrows={-Stealth[scale=1.1,length=10,width=10,dred]}]
\tikzstyle{dual}=[black, line width=1pt, dashed]
\tikzstyle{lightdual}=[black, line width=0.5pt, dashed]
\tikzstyle{cut}=[black, line width=1.0pt]

 \renewcommand{\tikz}[2]{
\begin{tikzpicture}[scale=#1,baseline=(current bounding box.center),>=stealth]
#2
\end{tikzpicture}}

\newcommand{\tikzbase}[3]{
\begin{tikzpicture}[scale=#1,baseline={([yshift=#2]current bounding box.center)},>=stealth]
#3
\end{tikzpicture}}

\newcommand\restr[2]{{
  \left.\kern-\nulldelimiterspace 
  #1 
  \vphantom{\big|} 
  \right|_{#2} 
 }}
 
 \DeclarePairedDelimiter\floor{\lfloor}{\rfloor}

\newtheorem{prop}{Proposition}[section]
\newtheorem{theo}[prop]{Theorem}
\newtheorem{conj}[prop]{Conjecture}
\newtheorem{lem}[prop]{Lemma}
\newtheorem{cor}[prop]{Corollary}

\newtheorem{ass}[prop]{Assumption}
\newtheorem{rem}[prop]{Remark}

\numberwithin{equation}{section}

\newcommand{\be}{\begin{equation*}}
\newcommand{\ee}{\end{equation*}}
\renewcommand{\(}{\left (}
\renewcommand{\)}{\right )}
\newcommand{\1}{\mathbbm 1}
\renewcommand{\i}{\mathbf i}

\DeclareMathOperator{\Beta}{Beta}
\DeclareMathOperator{\sign}{sign}
\DeclareMathOperator{\Row}{Row}
\DeclareMathOperator{\Col}{Col}
\DeclareMathOperator{\Diag}{Diag}

\newcommand{\ve}{\varepsilon}

\newcommand{\bA}{\bm A}
\newcommand{\bB}{\bm B}
\newcommand{\bC}{\bm C}
\newcommand{\bD}{\bm D}
\newcommand{\bI}{\bm I}
\newcommand{\bJ}{\bm J}
\newcommand{\bK}{\bm K}
\newcommand{\bL}{\bm L}
\newcommand{\bR}{\bm R}
\newcommand{\bP}{\bm P}
\newcommand{\bX}{\bm X}
\newcommand{\bY}{\bm Y}
\newcommand{\bZ}{\bm Z}

\newcommand{\e}{\mathbf e}
\newcommand{\bw}{\mathbf w}

\newcommand{\bc}{\mathbf c}
\newcommand{\bx}{\mathbf x}
\newcommand{\by}{\mathbf y}

\newcommand{\E}{\mathbb E}
\renewcommand{\P}{\mathbb P}

\newcommand{\A}{\mathcal A}
\newcommand{\B}{\mathcal B}
\newcommand{\W}{\mathcal W}
\newcommand{\K}{\mathcal K}
\newcommand{\D}{\mathcal D}
\newcommand{\Z}{\mathfrak Z}

\newcommand{\tv}{\widetilde{v}}
\newcommand{\tz}{\widetilde{z}}

\let\Re\relax
\DeclareMathOperator{\Re}{Re}
\let\Im\relax
\DeclareMathOperator{\Im}{Im}
\DeclareMathOperator{\Ai}{Ai}

\newcommand{\Q}{\mathcal Q}
\newcommand{\C}{\mathcal C}
\renewcommand{\L}{\mathcal L}
\newcommand{\cH}{\mathcal H}
\renewcommand{\S}{{\sf S}}

\newcommand{\NN}{{\sf N}}
\newcommand{\MM}{{\sf M}}
\newcommand{\LL}{{\sf L}}

\begin{document}

\title{Hidden diagonal integrability of \MakeLowercase{q}-Hahn vertex model and Beta polymer model}

\author{Sergei Korotkikh}

\maketitle

\begin{abstract}
We study a new integrable probabilistic system, defined in terms of a stochastic colored vertex model on a square lattice. The main distinctive feature of our model is a new family of parameters attached to diagonals rather than to rows or columns, like in other similar models. Because of these new parameters the previously known results about vertex models cannot be directly applied, but nevertheless the integrability remains, and we prove explicit integral expressions for $q$-deformed moments of the (colored) height functions of the model. Following known techniques our model can be interpreted as a $q$-discretization of the Beta polymer model from \cite{BC15b} with a new family of parameters, also attached to diagonals. To demonstrate how integrability with respect to the new diagonal parameters works, we extend the known results about Tracy-Widom large-scale fluctuations of the Beta polymer model.
\end{abstract}

\tableofcontents

\section{Introduction}\label{introSect}

\subsection{Overview} During the last decades a new conjectural universality law has been actively researched under the name \emph{KPZ universality}, originating from \cite{KPZ86}. Unfortunately, despite the significant breakthroughs made in recent years we are still unable to prove or sometimes even formulate the desired universality statements. However, for some exceptional models one can directly verify the large-scale properties expected from the KPZ class models by finding explicit expressions describing the precise behavior of natural observables. Such expressions come from intricate algebraic or combinatorial properties, usually referred to as \emph{integrability}, and various recent examples of large-scale analysis of integrable models can be found in \cite{TW07}, \cite{TW08a}, \cite{TW08b}, \cite{ACQ10}, \cite{BCF12}, \cite{FV13},  \cite{BCG14}, \cite{CSS14}, \cite{OO14}, \cite{BC15a}, \cite{BC15b}, \cite{BR19}, \cite{BCD20} and references therein. 

We focus on two integrable models, namely, the \emph{$q$-Hahn vertex model} and the \emph{directed Beta polymer model}. The former model was introduced in \cite{Pov13} (see also  \cite{Cor14}, \cite{CP15} for alternative descriptions) and it belongs to the family of stochastic solvable vertex models, which are of a particular interest for the following twofold reason. On one hand, vertex models have a rich algebraic integrability structure coming from the celebrated Yang-Baxter equation for $\mathcal R$-matrix of the underlying quantum groups. In particular, this allows to find explicit integral expressions for $q$-deformed moments of naturally defined \emph{height functions}, \emph{cf} \cite{BCG14}, \cite{BP16}, \cite{BW20}, \cite{BK20}. On the other hand, the stochastic vertex models can be degenerated to numerous other integrable models, including ASEP, $q$-Hahn and $q$-boson particle systems, and the mentioned directed Beta polymer model, \emph{cf.} \cite{CP15}, \cite{BP16}, \cite{BGW19}. During this transition the explicit integral expressions for the vertex models reduce to expressions describing natural observables of the other models, which allow to explicitly study various properties. In this way integrability of a number of probabilistic systems can be inferred from the Yang-Baxter integrability of the vertex models. 

The directed Beta polymer was originally introduced in \cite{BC15b} using two equivalent descriptions: it can either be viewed as a polymer model with Beta distributed weights or as a random walk in a random time-dependent Beta environment. This dual description already distinguished the Beta polymer model among the other integrable discrete polymers, namely, log-Gamma \cite{Sep09}, strict-weak \cite{CSS14}, \cite{OO14} and inverse-Beta \cite{TLD15} polymers, which cannot be readily interpreted as random walks in random environments (RWRE).  Moreover, there is a whole cluster of integrable models obtainable as degenerations of the Beta polymer, this cluster includes strict-weak polymer, uniform sticky Brownian motions \cite{LJL03}, \cite{HW06}, \cite{BR19}, Bernoulli-Exponential directed first passage percolation \cite{BC15b} as well as other percolation models. Going in the other direction, the Beta polymer itself can be obtained as a $q\to 1$ limit of the $q$-Hahn vertex model, and this connection is closely related to the approach used originally in \cite{BC15b} to obtain integrability and asymptotics.\footnote{The $q$-Hahn vertex model is closely related to the $q$-Hahn TASEP, which was originally used to establish the integrability of the Beta polymer model.}  Overall, these numerous connections  to other models give rise to a lot of interesting properties and conjectures, making Beta polymer a popular object in recent research, \emph{cf.} \cite{LDT17}, \cite{BRS18}, \cite{BLD19}, \cite{BR19}.

The present work is  devoted to a new layer of integrability of the $q$-Hahn vertex model, which has remained unnoticed until now. Namely, we find a new family of parameters, which are attached to the diagonals of the square lattice defining the model and which can be added to the model without breaking the integrability. To capture the most general situation containing the new parameters we introduce  \emph{diagonally inhomogeneous colored $q$-Hahn vertex model}, and the first main result of this work is the proof of integral expressions for the $q$-deformed moments of \emph{the colored height functions} of this model. Such explicit expressions justify our claims about integrability and encapsulate the behavior of the model in a form suitable for subsequent analysis. As a direct application, we demonstrate that the shift-invariance property from \cite{BGW19}, \cite{Gal20} holds for our model, indicating that a broad class of symmetries is compatible with the additional parameters.

Existence of such integrable parameters attached to diagonals was surprising to us, and it is so far unique for the solvable vertex models: integrability is usually justified by the Yang-Baxter equation, which is well-suited only for the parameters attached to rows and columns of the model, but other parameters are typically incompatible with the Yang-Baxter equation. Accordingly, all other known solvable vertex models only have parameters attached to rows or columns. The $q$-Hahn model itself was known to have families of natural integrable column and row parameters coming from fusion of the six-vertex model, however, that construction did not indicate existence of the diagonal parameters in any way. Moreover, to the best of our knowledge the new diagonally inhomogeneous $q$-Hahn model cannot be obtained from any other model via known procedures. We believe that this is one of the reasons why the integrability with respect to diagonals has been hidden. 

As an immediate consequence of our first result, we are able to degenerate the diagonal integrability to the Beta polymer model, finding an integrable inhomogeneous extension containing three families of parameters, attached to columns, rows and diagonals. Previously, only Beta polymer models inhomogeneous along two directions at most were known to be integrable.
Moreover, to the best of our knowledge, this is the first integrable example of a directed polymer model inhomogeneous in three directions (there exist several related models with two families of parameters, studied in \cite{COSZ11}, \cite{TV18}, \cite{BCD20}). We expect that a lot of currently known properties of the Beta polymer can be extended to our inhomogeneous version, moreover, some new properties might emerge based on the new extension. To reinforce this claim, as the second main result of this work we find the Tracy-Widom asymptotics for the inhomogeneous model, extending one of the original results about the Beta polymer from \cite{BC15b}.

Below we briefly describe the models and our results.

\subsection{Main results: integral expressions for the diagonally inhomogeneous $q$-Hahn model} Let $q\in (0,1)$ and $(\mu_0, \mu_1, \mu_2, \dots)$, $(\kappa_1, \kappa_2, \dots)$, $(\lambda_1,\lambda_2,\dots)$ be real parameters satisfying
\be
\lambda_d<\kappa_j<\mu_i,\qquad i\in\mathbb Z_{\geq 0},\  j\in\mathbb{Z}_{>0},\ d\in\mathbb Z_{\geq 0}.
\ee
The \emph{colored diagonally inhomogeneous $q$-Hahn vertex model} consists of a random ensemble of colored up-right paths in a positive quadrant $\mathbb Z_{\geq 0}^2$, with colors labeled by positive integers. The paths are sampled according to the following procedure:
\begin{itemize}
\item All paths start from the left boundary of the model $\{0\}\times\mathbb Z_{\geq 1}$, with $b_j$ paths of color $j$ starting at $(0,j)$. The numbers $b_j$ are independent and distributed according to
\be 
\mathbb P(b_j=b)=(\kappa_j/\mu_0)^b\frac{(\lambda_j/\kappa_j;q)_b}{(q;q)_b}\frac{(\kappa_j/\mu_0;q)_\infty}{(\lambda_j/\mu_0;q)_\infty}, \qquad b\in\mathbb Z_{\geq 0}.
\ee
\item The remaining configuration is sampled sequentially and independently around each integer point $(i,j)$, starting from the bottom and going along the rows. At each integer point all paths coming from the left turn upwards, while the paths coming from below might turn right following the probabilities
\be
\mathbb P\left(\tikzbase{0.5}{-0.56ex}{
	\draw[lgray, line width=3pt] (-1,0) -- (1,0);
	\draw[lgray, line width=3pt] (0,-1) -- (0,1);
	\node[left] at (-1,0) {\tiny $\bB$};\node[right] at (1,0) {\tiny $\bD$};
	\node[below] at (0,-1) {\tiny $\bA$};\node[above] at (0,1) {\tiny $\bA+\bB-\bD$};
} \ \Bigg|\  \bA,\bB\right)=\ (\kappa_j/\mu_i)^{|\bD|}\frac{(\kappa_j/\mu_i;q)_{|\bA|-|\bD|}(\lambda_{j-i}/\kappa_j;q)_{|\bD|}}{(\lambda_{j-i}/\mu_i;q)_{|\bA|}}q^{\sum_{i<j}D_i(A_j-D_j)}\prod_{i=1}^n\binom{A_i}{D_i}_q,
\ee
where $\bA=(A_1, A_2, \dots)$ is a finite integer sequence with $A_c$ equal to the number of paths of color $c$ entering from below, $\bB,\bC,\bD$ similarly encode the colors of paths along the left, top and the right edges respectively, and we set $|\bA|=A_1+A_2+\dots, |\bD|=D_1+D_2+\dots$.
\end{itemize}
The name for the model comes from the probabilities above: one can notice that, for a specific choice of parameters, the probabilities resemble the orthogonality weights of the $q$-Hahn polynomials. On the other hand, the probabilities depend on the parameters $\mu_i,\kappa_j,\lambda_{j-i}$ attached to columns $i=const$, rows $j=const$ or diagonals $j-i=const$ respectively, hence we call the model inhomogeneous, see Figure \ref{qHahnFigureIntro} for a schematic description of the parameters. Finally, the diagonal parameters make our model unique among the other vertex models, hence inhomogeneity in this direction is emphasized. 

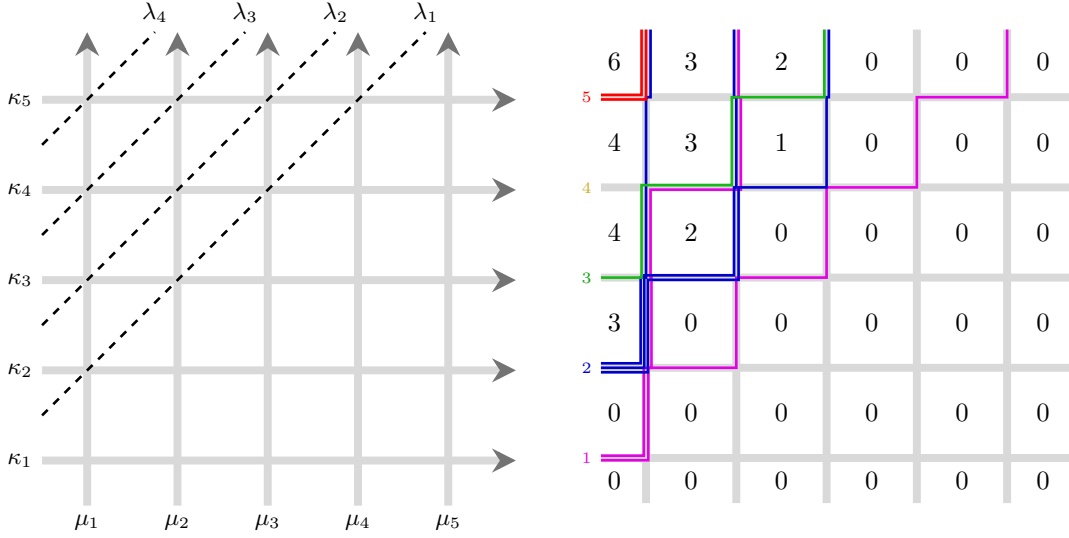
\begin{figure*}
\centerline{
\tikz{0.6}{
	\foreach\y in {1,...,5}{
		\draw[lgray, line width=3pt, arrows={-Stealth[scale=0.9,length=10,width=10,dgray]}] (1,2*\y) -- (11.5,2*\y);
	}
	\foreach\x in {1,...,5}{
		\draw[lgray, line width=3pt, arrows={-Stealth[scale=0.9,length=10,width=10,dgray]}] (2*\x,1) -- (2*\x,11.5);
	}
	\draw[black, line width = 1pt, dashed] (1,3) -- (9.5, 11.5) node[above] {\small $\lambda_1$};
	\draw[black, line width = 1pt, dashed] (1,5) -- (7.5, 11.5) node[above] {\small $\lambda_2$};
	\draw[black, line width = 1pt, dashed] (1,7) -- (5.5, 11.5) node[above] {\small $\lambda_3$};
	\draw[black, line width = 1pt, dashed] (1,9) -- (3.5, 11.5) node[above] {\small $\lambda_4$};
	\node[left] at (1,2) {\small $\kappa_1$};
	\node[left] at (1,4) {\small $\kappa_2$};
	\node[left] at (1,6) {\small $\kappa_3$};
	\node[left] at (1,8) {\small $\kappa_4$};
	\node[left] at (1,10) {\small $\kappa_5$};
	\node[below] at (10,1) {\small$\mu_5$};
	\node[below] at (8,1) {\small$\mu_4$};
	\node[below] at (6,1) {\small$\mu_3$};
	\node[below] at (4,1) {\small$\mu_2$};
	\node[below] at (2,1) {\small$\mu_1$};
}\qquad
\begin{tikzpicture}[xscale=0.6, yscale=0.6, baseline={([yshift=0]current bounding box.center)}]
	\foreach\y in {1,...,5}{
		\draw[lgray, line width=3pt] (1,2*\y) -- (11.5,2*\y);
	}
	\foreach\x in {1,...,5}{
		\draw[lgray, line width=3pt] (2*\x,1) -- (2*\x,11.5);
	}	
	\draw[purple!90!black, line width = 1pt] (1,1.95) -- (2.05,1.95) --  (2.05, 4) -- (4, 4) -- (4, 6) -- (6, 6) -- (6, 8) -- (8, 8) -- (8,10) -- (10, 10) -- (10, 11.5);
	\draw[purple!90!black, line width = 1pt] (1,2.05) -- (1.95, 2.05) -- (1.95, 4) -- (2.12, 4) -- (2.12, 6) -- (2.1, 6) -- (2.1, 7.95) -- (4.1,7.95) -- (4.1, 10) -- (4.05, 10) -- (4.05,11.5);
	
	\draw[blue!80!black, line width = 1pt] (1, 3.9) -- (2.04, 3.9) -- (2.04, 5.95) -- (4.05, 5.95) -- (4.05, 8) -- (6, 8) -- (6, 10) -- (6.05, 10) -- (6.05, 11.5);
	\draw[blue!80!black, line width = 1pt] (1, 4) -- (1.96, 4) -- (1.96, 6.05) -- (3.95, 6.05) -- (3.95, 8) -- (4, 8) -- (4, 10) -- (3.95, 10) -- (3.95, 11.5);
	\draw[blue!80!black, line width = 1pt] (1, 4.1) -- (1.88, 4.1) -- (1.88, 6) -- (2, 6) -- (2, 10) -- (2.1, 10) -- (2.1, 11.5);
	
	\draw[green!90!black, line width = 1pt] (1, 6) -- (1.9, 6) -- (1.9, 8.05) -- (3.9, 8.05) -- (3.9, 10) -- (5.95, 10) -- (5.95, 11.5);
	
	\draw[red, line width = 1pt] (1, 9.95) -- (2, 9.95) -- (2, 11.5);
	\draw[red, line width = 1pt] (1, 10.05) -- (1.9, 10.05) -- (1.9, 11.5);
	
	\node[left, purple!90!black] at (1,2) {\tiny $1$};
	\node[left, blue!80!black] at (1,4) {\tiny $2$};
	\node[left, green!80!black] at (1,6) {\tiny $3$};
	\node[left, yellow!80!black] at (1,8) {\tiny $4$};
	\node[left, red] at (1,10) {\tiny $5$};
	
	\node[black] at (1.3, 1.5) {$0$};
	\node[black] at (1.3, 3) {$0$};
	\node[black] at (1.3, 5) {$3$};
	\node[black] at (1.3, 7) {$4$};
	\node[black] at (1.3, 9) {$4$};
	\node[black] at (1.3, 10.8) {$6$};
	
	\node[black] at (3, 1.5) {$0$};
	\node[black] at (3, 3) {$0$};
	\node[black] at (3, 5) {$0$};
	\node[black] at (3, 7) {$2$};
	\node[black] at (3, 9) {$3$};
	\node[black] at (3, 10.8) {$3$};

	\node[black] at (5, 1.5) {$0$};
	\node[black] at (5, 3) {$0$};
	\node[black] at (5, 5) {$0$};
	\node[black] at (5, 7) {$0$};
	\node[black] at (5, 9) {$1$};
	\node[black] at (5, 10.8) {$2$};

	\node[black] at (7, 1.5) {$0$};
	\node[black] at (7, 3) {$0$};
	\node[black] at (7, 5) {$0$};
	\node[black] at (7, 7) {$0$};
	\node[black] at (7, 9) {$0$};
	\node[black] at (7, 10.8) {$0$};
	
	\node[black] at (9, 1.5) {$0$};
	\node[black] at (9, 3) {$0$};
	\node[black] at (9, 5) {$0$};
	\node[black] at (9, 7) {$0$};
	\node[black] at (9, 9) {$0$};
	\node[black] at (9, 10.8) {$0$};
	
	\node[black] at (10.8, 1.5) {$0$};
	\node[black] at (10.8, 3) {$0$};
	\node[black] at (10.8, 5) {$0$};
	\node[black] at (10.8, 7) {$0$};
	\node[black] at (10.8, 9) {$0$};
	\node[black] at (10.8, 10.8) {$0$};
	
\end{tikzpicture}}
\caption{\label{qHahnFigureIntro} Left: parameters of the $q$-Hahn model. Right: a configuration of the model along with the values of the height function $h_{\geq 2}$ for this configuration.}
\end{figure*}

 The configurations of our model can be described in terms of \emph{colored height functions} $h_{\geq c}^{(x,y)}$, which are assigned to the facets $(x,y)\in\(\mathbb Z+\frac{1}{2}\)^2$ formed by rows and columns and count the number of paths of color $\geq c$ passing below the facet, as depicted on Figure \ref{qHahnFigureIntro}. We consider the following observable: for collections 
 \be
x_1 \leq x_2\leq\dots\leq x_k,\qquad y_1\geq y_2\geq\dots\geq y_k, \qquad\qquad x_a,y_a\in\mathbb Z_{\geq 0}+\frac{1}{2};
\ee
\be
c_1\leq c_2\leq\dots\leq c_k,\qquad c_a\in\mathbb Z_{\geq 0};
\ee
and a permutation $\tau\in S_k$ set $\tau.\bc=(c_{\tau^{-1}(1)}, c_{\tau^{-1}(2)}, \dots, c_{\tau^{-1}(k)})$ and
\be
\Q_{\geq \tau.\bc}^{(\bx,\by)}:=\prod_{a=1}^kq^{h^{(x_{\tau(a)},y_{\tau(a)})}_{\geq c_{a}}}.
\ee
\begin{theo}[Theorem \ref{qHahnResultTheo} in the text] \label{theoFirstIntro}With the notation above we have
\begin{multline}\label{introIntegral}
\E\left[\Q^{\bx,\by}_{\geq\tau.\bc}(\Sigma)\right]=\frac{(-1)^kq^{\frac{k(k-1)}{2}-l(\tau)}}{(2\pi\i)^k}\oint_{\Gamma_1}\cdots\oint_{\Gamma_k}\prod_{a<b}\frac{w_b-w_a}{w_b-qw_a}\ T_\tau\(\prod_{a=1}^k\prod_{i=1}^{c_a-1}\frac{1-\lambda_iw_a}{1-\kappa_iw_a}\)\\
\times \prod_{a=1}^k\(\prod_{i=0}^{i<x_a}\frac{1}{1-\mu_iw_a}\prod_{j=1}^{j<y_a}(1-\kappa_jw_a)\prod_{d=1}^{d\leq y_a-x_a}\frac{1}{1-\lambda_dw_a}\)\frac{dw_a}{w_a},
\end{multline}
where the positively oriented contours $\Gamma_b$ are chosen to enclose $\mu_i^{-1}$ and $q\Gamma_a$ for $a<b$, with other singularities being outside of the contours, and $T_{\tau}$ denotes the polynomial representation of the Hecke algebra generated by the Demazure-Lusztig operators
\be
T_i=q+\frac{w_{i+1}-qw_{i}}{w_{i+1}-w_{i}}(\mathfrak{s}_i-1),\qquad \mathfrak{s}_if(\dots, w_i, w_{i+1}, \dots)=f(\dots, w_{i+1}, w_{i}, \dots).
\ee
\end{theo}

The proof is based on the idea which has already appeared  in \cite{BK20}: we show that both sides of \eqref{introIntegral} satisfy the same discrete recurrence relations, called \emph{local relations}. However, the required local relations and our way of proving them are novel: we use Yang-Baxter equations for this purpose, directly connecting the existence of local relations with the integrability of vertex models. Surprisingly, one of the needed equations is the recently discovered \emph{deformed Yang-Baxter equation} from \cite{BK21}, which has appeared in the context of spin $q$-Whittaker functions.

Theorem \ref{theoFirstIntro} and the recurrence relations based on the Yang-Baxter equations indicate that our model is integrable, however, they do not clearly explain where the integrability and the diagonal parameters come from.  To partially remedy this, let us briefly outline a more constructive approach to Theorem \ref{theoFirstIntro}, which however does not work in full generality. The main idea comes from \cite{BP16}, \cite{BW20}: By definition, the average of an observable $\mathcal O$ can be written as
\be
\E\left[\mathcal O\right]=\sum_{\lambda} \mathcal O(\lambda) \mathbb P[\lambda].
\ee
It turns out that for certain observables $\mathcal O$ similar to $\Q_{\geq \tau.\bc}^{(\bx,\by)}$ one can relate both $\mathcal O(\lambda)$ and $\mathbb P[\lambda]$ to spin deformations of Hall-Littlewood or $q$-Whittaker symmetric functions, which are constructed using higher spin six-vertex model (\emph{cf.} \cite{Bor14}, \cite{BP16}, \cite{BW18}) or $q$-Hahn model (\emph{cf.} \cite{BW17}, \cite{MP20}, \cite{BK21}). In this situation to compute  $\E\left[\mathcal O\right]$ one can first find a suitable Cauchy-type summation identity, which identifies the sum above with a single evaluation of a higher-spin symmetric function, and then find an integral expression for the latter.  

For this constructive approach to work, one needs Cauchy-type summation identities and integral expressions for higher-spin functions, which in the context of vertex models are usually found with a ``zipper"-argument based on a suitable Yang-Baxter equation. Unfortunately, we do not know how to do it for our model in full generality. However we can find suitable Yang-Baxter equations in two situations: In the previously known case when all parameters $\lambda_d$ attached to diagonals are equal this was done in \cite{BP16}, \cite{BW20} using Yang-Baxter equations coming from the colored six vertex-model. In the other case, when all column parameters $\mu_i$ are equal and we ignore the colors of the model, the integral expression can be constructed using the deformed Yang-Baxter equation from \cite{BK21}, which has an additional degree of freedom giving rise to the diagonal parameters. The exact form of our model and Theorem \ref{theoFirstIntro} were initially heuristically conjectured by taking superposition of these two cases.

\subsection{Main results: limit theorem for inhomogeneous Beta polymer} Consider a lattice on the plane formed by the edges $(i,j)\to(i+1,j+1)$ and $(i,j)\to(i,j+1)$ for $(i,j)\in\mathbb Z_{\geq0}^2$ such that $j\geq i$. Let $\widetilde{\sigma}_i,\widetilde{\rho}_j,\widetilde{\omega}_d$ be a family of real parameters attached to columns $i=const$, rows $j=const$ and diagonals $j-i=const$ respectively and satisfying 
\be
\widetilde{\omega}_d<\widetilde{\rho}_j<\widetilde{\sigma}_i,\qquad i\in\mathbb Z_{\geq 0},\  j\in\mathbb{Z}_{>0},\ d\in\mathbb Z_{> 0}.
\ee
To each edge we assign a weight in the following way: for each integer point $(i,j)$ we first independently sample a random variable $\eta_{i,j}\sim\Beta(\widetilde{\sigma}_i-\widetilde{\rho}_j, \widetilde{\rho}_j-\widetilde{\omega}_{j-i})$ using the Beta distribution
$$
f_{\eta_{i,j}}(x)=\1_{0< x< 1}\ x^{\widetilde{\sigma}_i-\widetilde{\rho}_j-1} (1-x)^{\widetilde{\rho}_j-\widetilde{\omega}_{j-i}-1}\frac{\Gamma(\widetilde{\sigma}_i-\widetilde{\omega}_{j-i})}{\Gamma(\widetilde{\sigma}_i-\widetilde{\rho}_j)\Gamma(\widetilde{\rho}_j-\widetilde{\omega}_{j-i})},
$$
and then we set
\be
w(e)=\begin{cases} \eta_{i,j}\quad &\text{if}\ e=(i,j-1)\to(i,j),\\ 
1-\eta_{i,j}\quad&\text{if}\ e=(i-1,j-1)\to(i,j).
\end{cases}
\ee
The \emph{delayed Beta polymer partition function} $\Z_{x,y}^{(r)}$ is defined as
\be
\Z_{x,y}^{(r)}=\sum_{\pi_0=(0,r)\to\pi_1\to\dots\to\pi_{y-r}=(x,y)}\ \prod_{i=f(\pi)}^{y-r-1}w(\pi_{i}\to\pi_{i+1}),
\ee
where the sum is over all directed lattice paths $\pi$ from $(0,r)$ to $(x,y)$, and $f(\pi)=\min\{i\mid \pi_{i+1}-\pi_i=(0,1)\}$, so the path starts accumulating its weight with a delay, waiting until the first vertical step, see Figure \ref{betaPolymerIntro}. 

The partition function $\Z_{x,y}^{(r)}$ also has an alternative description, in terms of a random walk in a random environment (RWRE). Fix $x,y$ and let $\eta_{i,j}\sim\Beta(\widetilde{\sigma}_i-\widetilde{\rho}_j, \widetilde{\rho}_j-\widetilde{\omega}_{j-i})$ be the random variables as in the polymer description above.  Then, conditionally on the random environment $\{\eta_{i,j}\}_{i,j}$, we consider a random walk $X_t$ on $\mathbb Z$, starting at $X_0=2x-y$ and doing independent $\pm1$ steps with probabilities
$$
{\sf P}(X_{t+1}=X_{t}+1\mid X_{t})=\eta_{\frac{X_{t}+y-t}{2},y-t},\qquad  {\sf P}(X_{t+1}=X_{t}-1\mid X_{t})=1-\eta_{\frac{X_{t}+y-t}{2},y-t},
$$  
where we use $\sf P$ to denote the probability space of the random walk, conditioned on the environment $\{\eta_{i,j}\}_{i,j}$. With this setup, we have $\Z_{x,y}^{(r)}={\sf P}(X_{y-r}\geq -r)$, which can be checked as in \cite{BC15b}\footnote{One way to check it is to identify the polymer paths contributing to $\Z_{x,y}^{(r)}$ with trajectories $(X_t,t)$ of the random walk reaching the half-line $X_t\geq r, X_t+t=y-2r$, using the change of coordinates $(i,j)=(\frac{X+y-t}{2}, y-t)$.}. Note that in the RWRE description the parameters $\widetilde{\sigma}_i,\widetilde{\rho}_j,\widetilde{\omega}_d$ are attached to the lines $t-X={const}$, $t=const$ and $t+X=const$ respectively.

\begin{figure}
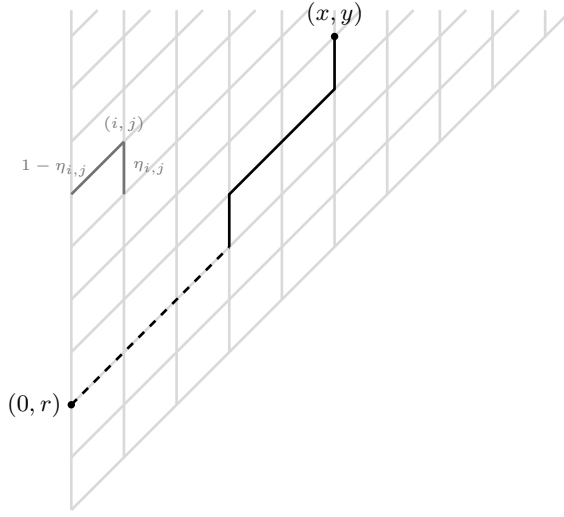

\centerline{
\tikz{0.7}{
	\foreach\y in {0,...,8}{
		\foreach\x in {0,...,\y}{
			\draw[lgray, line width=1pt] (\x,\y) -- (\x+1,\y+1);
			\draw[lgray, line width=1pt] (\x,\y) -- (\x,\y+1);
		}
	}
	\foreach\x in {0,...,9}{
		\draw[lgray, line width=1pt] (\x,9) -- (\x+0.5,9.5);
		\draw[lgray, line width=1pt] (\x,9) -- (\x,9.5);
	}
	\draw[black, dashed, line width = 1pt] (0,2) node[left] {\small $(0,r)$} -- (3,5);
	\draw[black, line width = 1pt] (3,5) -- (3,6) -- (5,8) -- (5,9) node[above] {\small $(x,y)$};
	\fill[black] (0,2) circle[radius=2pt];
	\fill[black] (5,9) circle[radius=2pt];
	
	\draw[dgray, line width=1pt] (0,6) -- (1,7);
	\draw[dgray, line width=1pt] (1,6) -- (1,7);
	\node[dgray, right] at (1, 6.5) {\tiny $\eta_{i,j}$};
	\node[dgray, left] at (0.5, 6.5) {\tiny $1-\eta_{i,j}$};
	\node[dgray, above] at (1, 7) {\tiny $(i,j)$};
}}
\caption{\label{betaPolymerIntro} Example of a path contributing to the Beta polymer partition function $\Z^{(r)}_{x,y}$. Due to the delay, the dashed edges do not contribute their weights.}
\end{figure}

Under a $q\to 1$ limit the diagonally inhomogeneous $q$-Hahn model converges to the inhomogeneous Beta polymer, and the expression from Theorem \ref{theoFirstIntro} produces an explicit integral expression for the moments:
\be
\E\left[\(\Z^{(r)}_{x,y}\)^{k}\right]=\oint_{\mathcal S_1}\cdots\oint_{\mathcal S_k}\prod_{a<b}\frac{v_b-v_a}{v_b-v_a-1}\prod_{a=1}^k\(\prod_{i=r+1}^{i\leq y}(v_a-\widetilde{\rho}_i)\prod_{i=0}^{i\leq x}\frac{1}{v_a-\widetilde{\sigma}_i}\prod_{i=r+1}^{i\leq y-x}\frac{1}{v_a-\widetilde{\omega}_i}\)\frac{dv_a}{2\pi\i},
\ee
where the contour $\mathcal S_b$ surrounds $\widetilde{\sigma}_i$ and $\mathcal S_a+1$ for $a<b$. Such integral expressions are well-suited for asymptotic analysis and the standard workflow consists of first obtaining a Fredholm determinant expression for the Laplace transform of $\Z^{(r)}_{x,y}$ (which is analytic since $\Z^{(r)}_{x,y}\in (0,1)$ almost surely) and then performing a steep descent analysis. Below we list our results.

Consider the large scale limit of $\Z^{(0)}_{\floor{xt},\floor{yt}}$ when the parameters $\widetilde{\sigma}_i,\widetilde{\rho}_j,\widetilde{\omega}_d$ have finite number of possible values $\sigma_i, \rho_j, \omega_d$ repeated with frequencies $\alpha_i,\beta_j,\gamma_d$. For example, when $\gamma_1=\gamma_2=\frac{1}{2}$ we assume that half of the diagonals have parameter $\widetilde{\omega}_d=\omega_1$ while the other half has parameter $\omega_2$; see Section \ref{limitSect} for a precise description of the model. We conjecture (Conjecture \ref{conjecture} in the text) that the following limit relation holds:
\begin{equation}
\label{introConjecture}
\lim_{t\to\infty}\P\(\frac{\ln\Z^{(0)}_{\floor{xt},\floor{yt}}+It}{c t^{1/3}}\leq r\)=F_{GUE}(r),
\end{equation}
where the slope $x/y\in (0, \frac{\sum \beta_i\rho_i-\sum\gamma_i\omega_i}{\sum\alpha_i\sigma_i-\sum\gamma_i\omega_i})$ is fixed,\footnote{The range for the slope looks more natural in the RWRE description: $\frac{\sum \beta_i\rho_i-\sum\gamma_i\omega_i}{\sum\alpha_i\sigma_i-\sum\gamma_i\omega_i}$ corresponds to the average drift of the walk, and we are looking at fluctuations at non-typical velocities. See \cite{BC15b} for more details about the RWRE interpretation of the model.} $F_{GUE}(s)$ denotes the GUE Tracy-Widom distribution \cite{TW92} and the constants $I,c$ are defined in an implicit way involving an auxiliary parameter $\theta\in(\max_i\sigma_i, \infty)$ and polygamma functions $\Psi_k$:
\be
x/y=\frac{\sum_i\beta_i\Psi_1(\theta-\rho_i)-\sum_i\gamma_i\Psi_1(\theta-\omega_i)}{\sum_i\alpha_i\Psi_1(\theta-\sigma_i)-\sum_i\gamma_i\Psi_1(\theta-\omega_i)},
\ee
\be
I=x \sum_i\alpha_i\Psi_0(\theta-\sigma_i)-y \sum_i\beta_i\Psi_0(\theta-\rho_i)+ (y-x)\sum_i\gamma_i\Psi_0(\theta-\omega_i),
\ee
\be
c^3=-\frac{x}{2}\sum_i\alpha_i\Psi_2(\theta-\sigma_i)+\frac{y}{2}\sum_i\beta_i\Psi_2(\theta-\rho_i)-\frac{y-x}{2}\sum_i\gamma_i\Psi_2(\theta-\omega_i).
\ee

Under assumptions dictated by purely technical reasons we are also able to partially verify our conjecture:
\begin{theo}[Theorem \ref{mainBetaResult} in the text] When $\theta\in (0,\frac{1}{2})$ and
 \be
 \sigma_1=\sigma_2=\dots=0,\qquad \rho_1=\rho_2=\dots=-1, \qquad \omega_1,\omega_2, \dots< -1 
 \ee
 we have 
\be
\lim_{t\to\infty}\P\(\frac{\ln\Z^{(0)}_{\floor{xt},\floor{yt}}+It}{c t^{1/3}}\leq r\)=F_{GUE}(r),
\ee
with the notation specified above.
\end{theo}
In other words, we are able to partially verify the conjectural limit statement for the model having only diagonal parameters $\omega_d<-1$, with Beta distributions $\Beta(1, -1-\omega_d)$. Note that even in the homogeneous case our result is more general than the analogous one from \cite{BC15b}, where the full proof is given only for the model with $\Beta(1,1)$ distributions. 

Shortly after this paper has been written, a more general limit result for the homogeneous Beta polymer has appeared in \cite{OPR21}, and it covers the model with $\Beta(\alpha,\beta)$ distributions for arbitrary $\alpha,\beta\in\mathbb R_{>0}$. It would be interesting to see in what generality the methods and results from \cite{OPR21} can be transferred to the inhomogeneous model.

\subsection{Layout of the paper} Sections \ref{vmSection}-\ref{shiftSect} are devoted to the integrability of the $q$-Hahn vertex model with diagonal parameters: In Section \ref{vmSection} we describe the needed background on vertex models and Yang-Baxter equations,  which is immediately used in Section \ref{localSect} to establish local recurrence relations for the $q$-Hahn model. In Section \ref{qHahnSec} we define the $q$-Hahn vertex model and prove one of the central results of this work, Theorem \ref{qHahnResultTheo}, using relations from the previous section. Section \ref{shiftSect} is a brief overview of the shift-invariance property for the vertex models, which can be established using Theorem \ref{qHahnResultTheo}. Then we pass to the Beta polymer side of the work, starting with Section \ref{reductionSect} where we describe the transition from the $q$-Hahn vertex model to the Beta polymer model, and continuing with Section \ref{fredSect} where we use integral expressions coming from Theorem \ref{qHahnResultTheo} to write a Fredholm determinant expression for the Laplace transform of the polymer partition function. Finally, Section \ref{limitSect} is devoted to the large-scale analysis of the inhomogeneous Beta polymer model, with the results stated in Conjecture \ref{conjecture} and Theorem \ref{mainBetaResult}.

\subsection{Notation} Throughout the work we follow the following standard notation. For $k\in\mathbb Z_{>0}$ the symmetric group of rank $k$ consists of permutations of $k$ elements and is denoted by $S_k$. For a permutation $\pi\in S_k$ its length is denoted by $l(\pi)$ and is defined as the number of pairs $(i,j)$ such that $i<j$ and $\pi(i)>\pi(j)$. A partition $\lambda$ is a finite integer sequence $\lambda_1\geq\lambda_2\geq\dots\geq\lambda_l>0$, with $l=l(\lambda)$ being called the length of $\lambda$. If $\lambda_1+\dots+\lambda_{l(\lambda)}=n$ we write $\lambda\vdash n$. The part multiplicities $m_k(\lambda)$ are defined by $m_k(\lambda)=\#\{i\mid \lambda_i=k\}$.

For any $n\in\mathbb Z$ the $q$-Pochhammer symbol is defined as 
\be
(x;q)_n:=\begin{cases}\prod_{i=1}^{n}(1-xq^{i-1}),\quad &n\geq 0;\\ \prod_{i=0}^{-n-1}(1-xq^{-i-1})^{-1},\quad &n\leq 0.\end{cases}
\ee
In particular, $(x;q)_0:=1$. We also set $(x;q)_\infty:=\prod_{i=1}^{\infty}(1-xq^{i-1})$. For $n,m\geq 0$ the $q$-binomial coefficient $\binom{n}{m}_q$ is given by
\be
\binom{n}{m}_q=\frac{(q;q)_{n}}{(q;q)_m(q;q)_{n-m}}=\frac{(q^{n-m+1};q)_m}{(q;q)_m}.
\ee 
With the assumption $|q|<1$, all expressions above make sense both numerically and formally in the space of power series in $q$. 

For a pair of parameters $(\alpha,\beta)\in\mathbb{R}_{\geq 0}^2$ the Beta distribution is defined by the following density function with respect to the Lebesgue measure: 
\be
f_{\Beta(\alpha,\beta)}(x)=\1_{0< x< 1}\ x^{\alpha-1} (1-x)^{\beta-1}\frac{\Gamma(\alpha+\beta)}{\Gamma(\alpha)\Gamma(\beta)},
\ee
where $\Gamma(z)$ is the Gamma function. We use $\Beta(\alpha,\beta)$ to denote a Beta-distributed random variable with the corresponding parameters. 

For a trace-class operator $K$ over $L^2(X)$ we use $\det(I+K)_{L^{2}(X)}$ to denote its Fredholm determinant. The only operators $K$ considered in this work are integral operators of the form
\be
[Kf](x)=\int_{X}f(y)K(x,y)dy,
\ee
where $K(x,y)$ is called the kernel of $K$. In this case the Fredholm determinant is given explicitly by
\be
\det(I+K)_{L^2(X)}=\sum_{l\geq 0}\frac{1}{l!}\int_{X}dx_1\dots\int_Xdx_l\det\left[K(x_a,x_b)\right]_{a,b=1}^l.
\ee
We often use Fredholm determinants over complex integration contours, so for a contour $\C\subset\mathbb C$ we let $L^2(\C)$ denote the space of functions on $\C$ with the complex valued integration measure $\frac{dz}{2\pi\i}$.

The Tracy-Widom GUE cumulative distribution function $F_{GUE}(r)$ is defined by
\be
F_{GUE}(r)=\det(I-K_{\mathrm{Ai}})_{L^2(r,\infty)},
\ee
\be
K_{\mathrm{Ai}}(\lambda,\mu)=\frac{\Ai(\lambda)\Ai'(\mu)-\Ai(\mu)\Ai'(\lambda)}{\lambda-\mu}=\int_{e^{-2\pi\i/3}\infty}^{e^{2\pi\i/3}\infty}\frac{d v}{2\pi \i}\int_{e^{-\pi\i/3}\infty}^{e^{\pi\i/3}\infty}\frac{dz}{2\pi \i}\ \frac{1}{z-v}\frac{e^{z^3/3-z\lambda}}{e^{v^3/3-v\mu}},
\ee
where $\Ai(z)$ is the Airy function and the integration contours in the last expression do not intersect and go to infinity along the prescribed directions.  

\subsection{Acknowledgements}
The author is grateful to Alexei Borodin for numerous helpful comments about the work. The author would also like to thank Guillaume Barraquand and Ivan Corwin for their clarifications regarding the work \cite{BC15b}.

\section{Vertex models and Yang-Baxter equation}\label{vmSection} The focus of this section is the Yang-Baxter equation, which can be seen as the fundamental reason for integrability of the models presented in this work. More precisely, here we introduce our notation for vertex models, describe vertex weights and provide known forms and consequences of the Yang-Baxter equation.

\subsection{Vertex configurations and weights.} Vertex models are described in terms of \emph{vertices}, which we graphically represent as intersections of oriented lines. Segments of the lines separated by the vertices are called \emph{edges}, so each vertex has a pair of incoming edges and a pair of outgoing edges. A \emph{configuration} of a vertex is an assignment of four labels to the adjacent edges. Finally, to each vertex configuration we assign a \emph{vertex weight}, depending on these four labels.
 
Configurations of our models are labeled by \emph{compositions} assigned to the edges, that is, by $n$-tuples of nonnegative integers $\bA=(A_1,A_2, \dots, A_n)\in\mathbb Z_{\geq 0}^n$. We often present configurations as ensembles of oriented colored lattice paths directed along the edges, with colors labeled by integers $1, 2, \dots, n$\footnote{The color $0$ is commonly reserved for the absence of a path.}. The two descriptions are identified by encoding the colors of paths occupying a single edge by a composition $\bA=(A_1, \dots, A_n)$, where $A_i$ is the number of paths with color $i$. 

Throughout the text we use the following notation regarding compositions. For a composition $\bA=(A_1, \dots, A_n)$ we set $|\bA|:=A_1+A_2+\dots+A_n$. More generally, for integers $1\leq i,j\leq n$ let
\begin{equation}
\label{Aijdef}
\bA_{[i,j]}:=\begin{cases}A_i+A_{i+1}+\dots+A_j,\quad \text{if}\ i\leq j,\\ 0\quad\text{otherwise}.
\end{cases}
\end{equation}
Finally, given a composition $\bA$ we can define an \emph{ordered} $|\bA|$-tuple of integers $1^{A_1}2^{A_2}\dots n^{A_n}$ which consists of $1$ repeated $A_1$ times, followed by $A_2$ copies of $2$, $A_3$ copies of $3$ and so on.

The main family of vertex weights used in this work is the family of \emph{$q$-Hahn}\footnote{In the case $n=1$ the vertex weights \eqref{Wdef} reproduce the orthogonality weights for the $q$-Hahn polynomials, hence the name.} vertex weights, denoted in one of the following ways
\begin{equation}
\label{Wnot}
\tikz{1}{
	\draw[fused] (-1,0) -- (1,0);
	\draw[fused] (0,-1) -- (0,1);
	\node[left] at (-1,0) {\tiny $\bB$};\node[right] at (1,0) {\tiny $\bD$};
	\node[below] at (0,-1) {\tiny $\bA$};\node[above] at (0,1) {\tiny $\bC$};
	\node[above right] at (0,0.1) {\tiny $W_{t,s}$};
}=
\tikz{1}{
	\draw[fused] (-1,0) -- (1,0);
	\draw[fused] (0,-1) -- (0,1);
	\node[left] at (-1,0) {\tiny $\bB$};\node[right] at (1,0) {\tiny $\bD$};
	\node[below] at (0,-1) {\tiny $\bA$};\node[above] at (0,1) {\tiny $\bC$};
	\node[above right] at (0,0.1) {\tiny ${(t,s)}$};
}=W_{t,s}(\bA,\bB;\bC,\bD)
\end{equation}
and given explicitly by
\begin{equation}
\label{Wdef}
W_{t,s}(\bA,\bB;\bC,\bD)=\1_{\bA+\bB=\bC+\bD}\ (s^2/t^2)^{|\bD|}\frac{(s^2/t^2;q)_{|\bA|-|\bD|}(t^2;q)_{|\bD|}}{(s^2;q)_{|\bA|}}q^{\sum_{i<j}D_i(A_j-D_j)}\prod_{i=1}^n\binom{A_i}{D_i}_q.
\end{equation}
Here $\bA=(A_1, \dots, A_n), \bB=(B_1, \dots, B_n), \bC=(C_1, \dots, C_n), \bD=(D_1, \dots, D_n)$ are compositions, $q$ is a globally fixed \emph{quantization} parameter and $s,t$ are \emph{spin} parameters. Note that the weights $W_{t,s}(\bA,\bB;\bC,\bD)$ vanish unless the following \emph{conservation law} holds:
\begin{equation}
\label{conservationLaw}
\bA+\bB=\bC+\bD.
\end{equation}
Moreover, the weights $W_{t,s}(\bA,\bB;\bC,\bD)$ vanish unless $\bA\geq\bD$, that is, $A_i\geq D_i$ for all $i$. Moreover, the weights $W_{t,s}(\bA,\bB;\bC,\bD)$ do not depend on $\bB$ and $\bC$ as long as the conservation law \eqref{conservationLaw} holds.

Another family of vertex weights used in this work consists of \emph{higher spin six-vertex} weights and is denoted by 
\begin{equation}
\label{hsnot}
\tikz{1}{
	\draw[unfused] (-1,0) -- (1,0);
	\draw[fused] (0,-1) -- (0,1);
	\node[left] at (-1,0) {\tiny $j$};\node[right] at (1,0) {\tiny $l$};
	\node[below] at (0,-1) {\tiny $\bI$};\node[above] at (0,1) {\tiny $\bK$};
	\node[above right] at (0,0) {\tiny $w_{z;s}$};
}
=\tikz{1}{
	\draw[unfused] (-1,0) -- (1,0);
	\draw[fused] (0,-1) -- (0,1);
	\node[left] at (-1,0) {\tiny $j$};\node[right] at (1,0) {\tiny $l$};
	\node[below] at (0,-1) {\tiny $\bI$};\node[above] at (0,1) {\tiny $\bK$};
	\node[above right] at (0,0) {\tiny ${(z;s)}$};
}
=w_{z;s}(\bI,j;\bK,l),
\end{equation}
where $j,l$ are integers from $[0, n]$ and $\bI,\bK$ are compositions. To keep consistency with the edge labelling by compositions, we can interpret integers $j,l$ as compositions $\bJ,\bL$ satisfying $|\bJ|,|\bL|\leq 1$: if $j=0$ set $\bJ=\e^0=\mathbf 0=(0, \dots, 0)$, while for $j\neq 0$ set $\bJ$ equal to the standard basis vector $\e^j\in\mathbb Z^n$:
\be
\e^j_i=\1_{i=j}.
\ee 
The values of the weights $w_{z;s}$ are summarized in the table below:
\begin{align}
\label{hsdef}
\begin{tabular}{|c|c|c|}
\hline
\quad
\tikz{0.7}{
	\draw[lgray,line width=1.5pt,->] (-1,0) -- (1,0);
	\draw[lgray,line width=4pt,->] (0,-1) -- (0,1);
	\node[left] at (-1,0) {\tiny $0$};\node[right] at (1,0) {\tiny $0$};
	\node[below] at (0,-1) {\tiny $\bI$};\node[above] at (0,1) {\tiny $\bI$};
}
\quad
&
\quad
\tikz{0.7}{
	\draw[lgray,line width=1.5pt,->] (-1,0) -- (1,0);
	\draw[lgray,line width=4pt,->] (0,-1) -- (0,1);
	\node[left] at (-1,0) {\tiny $i$};\node[right] at (1,0) {\tiny $i$};
	\node[below] at (0,-1) {\tiny $\bI$};\node[above] at (0,1) {\tiny $\bI$};
}
\quad
&
\quad
\tikz{0.7}{
	\draw[lgray,line width=1.5pt,->] (-1,0) -- (1,0);
	\draw[lgray,line width=4pt,->] (0,-1) -- (0,1);
	\node[left] at (-1,0) {\tiny $0$};\node[right] at (1,0) {\tiny $i$};
	\node[below] at (0,-1) {\tiny $\bI$};\node[above] at (0,1) {\tiny $\bI-\e^i$};
}
\quad
\\[1.3cm]
\quad
$\dfrac{1-s z q^{I_{[1,n]}}}{1-sz}$
\quad
& 
\quad
$\dfrac{(s^2q^{I_i}-sz) q^{I_{[i+1,n]}}}{1-sz}$
\quad
& 
\quad
$\dfrac{sz(q^{I_i}-1) q^{I_{[i+1,n]}}}{1-sz}$
\quad
\\[0.7cm]
\hline
\quad
\tikz{0.7}{
	\draw[lgray,line width=1.5pt,->] (-1,0) -- (1,0);
	\draw[lgray,line width=4pt,->] (0,-1) -- (0,1);
	\node[left] at (-1,0) {\tiny $i$};\node[right] at (1,0) {\tiny $0$};
	\node[below] at (0,-1) {\tiny $\bI$};\node[above] at (0,1) {\tiny $\bI+\e^i$};
}
\quad
&
\quad
\tikz{0.7}{
	\draw[lgray,line width=1.5pt,->] (-1,0) -- (1,0);
	\draw[lgray,line width=4pt,->] (0,-1) -- (0,1);
	\node[left] at (-1,0) {\tiny $i$};\node[right] at (1,0) {\tiny $j$};
	\node[below] at (0,-1) {\tiny $\bI$};\node[above] at (0,1) 
	{\tiny $\bI+\e^i-\e^j$};
}
\quad
&
\quad
\tikz{0.7}{
	\draw[lgray,line width=1.5pt,->] (-1,0) -- (1,0);
	\draw[lgray,line width=4pt,->] (0,-1) -- (0,1);
	\node[left] at (-1,0) {\tiny $j$};\node[right] at (1,0) {\tiny $i$};
	\node[below] at (0,-1) {\tiny $\bI$};\node[above] at (0,1) {\tiny $\bI+\e^j-\e^i$};
}
\quad
\\[1.3cm] 
\quad
$\dfrac{1-s^2 q^{I_{[1,n]}}}{1-sz}$
\quad
& 
\quad
$\dfrac{sz(q^{I_j}-1) q^{I_{[j+1,n]}}}{1-sz}$
\quad
&
\quad
$\dfrac{s^2(q^{I_i}-1)q^{I_{[i+1,n]}}}{1-sz}$
\quad
\\[0.7cm]
\hline
\end{tabular} 
\end{align}
where $i,j$ are integers from $[1,n]$ satisfying $i<j$,  $\bI=(I_1, \dots, I_n)$ is a composition and all unlisted weights are assumed to be $0$ since they violate the conservation law.  The parameter $q$ is the same quantization parameter as in the $q$-Hahn weights, while the parameters $z$ and $s$ are \emph{spectral} and \emph{spin} parameters respectively.

Both vertex weights $W_{t,s}$ and $w_{z;s}$ are particular cases of more general weights $\mathcal W_{z}^{\NN,\MM}(\bA,\bB;\bC,\bD)$\footnote{To avoid possible confusion with the weights $W_{t,s}$, we refrain from using any graphical notation for the weights $\W^{\NN,\MM}_z$.}, which are defined for compositions satisfying
\begin{equation}
\label{NMrestr}
|\bA|,|\bC|\leq \MM,\qquad |\bB|,|\bD|\leq \NN
\end{equation}
and are given explicitly by
\begin{multline}
\label{fullyfusedexpression}
\W_{z}^{\NN,\MM}\(\bA,\bB;\bC,\bD\)=\1_{\bA+\bB=\bC+\bD} z^{|\bD|-|\bB|}q^{|\bA|\NN-|\bD|\MM}\\
\times \sum_{\bP}\Phi(\bC-\bP, \bC+\bD-\bP;q^{\NN-\MM}z;q^{-\MM}z)\Phi(\bP,\bB;q^{-\NN}/z;q^{-\NN}),
\end{multline} 
where the sum is over compositions $\bP=(P_1, \dots, P_n)$ such that $P_i\leq \min(B_i, C_i)$ and we set
\be
\Phi(\bA,\bB;x,y):=(y/x)^{|\bA|}\frac{(x;q)_{|\bA|}(y/x;q)_{|\bB-\bA|}}{(y;q)_{|\bB|}}q^{\sum_{i<j}(B_i-A_i)A_j}\prod_{i=1}^n\binom{B_i}{A_i}_q.
\ee
The weights $W_{t,s}$ and $w_{u;s}$ are recovered using the following specializations:
\begin{equation}
\label{spec}
W_{t,s}(\bA,\bB;\bC, \bD)=\restr{\W_{1}^{\NN,\MM}(\bA,\bB;\bC,\bD)}{\substack{q^{-\NN}=t^2\\q^{-\MM}=s^2}},\qquad w_{z;s}(\bI,j;\bK,l)=\restr{\W_{z/s}^{1,\MM}(\bI,\e^j;\bK, \e^l)}{q^{-\MM}=s^2},
\end{equation}
where we use the rational dependence of the weights $\W_{z}^{\NN,\MM}$ on $q^{-\NN}$ and $q^{-\MM}$ to perform analytic continuations, replacing $q^{-\NN}$ and $q^{-\MM}$ by $t^2$ and $s^2$ respectively. 

Though it is not needed for our purposes, let us briefly elaborate on the origins of these weights. The weights $\mathcal W_{z}^{\NN,\MM}(\bA,\bB;\bC,\bD)$ originate from the quantum group $U_q(\widehat{\mathfrak{sl}}_{n+1})$ and can be seen as the matrix coefficients of a renormalized $\mathcal R$-matrix acting on symmetric tensor representations, see \cite{KMMO16}. The explicit expression \eqref{fullyfusedexpression} was obtained in \cite{BM16} using methods of three-dimensional solvability. See \cite[Appendix C]{BW18} for a condensed summary of the properties of $\mathcal W_{z}^{\NN,\MM}(\bA,\bB;\bC,\bD)$ relevant for the current work. The colored higher spin six-vertex weights are coefficients of the $\mathcal R$-matrix above when one of the representations is the standard representation, and these weights are related to a spin deformation of the non-symmetric Hall-Littlewood polynomials, see \cite{BW18} and Remark \ref{spinHLRemark}. The $q$-Hahn weights are not as well-understood as the other weights; we limit ourselves to stating that in one-color case they are related to spin deformations of the $q$-Whittaker symmetric functions, see \cite{BW17}, \cite{MP20}, \cite{BK21}.

\subsection{Yang-Baxter equations.} The vertex weights described above are distinguished by an algebraic relation called the \emph{Yang-Baxter equation}. In the most general situation relevant to this work the Yang-Baxter equation is given in \cite[Appendix C]{BW18}:
\begin{multline}
\label{masterYB}
\sum_{\bK_1,\bK_2,\bK_3}\W^{\NN,\MM}_{x/y}(\bA_2,\bA_1;\bK_2,\bK_1)\W_{x/z}^{\NN,\LL}(\bA_3,\bK_1;\bK_3,\bB_1)\W_{y/z}^{\MM,\LL}(\bK_3,\bK_2;\bB_3,\bB_2)\\
=\sum_{\bK_1,\bK_2,\bK_3}\W_{y/z}^{\MM,\LL}(\bA_3,\bA_2;\bK_3,\bK_2)\W_{x/z}^{\NN,\LL}(\bK_3,\bA_1;\bB_3,\bK_1)\W_{x/y}^{\NN,\MM}(\bK_2,\bK_1;\bB_2,\bB_1),
\end{multline}
where both sums are taken over arbitrary triples of compositions $\bK_1, \bK_2, \bK_3$. Note that both sums are actually finite due to the non-negativity of compositions and the conservation law. 

Setting $x=y=z$ and using specialization \eqref{spec} we obtain the Yang-Baxter equation for the weights $W_{t,s}$:
\begin{multline}
\label{WYBeq}
\sum_{\bK_1,\bK_2,\bK_3}W_{t_1,t_2}(\bA_2,\bA_1;\bK_2,\bK_1)W_{t_1,t_3}(\bA_3,\bK_1;\bK_3,\bB_1)W_{t_2,t_3}(\bK_3,\bK_2;\bB_3,\bB_2)\\
=\sum_{\bK_1,\bK_2,\bK_3}W_{t_2,t_3}(\bA_3,\bA_2;\bK_3,\bK_2)W_{t_1,t_3}(\bK_3,\bA_1;\bB_3,\bK_1)W_{t_1,t_2}(\bK_2,\bK_1;\bB_2,\bB_1).
\end{multline}
Throughout the text we usually write such equations using graphical notation. For instance, \eqref{WYBeq} is represented by the following diagrammatic equation:
\begin{equation}
\label{WYB}
\tikzbase{1.2}{-3}{
	\draw[fused]
	(-2,0.5) node[above,scale=0.6] {\color{black} $\bA_1$} -- (-1,-0.5) -- (1,-0.5) node[right,scale=0.6] {\color{black} $\bB_1$};
	\draw[fused] 
	(-2,-0.5) node[below,scale=0.6] {\color{black} $\bA_2$} -- (-1,0.5)  -- (1,0.5) node[right,scale=0.6] {\color{black} $\bB_2$};
	\draw[fused] 
	(0,-1.5) node[below,scale=0.6] {\color{black} $\bA_3$} -- (0,0) -- (0,1.5) node[above,scale=0.6] {\color{black} $\bB_3$};
	\node[above right] at (0,-0.4) {\tiny{ $W_{t_1,t_3}$}};
	\node[above right] at (0,0.6) {\tiny{ $W_{t_2,t_3}$}};
	\node[right] at (-1.5,0) {\ \tiny{$W_{t_1,t_2}$}};
}\qquad
=\qquad
\tikzbase{1.2}{-3}{
	\draw[fused] 
	(-1,1) node[left,scale=0.6] {\color{black} $\bA_1$} -- (1,1) -- (2,0) node[below,scale=0.6] {\color{black} $\bB_1$};
	\draw[fused] 
	(-1,0) node[left,scale=0.6] {\color{black} $\bA_2$} -- (1,0) -- (2,1) node[above,scale=0.6] {\color{black} $\bB_2$};
	\draw[fused] 
	(0,-1) node[below,scale=0.6] {\color{black} $\bA_3$} -- (0,0.5) -- (0,2) node[above,scale=0.6] {\color{black} $\bB_3$};
	\node[above right] at (0,0) {\tiny{ $W_{t_2,t_3}$}};
	\node[above right] at (0,1) {\tiny{ $W_{t_1,t_3}$}};
	\node[right] at (1.5,0.5) {\ \tiny{$W_{t_1,t_2}$}};
}
\end{equation}
More precisely, the diagrams like the above one denote \emph{partition functions} of the corresponding models.  That is, given such a diagram with a fixed labelling of the boundary edges and an assignment of vertex weights to the vertices, the corresponding partition function is defined by taking the sum of products of vertex weights over all possible configurations of the internal edges. 

Another instance of the Yang-Baxter equation is obtained by setting $\NN=1, y=z=1$ in \eqref{masterYB} and applying \eqref{spec}. The result is graphically represented by
\begin{equation}
\label{hsYB}
\tikzbase{1.2}{-3}{
	\draw[unfused]
	(-2,0.5) node[above,scale=0.6] {\color{black} $\bA_1$} -- (-1,-0.5) -- (1,-0.5) node[right,scale=0.6] {\color{black} $\bB_1$};
	\draw[fused] 
	(-2,-0.5) node[below,scale=0.6] {\color{black} $\bA_2$} -- (-1,0.5)  -- (1,0.5) node[right,scale=0.6] {\color{black} $\bB_2$};
	\draw[fused] 
	(0,-1.5) node[below,scale=0.6] {\color{black} $\bA_3$} -- (0,0) -- (0,1.5) node[above,scale=0.6] {\color{black} $\bB_3$};
	\node[above right] at (0,-0.5) {\tiny{ $w_{xs;s}$}};
	\node[above right] at (0,0.6) {\tiny{ $W_{t,s}$}};
	\node[right] at (-1.5,0) {\ \tiny{$w_{xt;t}$}};
}\qquad
=\qquad
\tikzbase{1.2}{-3}{
	\draw[unfused] 
	(-1,1) node[left,scale=0.6] {\color{black} $\bA_1$} -- (1,1) -- (2,0) node[below,scale=0.6] {\color{black} $\bB_1$};
	\draw[fused] 
	(-1,0) node[left,scale=0.6] {\color{black} $\bA_2$} -- (1,0) -- (2,1) node[above,scale=0.6] {\color{black} $\bB_2$};
	\draw[fused] 
	(0,-1) node[below,scale=0.6] {\color{black} $\bA_3$} -- (0,0.5) -- (0,2) node[above,scale=0.6] {\color{black} $\bB_3$};
	\node[above right] at (0,0) {\tiny{ $W_{t,s}$}};
	\node[above right] at (0,1) {\tiny{ $w_{xs;s}$}};
	\node[right] at (1.5,0.5) {\ \tiny{$w_{xt;t}$}};
}
\end{equation}
Note that we again use diagrams to write the equation, but this time the lines have different thickness. Throughout the text we follow the following convention: thick edges can be labelled by any compositions, while thin edges are labeled by compositions $\bI$ with $|\bI|\leq 1$. 

In \cite{BK21} it was shown that the Yang Baxter equations involving the weights $W_{t,s}$ can be sometimes deformed by adding an additional parameter to the equation in a nontrivial way. More precisely, the following \emph{deformed Yang-Baxter} equations hold:
\begin{equation}
\label{defWYB}
\tikzbase{1.2}{-3}{
	\draw[fused]
	(-2,0.5) node[above,scale=0.6] {\color{black} $\bA_1$} -- (-1,-0.5) -- (1,-0.5) node[right,scale=0.6] {\color{black} $\bB_1$};
	\draw[fused] 
	(-2,-0.5) node[below,scale=0.6] {\color{black} $\bA_2$} -- (-1,0.5)  -- (1,0.5) node[right,scale=0.6] {\color{black} $\bB_2$};
	\draw[fused] 
	(0,-1.5) node[below,scale=0.6] {\color{black} $\bA_3$} -- (0,0) -- (0,1.5) node[above,scale=0.6] {\color{black} $\bB_3$};
	\node[above right] at (0,-0.4) {\tiny{ $W_{\eta t_1,\eta t_3}$}};
	\node[above right] at (0,0.6) {\tiny{ $W_{t_2,t_3}$}};
	\node[right] at (-1.5,0) {\ \tiny{$W_{t_1, t_2}$}};
}\qquad
=\qquad
\tikzbase{1.2}{-3}{
	\draw[fused] 
	(-1,1) node[left,scale=0.6] {\color{black} $\bA_1$} -- (1,1) -- (2,0) node[below,scale=0.6] {\color{black} $\bB_1$};
	\draw[fused] 
	(-1,0) node[left,scale=0.6] {\color{black} $\bA_2$} -- (1,0) -- (2,1) node[above,scale=0.6] {\color{black} $\bB_2$};
	\draw[fused] 
	(0,-1) node[below,scale=0.6] {\color{black} $\bA_3$} -- (0,0.5) -- (0,2) node[above,scale=0.6] {\color{black} $\bB_3$};
	\node[above right] at (0,0) {\tiny{ $W_{\eta t_2,\eta t_3}$}};
	\node[above right] at (0,1) {\tiny{ $W_{t_1,t_3}$}};
	\node[right] at (1.5,0.5) {\ \tiny{$W_{\eta t_1,\eta t_2}$}};
}
\end{equation}
\begin{equation}
\label{defhsYB}
\tikzbase{1.2}{-3}{
	\draw[unfused]
	(-2,0.5) node[above,scale=0.6] {\color{black} $\bA_1$} -- (-1,-0.5) -- (1,-0.5) node[right,scale=0.6] {\color{black} $\bB_1$};
	\draw[fused] 
	(-2,-0.5) node[below,scale=0.6] {\color{black} $\bA_2$} -- (-1,0.5)  -- (1,0.5) node[right,scale=0.6] {\color{black} $\bB_2$};
	\draw[fused] 
	(0,-1.5) node[below,scale=0.6] {\color{black} $\bA_3$} -- (0,0) -- (0,1.5) node[above,scale=0.6] {\color{black} $\bB_3$};
	\node[above right] at (0,-0.5) {\tiny{ $w_{xs;s}$}};
	\node[above right] at (0,0.6) {\tiny{ $W_{t,s}$}};
	\node[right] at (-1.5,0) {\ \tiny{$w_{xt/\eta;t\eta}$}};
}\qquad
=\qquad
\tikzbase{1.2}{-3}{
	\draw[unfused] 
	(-1,1) node[left,scale=0.6] {\color{black} $\bA_1$} -- (1,1) -- (2,0) node[below,scale=0.6] {\color{black} $\bB_1$};
	\draw[fused] 
	(-1,0) node[left,scale=0.6] {\color{black} $\bA_2$} -- (1,0) -- (2,1) node[above,scale=0.6] {\color{black} $\bB_2$};
	\draw[fused] 
	(0,-1) node[below,scale=0.6] {\color{black} $\bA_3$} -- (0,0.5) -- (0,2) node[above,scale=0.6] {\color{black} $\bB_3$};
	\node[above right] at (0,0) {\tiny{ $W_{t,s}$}};
	\node[above right] at (0,1) {\tiny{ $w_{xs/\eta;s\eta}$}};
	\node[right] at (1.5,0.5) {\ \tiny{$w_{xt;t}$}};
}
\end{equation}
where $\eta$ in both cases is the added deformation parameter. 

Since the context of \cite{BK21} required only one-colored models, the deformed Yang-Baxter equations were only proved in that case, while for the colored situation the proofs can be repeated verbatim. However, for the completeness sake, we provide the proofs for the colored setting in Appendix \ref{app:deformed}.

\subsection{Stochasticity} Apart from the Yang-Baxter equation the vertex weights $W_{t,s}$ and $w_{z;s}$ satisfy another property, namely, they are \emph{stochastic}: the sum of vertex weights with the same valid incoming configuration is equal to $1$. In other words, the following identities hold
\begin{equation}
\label{stochW}
\sum_{\bC,\bD}
\tikzbase{1}{-0.56ex}{
	\draw[fused] (-1,0) -- (1,0);
	\draw[fused] (0,-1) -- (0,1);
	\node[left] at (-1,0) {\tiny $\bB$};\node[right] at (1,0) {\tiny $\bD$};
	\node[below] at (0,-1) {\tiny $\bA$};\node[above] at (0,1) {\tiny $\bC$};
	\node[above right] at (0,0.1) {\tiny{ $W_{t,s}$}};
}
=\sum_{\bC,\bD}W_{t,s}(\bA,\bB;\bC,\bD)=1,
\end{equation}
\begin{equation}
\label{stochhs}
\sum_{\bK,l}
\tikzbase{1}{-0.56ex}{
	\draw[unfused] (-1,0) -- (1,0);
	\draw[fused] (0,-1) -- (0,1);
	\node[left] at (-1,0) {\tiny $j$};\node[right] at (1,0) {\tiny $l$};
	\node[below] at (0,-1) {\tiny $\bI$};\node[above] at (0,1) {\tiny $\bK$};
	\node[above right] at (0,0) {\tiny{ $w_{z;s}$}};
}
=\sum_{\bK,l}w_{z;s}(\bI,j;\bK,l)=1,
\end{equation}
where for the first identity $\bA,\bB$ are arbitrary compositions, and for the second one $\bI$ is an arbitrary composition and $j\in [0, n]$.

The importance of the stochasticity comes from the probabilistic interpretation of the vertex models, where the stochastic weights can be used to define a stochastic sampling rule for an outgoing configuration of a vertex given an incoming one. We postpone the detailed discussion of this construction until Section \ref{qHahnSec}.

There are several ways to establish relations \eqref{stochW}, \eqref{stochhs}. One option is to verify everything by a direct computation, which can be readily performed for the weights $w_{z;s}$, but gets more involved for the weights $W_{t,s}$. Another approach is to use the general weights $\W_{z}^{\NN,\MM}$, which turn out to be stochastic as well:
\be
\sum_{\bC,\bD}\W_z^{\NN,\MM}(\bA,\bB;\bC,\bD)=1,
\ee
where $\bA,\bB$ are arbitrary compositions satisfying the constraints \eqref{NMrestr}. The latter fact can be derived, for example, from the fusion construction of the weights $\W_{z}^{\NN,\MM}$, see \cite[Appendix]{BGW19}.

Finally, for the weights $W_{t,s}$ one can use the Yang-Baxter equation to prove \eqref{stochW}. This approach is described in Section \ref{localSect}, where a more general \emph{local relation} is considered.

\subsection{Exchange relations and Hecke algebras.} Given a vertex model, it is usual to consider row-to-row transfer matrices. For a fixed $L$ let $V_L$ denote an infinite dimensional vector space with basis enumerated by $L$-tuples of compositions. We denote the elements of this basis by $|\bI_1, \dots, \bI_L\rangle$ and we let $\langle\bI_1, \dots, \bI_L|\in V_L^*$ denote the elements of the basis dual to $\{|\bI_1, \dots, \bI_L\rangle\}_{\bI_1, \dots, \bI_L}$.

For a fixed pair of families $\Xi=(\xi_1, \dots, \xi_L)\in \mathbb R^L,\S=(s_1, \dots, s_L)\in\mathbb R^L$, define operators $\C_i(u\mid \Xi,\S)$ on $V_L$ by
\begin{equation}
\label{rowdef}
\langle \bJ_1, \dots, \bJ_L|\ \C_i(u\mid \Xi,\S)\ |\bI_1, \dots, \bI_L\rangle:= 
\tikzbase{0.8}{-1}{
	\draw[unfused*] (1,2) -- (5.5,2);
	\draw[unfused] (6.5,2) -- (9,2);
	\foreach\x in {2,3}{
		\draw[fused] (\x*2-2,1) -- (\x*2-2,3);
	}
	\draw[fused] (7.5,1) -- (7.5,3);
	\node at (6,2) {$\cdots$};
	\node[above] at (7.5,3) {$\bJ_{L}$};
	\node[above] at (6,3) {$\cdots$};
	\node[above] at (4,3) {$\bJ_2$};
	\node[above] at (2,3) {$\bJ_1$};
	\node[left] at (1,2) {$i$};
	 \node[below] at (7.5,1) {$\bI_{L}$};
	\node[below] at (6,1) {$\cdots$};
	\node[below] at (4,1) {$\bI_2$};
	\node[below] at (2,1) {$\bI_1$};
	\node[right] at (9,2) {$0$};
	\node[above right] at (2,2) {\tiny $(u\xi_1; s_1)$};
	\node[above right] at (4,2) {\tiny $(u\xi_2; s_2)$};
	\node[above right] at (7.5,2) {\tiny $(u\xi_L; s_L)$};
},
\end{equation}
where $1\leq i\leq n$ and the partition function on the right-hand side consists of $L$ vertices with the weights $w_{u\xi_1;s_1}, w_{u\xi_2;s_2}, \dots, w_{u\xi_L, s_L}$. For any $1\leq i<j\leq n$ the operators $\C_i,\C_j$ satisfy the following commutation relations:
\begin{equation}
\C_i(v\mid \Xi,\S)\C_i(u\mid\Xi,\S)=\C_i(u\mid \Xi,\S)\C_i(v\mid\Xi,\S),
\end{equation}
\begin{equation}
\label{exch2}
q\C_j(v\mid\Xi,\S)\C_i(u\mid\Xi,\S)=\frac{u-qv}{u-v}\C_i(u\mid\Xi,\S)\C_j(v\mid\Xi,\S)-\frac{(1-q)u}{u-v}\C_i(v\mid \Xi,\S)\C_j(u\mid\Xi,\S),
\end{equation}
\begin{equation}
\label{exch3}
\C_i(v\mid\Xi,\S)\C_j(u\mid\Xi,\S)=\frac{u-qv}{u-v}\C_j(u\mid\Xi,\S)\C_i(v\mid\Xi,\S)-\frac{(1-q)v}{u-v}\C_j(v\mid\Xi,\S)\C_i(u\mid\Xi,\S).
\end{equation}
All these relations can be proved using the Yang-Baxter equation, namely, set $\NN=\MM=1$ in \eqref{masterYB} and apply a zipper-like argument to move a vertex with weights $\W^{1,1}_{u/v}$ between a pair of rows. The exact argument is given in \cite[Theorem 3.2.2]{BW18}.\footnote{In \cite[Theorem 3.2.2]{BW18} only the \emph{homogeneous} case is considered, when both families $(\xi_1, \dots, \xi_n)$ and $(s_1, \dots, s_n)$ consist of equal parameters. However, the proof given in \cite[Theorem 3.2.2]{BW18} works verbatim for general sequences $\Xi,\S$.}

For the purposes of this work we only need the commutation relation \eqref{exch2}, which we rewrite as
\begin{equation}
\label{Texch}
q\langle \bJ_1, \dots| \C_j(w_1\mid\Xi,\S)\C_i(w_2\mid\Xi,\S)|\bI_1, \dots \rangle=T^{(w_1,w_2)}\langle \bJ_1, \dots|\C_i(w_1\mid\Xi,\S)\C_j(w_2\mid\Xi,\S)|\bI_1, \dots \rangle,
\end{equation}
where $T^{(w_1,w_2)}$ is an operator acting on rational functions in $w_1,w_2$ by
\be
T^{(w_1,w_2)}=q+\frac{w_{2}-qw_1}{w_2-w_1}(\mathfrak{s}^{(w_1, w_2)}-1),\qquad \mathfrak{s}^{(w_1,w_2)}f(w_1,w_2)=f(w_2,w_1).
\ee

Operators $T^{(w_1,w_2)}$ are closely related to the polynomial representation of the \emph{Hecke algebra} via \emph{Demazure-Lusztig operators}. The Hecke algebra $\mathcal H_k$ is a $\mathbb C$-algebra with basis $\{T_\pi\}_{\pi\in S_k}$ enumerated by permutations $\pi\in S_k$ and satisfying the following defining relations
\be
T_{\pi}T_{\tau}=T_{\pi\tau}\quad \text{if}\ \ l(\pi)+l(\tau)=l(\pi\tau),\qquad (T_i-q)(T_i+1)=0,
\ee
where $T_i=T_{\sigma_i}$ denote the basis elements corresponding to the simple transpositions $\sigma_i$ exchanging $i$ and $i+1$. Note that the elements $T_i$ generate $\mathcal H_k$, and thus we can define a representation of $\mathcal H_k$ on the rational functions in $w_1, \dots, w_k$ by setting
\begin{equation}
\label{Tdef}
T_i=q+\frac{w_{i+1}-qw_{i}}{w_{i+1}-w_{i}}(\mathfrak{s}_i-1).
\end{equation}
where $\mathfrak{s}_i$ denotes the operator exchanging the variables $w_i$ and $w_{i+1}$.\footnote{Throughout the text the Hecke algebra appears only via this faithful representation, so identification of the basis elements $T_\pi$ with the operators $T_\pi$ should not cause confusion.} Note that these operators preserve polynomials in $w_1, \dots, w_k$, hence for any rational function $f$ in $w_1,\dots, w_k$ the function $T_\pi f$ is well defined on the domain of $f$, that is, $T_\pi f$ does not have any additional singularities compared to $f$.

\section{Local relations} \label{localSect} 
In this section we use the Yang-Baxter equation to derive two recurrence relations: the first relation captures the local behavior of $q$-moments of the colored height functions, defined in the next section, while the other relation compares rational function of certain form. While the underlying objects behind these relations are different, surprisingly, the relations themselves turn out to be identical. This observation will play a central role in the proof of one of the main results of this work, namely Theorem \ref{qHahnResultTheo}.

\subsection{Ordered permutations.}\label{permSect} Both local relations in this section can be described in terms of {order-preserving} length $p$ subsequences of a generic sequence $\bc=(c_1, c_2, \dots, c_r)$, which can be naturally encoded by subsets $\mathcal I=\{i_1, i_2, \dots, i_p\}\subset \{1, 2, \dots, r\}$:
\be
\bc[\mathcal I]=(c_{i_1}, \dots, c_{i_p}),\quad \text{where}\ i_1<i_2<\dots<i_p.
\ee 
But for our purposes it is more convenient to encode such subsequences using permutations. Recall that the action of the symmetric group on sequences is defined by
\be
\pi.\bc=\(c_{\pi^{-1}(1)}, c_{\pi^{-1}(2)}, \dots, c_{\pi^{-1}(r)}\).
\ee
Then any subsequence of $\bc$ can be constructed as 
\be
\pi.\bc[1,p]:=(c_{\pi^{-1}(1)}, \dots, c_{\pi^{-1}(p)})
\ee
for a permutation $\pi$ satisfying
\be
\pi^{-1}(1)<\pi^{-1}(2)<\dots<\pi^{-1}(p).
\ee
This leads us to the following definitions.

A permutation $\pi$ is called $[a,b]$-ordered if
\be
\pi^{-1}(a)<\pi^{-1}(a+1)<\dots<\pi^{-1}(b).
\ee
More generally, a permutation $\pi$ is $[a_1,b_1]\times [a_2,b_2]$-ordered if it is simultaneously $[a_1,b_1]$-ordered and $[a_2,b_2]$-ordered. With this notation the following can be readily verified:
\begin{lem}\label{simpleLemma}
There is a one-to-one correspondence between integer subsets $\{i_1, \dots, i_p\}\subset [1,r]$ and {\mbox{$[1,p]\times[p+1,r]$}}-ordered permutations $\pi$ determined by setting $\pi^{-1}(a)=i_a$ for $1\leq a\leq p$, where we have ordered $i_1<i_2<\dots<i_p$. Moreover, for a $[1,p]\times[p+1,r]$-ordered permutation $\pi$ we have
\be
l(\pi)=\sum_{i=1}^p \(\pi^{-1}(i)-i\)=\sum_{i=1}^p \pi^{-1}(i) -  \frac{p(p+1)}{2}.
\ee
\end{lem}
\qed

 So, each order-preserving length $p$ subsequence of $\bc=(c_1, \dots, c_r)$ can be uniquely described as $\pi.\bc[1,p]$ for a $[1,p]\times[p+1, r]$-ordered permutation $\pi$\footnote{Though we do not use this terminology, $[1,p]\times[p+1,r]$-ordered permutations are commonly called \emph{$(p,r-p)$-shuffles}.}. We denote the set of such permutations by $S^{p\mid r-p}$.

\subsection{Local relation: $q$-moments.} The first local relation describes the local behavior of a certain observable of the weights $W_{s,t}(\bA,\bB;\bC,\bD)$, which are treated as a stochastic sampling rule for compositions $(\bC,\bD)$ given $(\bA,\bB)$. 

\begin{prop}\label{localRelationAlg} For any compositions $\bA,\bB, \bR\in\mathbb Z_{\geq 0}^{n}$ we have
\begin{multline}
\label{localRelationAlgEq}
\sum_{\bC,\bD}q^{\sum_{i\leq j}R_iD_j} W_{t,s}(\bA,\bB;\bC,\bD)\\
=\sum_{\bP\leq\bR}(s^2/t^2)^{|\bP|}\frac{(s^2/t^2;q)_{|\bR|-|\bP|}(t^2;q)_{|\bP|}}{(s^2;q)_{|\bR|}}q^{\sum_{i<j}(R_i-P_i)P_j}\prod_{i=1}^n\binom{R_i}{P_i}_qq^{\sum_{i\leq j}P_iA_j},
\end{multline}
where the sum in the right-hand side is over compositions $\bP=(P_1, \dots, P_n)$ such that $P_i\leq R_i$.
\end{prop}

The proof of Proposition \ref{localRelationAlg} will be given shortly, but first we want to explain the probabilisitic interpretation of \eqref{localRelationAlgEq}.  For a sequence $\bc=(c_1, c_2, \dots, c_r)$ and a composition $\bA=(A_1,\dots, A_n)$ define
\begin{equation}
\label{defoflocalQ}
\Q_{\geq\bc}(\bA)=\prod_{i=1}^r\prod_{j\geq c_i}q^{A_j}=q^{A_{[c_1,n]}+A_{[c_2,n]}+\dots+A_{[c_r,n]}}.
\end{equation}
where we use the notation $A_{[i,j]}$ from \eqref{Aijdef}. Since the weights $W_{t,s}(\bA,\bB;\bC,\bD)$ are stochastic we can define a (complex-valued) probability measure on configurations of the outgoing edges $(\bC,\bD)$ conditioned on a given fixed incoming configuration $(\bA,\bB)$. We denote this probability measure as
\be
\mathbb P_{t,s}\left[(\bC,\bD)\mid(\bA, \bB)\right]=W_{t,s}(\bA,\bB;\bC,\bD)
\ee
and let $\E_{t,s}$ denote the expectation with respect to this probability measure.

\begin{cor} \label{localRelationSub}  For a sequence $\bc=(c_1, \dots, c_r)$ such that
\be
c_1\leq c_2\leq\dots\leq c_r
\ee
and arbitrary compositions $\bA,\bB$ the following relation holds
\begin{equation}
\label{localRelationSubEqI}
\E_{t,s}\left[\Q_{\geq\bc}(\bD)\ \big|\ (\bA,\bB)\right]=\sum_{\mathcal I}(s^2/t^2)^{l(\mathcal I)}\frac{(s^2/t^2;q)_{r-l(\mathcal I)}(t^2;q)_{l(\mathcal I)}}{(s^2;q)_{r}}q^{-\binom{l(\mathcal I)+1}{2}+\sum_{i\in\mathcal I}i}\ \Q_{\geq\bc[\mathcal I]}(\bA),
\end{equation}
where the sum is over subsets $\mathcal I\subset\{1, 2, \dots, r\}$, $l(\mathcal I)$ denotes the number of elements in $\mathcal I$ and 
\be
\bc[\mathcal I]:=(c_{i_1}, \dots, c_{i_{|\mathcal I|}}).
\ee
Equivalently, 
\begin{equation}
\label{localRelationSubEqT}
\E_{t,s}\left[\Q_{\geq\bc}(\bD)\ \big|\ (\bA,\bB)\right]=\sum_{p=0}^r\sum_{\tau\in S^{p|r-p}}(s^2/t^2)^{p}\frac{(s^2/t^2;q)_{r-p}(t^2;q)_{p}}{(s^2;q)_{r}}q^{l(\tau)} \Q_{\geq\tau.\bc[1,p]}(\bA),
\end{equation}
where the sum is over all $[1,p]\times [p+1, r]$-ordered permutations and 
\be
\tau.\bc[1,p]:=\(c_{\tau^{-1}(1)}, \dots, c_{\tau^{-1}(p)}\).
\ee
\end{cor}
\begin{proof}[Reduction of Corollary \ref{localRelationSub} to Proposition \ref{localRelationAlg}]
The second part follows at once from Lemma~\ref{simpleLemma}. So we focus on the first part.

Let $\bc=1^{R_1}2^{R_2}\dots n^{R_n}$ for a composition $\bR=(R_1,\dots, R_n)$ with $|\bR|=r$, so that
\be
R_i=\#\{j\in[1;r]\mid c_j=i\}.
\ee
 Then for any composition $\bD$ we have
\be
\Q_{\geq\bc}(\bD)=q^{D_{[c_1,n]}+D_{[c_2,n]}+\dots+D_{[c_r,n]}}=q^{R_1D_{[1,n]}+R_2D_{[2,n]}+\dots+R_nD_n}=q^{\sum_{i\leq j}R_iD_j}.
\ee
Similarly, if $\bc[\mathcal I]=1^{P_1}2^{P_2}\dots n^{P_n}$ for a composition $\bP$, then $\Q_{\geq\bc[\mathcal I]}(\bA)=q^{\sum_{i\leq j}P_iA_j}$. Thus, \eqref{localRelationSubEqI} is equivalent to
\be
\E_{t,s}\left[q^{\sum_{i\leq j}R_iD_j}\ \big|\ (\bA,\bB)\right]=\sum_{\bP\leq\bR}\sum_{\substack{\mathcal I\subset\{1,\dots, |\bR|\}\\\bc[\mathcal I]=1^{P_1}2^{P_2}\dots}}(s^2/t^2)^{|\bP|}\frac{(s^2/t^2;q)_{|\bR|-|\bP|}(t^2;q)_{|\bP|}}{(s^2;q)_{|\bR|}} q^{-\binom{l(\mathcal I)+1}{2}+\sum_{i\in\mathcal I}i}\ q^{\sum_{i\leq j}P_iA_j}.
\ee
Note that $P_i=\left|\mathcal I\cap \left\{R_{[1,i-1]}+1,\dots,R_{[1,i]}\right\}\right|$, so to reduce Corollary \ref{localRelationSub} to Proposition \ref{localRelationAlg} it is enough to show for any compositions $\bP\leq\bR$ that
\be
\sum_{\substack{\mathcal I\subset\{1,\dots, |\bR|\}:\\  \left|\mathcal I\cap \left\{R_{[1,i-1]}+1,\dots,R_{[1,i]}\right\}\right|=P_i}}q^{-\binom{l(\mathcal I)+1}{2}+\sum_{i\in\mathcal I}i}=q^{\sum_{i<j}(R_i-P_i)P_j}\prod_{i=1}^n\binom{R_i}{P_i}_q.
\ee
The latter relation readily follows from the standard expression for the $q$-binomial coefficients, \emph{cf.} \cite[Theorem 6.1]{KC02}:
\be
\sum_{\substack{\mathcal I_k\subset\{R_{[1,k-1]}+1,\dots, R_{[1,k]}\}\\ |\mathcal I_k|=P_k}}q^{-\binom{P_k+1}{2}-R_{[1,k-1]}P_k+\sum_{i\in\mathcal I_k}i}=\binom{R_k}{P_k}_q,
\ee
applied for $1\leq k\leq n$ to the decomposition $\mathcal I=\mathcal I_1\sqcup\mathcal I_2\sqcup\dots\sqcup\mathcal I_n$, with $\mathcal I_k:=\mathcal I\cap\left\{R_{[1,k-1]}+1,\dots, R_{[1,k]}\right\}$.
\end{proof}

\begin{proof}[Proof of Proposition \ref{localRelationAlg}]
The main idea is to see the relation \eqref{localRelationAlgEq} as a particular instance of the Yang-Baxter equation. We start with the Yang-Baxter equation \eqref{WYB} with the following parameters and boundary conditions
\begin{equation}
\label{localProofInitialYB}
\tikzbase{1.2}{-3}{
	\draw[fused]
	(-2,0.5) node[above,scale=0.7] {\color{black} $\bB$} -- (-1,-0.5) -- (1.5,-0.5) node[right,scale=0.6] {\color{black} $\bm 0$};
	\draw[fused] 
	(-2,-0.5) node[below,scale=0.7] {\color{black} $\bA$} -- (-1,0.5)  -- (1.5,0.5) node[right,scale=0.6] {\color{black} $\bR$};
	\draw[fused] 
	(0,-1.5) node[below,scale=0.7] {\color{black} $\bm X$} -- (0,0) -- (0,1.5) node[above,scale=0.6] {\color{black} $\bm Y$};
	\node[above right] at (0,-0.5) {\tiny{ $W_{t,\varepsilon}$}};
	\node[above right] at (0,0.5) {\tiny{ $W_{s,\varepsilon}$}};
	\node[right] at (-1.5,0) {\ \tiny{$W_{t,s}$}};
}\quad
=
\quad
\tikzbase{1.2}{-3}{
	\draw[fused] 
	(-1,1) node[left,scale=0.7] {\color{black} $\bB$} -- (1.2,1) -- (2.2,0) node[below,scale=0.6] {\color{black} $\bm 0$};
	\draw[fused] 
	(-1,0) node[left,scale=0.7] {\color{black} $\bA$} -- (1.2,0) -- (2.2,1) node[above,scale=0.6] {\color{black} $\bR$};
	\draw[fused] 
	(0,-1) node[below,scale=0.7] {\color{black} $\bm X$} -- (0,0.5) -- (0,2) node[above,scale=0.6] {\color{black} $\bm Y$};
	\node[above right] at (0,0) {\tiny{ $W_{s,\varepsilon}$}};
	\node[above right] at (0,1) {\tiny{ $W_{t,\varepsilon}$}};
	\node[right] at (1.7,0.5) {\ \tiny{$W_{t,s}$}};
}
\end{equation}
where $t,s,\bA,\bB,\bR$ are the same as in \eqref{localRelationAlgEq}, while $\varepsilon$ is an arbitrary parameter and $\bm X,\bm Y$ are arbitrary compositions satisfying
\be
\bm Y=\bm X+\bA+\bB-\bR.
\ee
The next step is to take $\ve\to 0$ and to analytically continue the labels of the vertical edges. Namely, consider the following vertex weights:
\be
\tikz{1}{
	\draw[fused] (-1,0) -- (0,0) -- (1,0);
	\draw[cont] (0,-1) -- (0,0) -- (0,1);
	\node[left] at (-1,0) {\tiny $\bJ$};\node[right] at (1,0) {\tiny $\bL$};
	\node[below] at (0,-1) {\tiny $\bm \alpha$};\node[above] at (0,1) {\tiny $\bm \beta$};
	\node[above right] at (0,0.1) {\tiny $\widetilde{W}_{t}$};
}=\widetilde{W}_{t}(\bm\alpha,\bJ;\bm\beta,\bL):=\1_{\bm\alpha q^{\bJ}=\bm\beta q^{\bL}}\ t^{-2|\bL|}(t^2;q)_{|\bL|}q^{\sum_{i<j}-L_iL_j}\prod_{i=1}^n\alpha_i^{L_{[1,i-1]}}\frac{(q^{1-L_i}\alpha_i;q)_{L_i}}{(q;q)_{L_i}},
\ee
where $\bm \alpha=(\alpha_1, \dots, \alpha_n)\in\mathbb C^n$ and $\bm \beta=(\beta_1, \dots, \beta_n)\in\mathbb C^n$ are analytically continued compositions and we use the notation $\bm\alpha q^{\bI}=(\alpha_1q^{I_1}, \dots, \alpha_nq^{I_n})$. The weights $\widetilde{W}_t$ are obtained from the weights $W_{t,\ve}$ by taking $\ve\to0$ and performing analytic continuation: using \eqref{Wdef} one can readily verify that
\begin{equation}
\label{localProofCont}
\tikz{1}{
	\draw[fused] (-1,0) -- (0,0) -- (1,0);
	\draw[cont] (0,-1) -- (0,0) -- (0,1);
	\node[left] at (-1,0) {\tiny $\bJ$};\node[right] at (1,0) {\tiny $\bL$};
	\node[below] at (0,-1) {\tiny $q^{\bI}$};\node[above] at (0,1) {\tiny $q^{\bK}$};
	\node[above right] at (0,0.1) {\tiny $\widetilde{W}_{t}$};
}=\restr{\(\varepsilon^{-2|\bL|}
\tikz{1}{
	\draw[fused] (-1,0) -- (0,0) -- (1,0);
	\draw[fused] (0,-1) -- (0,0) -- (0,1);
	\node[left] at (-1,0) {\tiny $\bJ$};\node[right] at (1,0) {\tiny $\bL$};
	\node[below] at (0,-1) {\tiny ${\bI}$};\node[above] at (0,1) {\tiny ${\bK}$};
	\node[above right] at (0,0.1) {\tiny ${W}_{t,\varepsilon}$};
}\)}{\varepsilon=0}.
\end{equation}
Note that the weight $\widetilde{W}_{t}(\bm\alpha,\bJ;\bm\alpha q^{\bJ-\bL},\bL)$ is a polynomial in $\alpha_1, \alpha_2, \dots, \alpha_n$, with the highest degree term\footnote{The degree of a monomial $\alpha_1^{l_1}\dots\alpha_n^{l_n}$ is $l_1+\dots+l_n$.} having an explicit expression
\begin{equation}
\label{localProofHighestDegree}
\widetilde{W}_{t}(\bm\alpha,\bJ;\bm\alpha q^{\bJ-\bL},\bL)=(-1)^{|\bL|}t^{-2|\bL|}(t^2;q)_{|\bL|}q^{-|\bL|(|\bL|-1)/2}\prod_{i=1}^n\frac{\alpha_i^{L_{[1,i]}}}{(q;q)_{L_i}}+\ \text{lower\ degree\ terms}.
\end{equation}

Multiplying both sides of \eqref{localProofInitialYB} by $\varepsilon^{-2|\bR|}$, taking $\varepsilon=0$ and applying \eqref{localProofCont} we obtain
\begin{equation}
\label{localProofContYB}
\sum_{\substack{\bC,\bD:\\ \bC+\bD=\bA+\bB}}
\quad
\tikzbase{1.2}{-3}{
	\draw[fused]
	(-2,0.5) node[above,scale=0.7] {\color{black} $\bB$} -- (-1,-0.5) node[below,scale=0.7] {\color{black} $\bD$} -- (1.5,-0.5) node[right,scale=0.6] {\color{black} $0$};
	\draw[fused] 
	(-2,-0.5) node[below,scale=0.7] {\color{black} $\bA$} -- (-1,0.5) node[above,scale=0.7] {\color{black} $\bC$} -- (1.5,0.5) node[right,scale=0.6] {\color{black} $\bR$};
	\draw[cont] 
	(0,-1.5) node[below,scale=0.7] {\color{black} $\bm \alpha$} -- (0,0) node[left,scale=0.7] {\color{black} $\bm \alpha q^{\bD}$} -- (0,1.5) node[above,scale=0.6] {\color{black} $\bm \alpha q^{\bA+\bB-\bR}$};
	\node[above right] at (0,-0.5) {\tiny{ $\widetilde W_{t}$}};
	\node[above right] at (0,0.5) {\tiny{ $\widetilde W_{s}$}};
	\node[right] at (-1.5,0) {\ \tiny{$W_{t,s}$}};
}\quad
=
\sum_{\bP\leq\bR}
\quad
\tikzbase{1.2}{-3}{
	\draw[fused] 
	(-1,1) node[left,scale=0.7] {\color{black} $\bB$} -- (1.2,1)  node[above,scale=0.6] {\color{black} $\bP$} -- (2.2,0) node[below,scale=0.6] {\color{black} $0$};
	\draw[fused] 
	(-1,0) node[left,scale=0.7] {\color{black} $\bA$} -- (1.2,0) node[below,scale=0.6] {\color{black} $\bR-\bP$} -- (2.2,1) node[above,scale=0.6] {\color{black} $\bR$};
	\draw[cont] 
	(0,-1) node[below,scale=0.7] {\color{black} $\bm \alpha$} -- (0,0.5) node[left,scale=0.7] {\color{black} $\bm \alpha q^{\bA-\bR+\bP}$} -- (0,2) node[above,scale=0.6] {\color{black} $\bm \alpha q^{\bA+\bB-\bR}$};
	\node[above right] at (0,0) {\tiny{ $\widetilde W_{s}$}};
	\node[above right] at (0,1) {\tiny{ $\widetilde W_{t}$}};
	\node[right] at (1.7,0.5) {\ \tiny{$W_{t,s}$}};
}
\end{equation}
for any $\bm \alpha=(q^{X_1}, q^{X_2}, \dots, q^{X_n})$, where $\bm X=(X_1, \dots, X_n)$ is a composition satisfying $\bm X+\bR-\bA-\bB\geq0$. Note that we have explicitly indicated the summations over internal edges in both sides of \eqref{localProofContYB}, specifying the internal labels. In particular, both sides are clearly finite sums of polynomial functions in $\alpha_1, \dots, \alpha_n$ over $\mathbb C(q,s,t)$, hence the equality for $\bm \alpha=q^{\bm X}$ can be continued to any $\bm\alpha=(\alpha_1, \dots, \alpha_n)\in\mathbb C^n$. 

Having established \eqref{localProofContYB} for arbitrary $\bm \alpha$, we can now set $\bm{\alpha}=(\alpha, \alpha, \dots, \alpha)$, that is, all $\alpha_i$ are equal to a single variable $\alpha$. Then both sides of \eqref{localProofContYB} become polynomials in $\alpha$ of degree $nR_1+(n-1)R_2+\dots+ R_n$. Taking the highest degree coefficients using \eqref{localProofHighestDegree} we get
\begin{multline*}
\sum_{\substack{\bC,\bD:\\ \bC+\bD=\bA+\bB}}\ W_{t,s}(\bA,\bB;\bC,\bD)\times(-1)^{|\bR|}s^{-2|\bR|}(s^2;q)_{|\bR|}q^{-\binom{|\bR|}{2}}\prod_{i=1}^n\frac{q^{D_iR_{[1,i]}}}{(q;q)_{R_i}}\\
=\sum_{\bP\leq\bR}(-1)^{|\bR|}s^{-2|\bR|+2|\bP|}t^{-2|\bP|}(s^2/t^2;q)_{|\bR|-|\bP|}(t^2;q)_{|\bP|}q^{-\binom{|\bR|-|\bP|}{2}-\binom{|\bP|}{2}}\prod_{i=1}^n\frac{q^{(A_i-R_i+P_i)P_{[1,i]}}}{(q;q)_{P_i}(q;q)_{R_i-P_i}},
\end{multline*}
which is equivalent to \eqref{localRelationAlgEq} after elementary algebraic manipulations.
\end{proof}

\begin{rem}
\normalfont
We have not used \eqref{stochW} during the proof of Proposition \ref{localRelationAlg}, so one can use the Yang-Baxter equation \eqref{WYB} to show that the weights $W_{t,s}$ are stochastic.
\end{rem}

\begin{rem}\normalfont The local relation above also holds for more general weights $\W^{\NN,\MM}_{z}$ from \eqref{fullyfusedexpression}, with the analogue of Proposition \ref{localRelationAlg} having the following form:
\begin{multline}
\label{fullyFusedLocal}
\sum_{\bC,\bD}q^{\sum_{i\leq j}R_iD_j} \W^{\NN,\MM}_{z}(\bA,\bB;\bC,\bD)=\sum_{\substack{\bX,\bY,\bZ\\ \bX+\bY+\bZ=\bR}}q^{-\MM|\bX|}\(zq^{\NN-\MM}\)^{|\bY|}\frac{(z;q)_{|\bX|}(q^{-\NN}/z;q)_{|\bY|}(zq^{\NN-\MM};q)_{|\bZ|}}{(zq^{-\MM};q)_{|\bR|}}\\
q^{\sum_{i<j}Y_iX_j+Z_iX_j+Z_iY_j}\prod_{i=1}^n\frac{(q;q)_{R_i}}{(q;q)_{X_i}(q;q)_{Y_i}(q;q)_{Z_i}}q^{\sum_{i\leq j}Y_iA_j+X_i(A_j+B_j)}.
\end{multline}
Using the same methods as in Section \ref{qHahnSec} below, this identity can be written in a form similar to \eqref{nonlocalQ}, describing the local behavior of the height functions around a vertex with weights $\W^{\NN,\MM}_{z}$. Since \eqref{fullyFusedLocal} is not used in the present work we omit its proof, but the identity from Proposition \ref{localRelationAlg}  constitutes the majority of that proof: plugging the explicit expression \eqref{fullyfusedexpression} into the left-hand side of \eqref{fullyFusedLocal} one obtains a double summation which can be taken in the order allowing two subsequent applications of Proposition \ref{localRelationAlg}, readily leading to the right-hand side of \eqref{fullyFusedLocal} after algebraic manipulations.
\end{rem}

\begin{rem}\normalfont
For the six-vertex model the recurrence relation analogous to Corollary \ref{localRelationSub} can be found, in increasing generality, in \cite{BG18}, \cite{BGW19} and \cite{BK20} (in the former two the relation is called \emph{the four point relation}). We note that the local relation for the $q$-Hahn model from the current work does not directly follows from those earlier results, and vice versa, the local relations for the six-vertex model cannot be obtained as degenerations of Corollary \ref{localRelationSub}.
\end{rem}

\subsection{Local relation: rational functions} Now we turn to the second relation, which considers certain rational functions in $w_1, \dots, w_r$. Recall that the action of the Hecke algebra $\mathcal H_r$ on the rational functions in $w_1, \dots, w_r$ is generated by 
\be
T_i=q+\frac{w_{i+1}-qw_i}{w_{i+1}-w_i}(\mathfrak s_i-1),\qquad \mathfrak s_{i}f(\dots, w_i, w_{i+1} \dots)=f(\dots, w_{i+1}, w_i, \dots).
\ee

\begin{prop} \label{localRelationRat}For any integer $r\in\mathbb Z_{\geq 0}$ and any complex parameters $t,s,\lambda$ such that $s^2\neq q^{-n}$ for any $n\in\mathbb Z_{\geq 0}$ the following equality of rational functions in $w_1, \dots, w_r$ holds
\be
\prod_{i=1}^r\frac{1-\lambda t^{-2}w_i}{1-\lambda w_i}=\sum_{p=0}^r\sum_{\tau\in S^{p|r-p}}(s^2/t^2)^p\frac{(s^2/t^2;q)_{r-p}(t^2;q)_p}{(s^2;q)_r}\ T_{\tau^{-1}}\(\prod_{i=1}^p\frac{1-\lambda s^{-2}w_i}{1-\lambda w_i}\),
\ee
where the sum is taken over $[1,p]\times[p+1, r]$-ordered permutations $\tau$, that is, permutations satisfying
\be
\tau^{-1}(1)<\tau^{-1}(2)<\dots<\tau^{-1}(p),\qquad \tau^{-1}(p+1)<\tau^{-1}(p+2)<\dots<\tau^{-1}(r).
\ee
\end{prop}
\begin{proof}

The main idea of the proof is to take two columns of vertices with weights $w_{u;s}$, find an explicit expression for their partition function using exchange relations and then apply the deformed Yang-Baxter equation \eqref{defhsYB} to exchange the columns. The resulting identity turns out to be equivalent to the claim. Throughout the proof we use additional complex parameters $x,y,\ve$ and we fix a composition $\bm 1^r=(1, 1, \dots, 1)\in \mathbb Z^r$. 

We start with defining the two-column functions mentioned above. For any composition $\bP\leq\bm 1^r$ define
\be
f_{\bm 1^r-\bP\mid\bP}(w_1, \dots, w_r\mid x,y)=\langle \bm 1^r-\bP, \bP|\ \C_{r}(w_1\mid\Xi_{x,y},\S_{x,y})\dots\C_{1}(w_r\mid\Xi_{x,y},\S_{x,y})\ |\bm 0, \bm 0\rangle
\ee
where we set
\be
\Xi_{x,y}=\(\frac{\lambda}{xs}, \frac{\lambda}{y\epsilon}\), \quad \S_{x,y}=\(\frac{s}{x}, \frac{\epsilon}{y}\),
\ee
and $\C_i$ are the row operators \eqref{rowdef}. Alternatively, we can depict these functions as
\be
f_{\bm 1^r-\bP\mid\bP}(w_1, \dots, w_r\mid x,y)\quad =\quad\tikz{1}{
	\foreach\y in {2,...,5}{
		\draw[unfused] (1,\y) -- (6,\y);
	}
	\draw[fused] (2,1) -- (2,6);
	\draw[fused] (4,1) -- (4,6);
	\node[above right] at (2,5) {\small $\(\frac{\lambda w_1}{xs}; \frac{s}{x}\)$};
	\node[above right] at (4,5) {\small $\(\frac{\lambda w_1}{y\epsilon}; \frac{\epsilon}{y}\)$};
	
	\node[above right] at (2,3) {\small $\(\frac{\lambda w_{r-1}}{xs}; \frac{s}{x}\)$};
	\node[above right] at (4,3) {\small $\(\frac{\lambda w_{r-1}}{y\epsilon}; \frac{\epsilon}{y}\)$};
	
	\node[above right] at (2,2) {\small $\(\frac{\lambda w_r}{xs}; \frac{s}{x}\)$};
	\node[above right] at (4,2) {\small $\(\frac{\lambda w_r}{y\epsilon}; \frac{\epsilon}{y}\)$};
	\node[above] at (4,6) {$\bP$};
	\node[above] at (2,6) {$\bm{1}^r-\bP$};
	\node[left] at (1,5) {\small $r$};
	\node[left] at (1,4) {\small $\vdots$};
	\node[left] at (1,3) {\small $2$};
	\node[left] at (1,2) {\small $1$};
	\node[below] at (4,1) {$\bm 0$};
	\node[below] at (2,1) {$\bm 0$};
	\node[right] at (6,5) {\small $0$};
	\node[right] at (6,4) {\small $\vdots$};
	\node[right] at (6,3) {\small $0$};
	\node[right] at (6,2) {\small $0$};
}
\ee
where the diagram consists of $2r$ vertices with the weights $w_{z,t}$ from \eqref{hsnot}. The vertex at the intersection of row $i$ and the first column has the spectral and spin parameters $\(\frac{\lambda w_{r-i+1}}{xs}; \frac{s}{x}\)$, while the parameters of the other vertex in the same row are $\(\frac{\lambda w_{r-i+1}}{y\epsilon}; \frac{\epsilon}{y}\)$.

Now we want to use the exchange relations to obtain an explicit expression for $f_{\bm 1^r-\bP\mid\bP}$. Since $x,y$ are fixed in this part of an argument, for now we write $\Xi,\S$ instead of $\Xi_{x,y}, \S_{x,y}$. The exchange relation \eqref{Texch} implies that
\begin{multline*}
\langle \bm 1^r-\bP, \bP|\ \C_{\pi(r)}(w_1\mid\Xi,\S)\dots\C_{\pi(1)}(w_r\mid\Xi,\S)\ |\bm 0, \bm 0\rangle\\
=q^{-1}T_{r-i}\ \langle \bm 1^r-\bP, \bP|\ \C_{\pi \sigma_i(r)}(w_1\mid\Xi,\S)\dots\C_{\pi \sigma_i(1)}(w_r\mid\Xi,\S)\ |\bm 0, \bm 0\rangle
\end{multline*}
for any permutation $\pi$ such that $l(\pi)<l(\pi\sigma_i)$, where $\sigma_i$ is the transposition exchanging $i$ and $i+1$. The index $r-i$ of the inverse Hecke algebra operator comes from the fact that $w$ are ordered in the reverse order compared to $\C$. Iterating the identity above for a reduced expression $\tau=\sigma_{i_1}\dots\sigma_{i_l}$ we obtain
\begin{multline}
\label{localRatLemmaProof}
\langle \bm 1^r-\bP, \bP|\ \C_{r}(w_1\mid\Xi,\S)\dots\C_{1}(w_r\mid\Xi,\S)\ |\bm 0, \bm 0\rangle\\
=q^{-l(\tau)}T_{r-i_l}\dots T_{r-i_1} \langle \bm 1^r-\bP, \bP|\ \C_{\tau^{-1}(r)}(w_1\mid\Xi,\S)\dots\C_{\tau^{-1}(1)}(w_r\mid\Xi,\S)\ |\bm 0, \bm 0\rangle\\
=q^{-l(\tau)}T_{\widetilde{\tau}^{-1}} \langle \bm 1^r-\bP, \bP|\ \C_{\tau^{-1}(r)}(w_1\mid\Xi,\S)\dots\C_{\tau^{-1}(1)}(w_r\mid\Xi,\S)\ |\bm 0, \bm 0\rangle,
\end{multline}
where $\widetilde{\tau}$ is the image of $\tau$ under the automorphism of $S_r$ sending $\sigma_i$ to $\sigma_{r-i}$.

For now we focus on the partition function in the right-hand side of \eqref{localRatLemmaProof}, namely
\begin{equation}
\label{partition}
\langle \bm 1^r-\bP, \bP|\ \C_{\tau^{-1}(r)}(w_1\mid\Xi,\S)\dots\C_{\tau^{-1}(1)}(w_r\mid\Xi,\S)\ |\bm 0, \bm 0\rangle=\tikz{1}{
	\foreach\y in {2,...,5}{
		\draw[unfused] (1,\y) -- (6,\y);
	}
	\draw[fused] (2,1) -- (2,6);
	\draw[fused] (4,1) -- (4,6);
	\node[above right] at (2,5) {\small $\(\frac{\lambda w_1}{xs}; \frac{s}{x}\)$};
	\node[above right] at (4,5) {\small $\(\frac{\lambda w_1}{y\epsilon}; \frac{\epsilon}{y}\)$};
	
	\node[above right] at (2,3) {\small $\(\frac{\lambda w_{r-1}}{xs}; \frac{s}{x}\)$};
	\node[above right] at (4,3) {\small $\(\frac{\lambda w_{r-1}}{y\epsilon}; \frac{\epsilon}{y}\)$};
	
	\node[above right] at (2,2) {\small $\(\frac{\lambda w_r}{xs}; \frac{s}{x}\)$};
	\node[above right] at (4,2) {\small $\(\frac{\lambda w_r}{y\epsilon}; \frac{\epsilon}{y}\)$};
	\node[above] at (4,6) {$\bP$};
	\node[above] at (2,6) {$\bm{1}^r-\bP$};
	\node[left] at (1,5) {\small $\tau^{-1}(r)$};
	\node[left] at (1,4) {\small $\vdots$};
	\node[left] at (1,3) {\small $\tau^{-1}(2)$};
	\node[left] at (1,2) {\small $\tau^{-1}(1)$};
	\node[below] at (4,1) {$\bm 0$};
	\node[below] at (2,1) {$\bm 0$};
	\node[right] at (6,5) {\small $0$};
	\node[right] at (6,4) {\small $\vdots$};
	\node[right] at (6,3) {\small $0$};
	\node[right] at (6,2) {\small $0$};
}
\end{equation}
It turns out that the partition function above can be explicitly computed for a certain choice of $\tau$. Namely, for a composition $\bP\leq \bm 1^r$ let $\tau_{\bP}\in S_r$ be the unique permutation satisfying
\be
\tau_{\bP}^{-1}(1)<\tau_{\bP}^{-1}(2)<\dots<\tau_{\bP}^{-1}(r-|\bP|),\qquad \tau_{\bP}^{-1}(r-|\bP|+1)<\dots<\tau_{\bP}^{-1}(r),
\ee
\be
P_{\tau_{\bP}^{-1}(1)}=\dots =P_{\tau_{\bP}^{-1}(r-|\bP|)}=0,\qquad P_{\tau_{\bP}^{-1}(r-|\bP|+1)}=\dots =P_{\tau_{\bP}^{-1}(r)}=1.
\ee
In other words, $\tau_{\bP}$ is the minimal permutation such that $\tau^{-1}_{\bP}(r-|\bP|+1), \dots, \tau^{-1}_{\bP}(r)$ are the colors present in $\bP$. Then for $\tau=\tau_{\bP}$ the partition function \eqref{partition} has only one configuration with non-vanishing contribution: the path of color $\tau^{-1}_{\bP}(i)$ for $1\leq i\leq r-|\bP|$ enters at row $i$ and immediately goes upwards, exiting through the first column, while the path of color $\tau^{-1}_{\bP}(i)$ for $r-|\bP|< i\leq r$ enters at row $i$, crosses the first column and exits through the second one. Since each horizontal edge can be occupied by at most one path, this is the unique possible configuration. Computing the vertex weights for this configuration gives
\begin{multline*}
\langle \bm{1}^r-\bP, \bP|\ \C_{\tau^{-1}_{\bP}(r)}(w_1\mid\Xi,\S)\dots\C_{\tau^{-1}_{\bP}(1)}(w_r\mid\Xi,\S)\ |\bm 0, \bm 0\rangle\\
=\prod_{i=1}^{r-p}\frac{1-s^2x^{-2}q^{i-1}}{1-\lambda x^{-2}w_{r+1-i}}\prod_{i=r-p+1}^{r}\frac{(s^2x^{-2}-\lambda x^{-2}w_{r+1-i})q^{\#\left\{j\leq r-p \mid \tau^{-1}_{\bP}(j)>\tau^{-1}_{\bP}(i)\right\}}}{1-\lambda x^{-2}w_{r+1-i}}\frac{1-\epsilon^2y^{-2}q^{i-r+p-1}}{1-\lambda y^{-2}w_{r+1-i}}\\
=(s^2/x^2)^p(s^2/x^2;q)_{r-p}(\epsilon^2/y^2;q)_{p}\ q^{l(\tau_{\bP})}\prod_{i=1}^r\frac{1}{1-\lambda x^{-2}w_i}\prod_{i=1}^{p}\frac{1-\lambda s^{-2}w_i}{1-\lambda y^{-2}w_i}.
\end{multline*}
where we set $p:=|\bP|$ and we use the definition of $\tau_{\bP}$ to deduce
\be
\#\left\{i>r-p, j\leq r-p \mid \tau^{-1}_{\bP}(j)>\tau^{-1}_{\bP}(i)\right\}=l(\tau_{\bP}).
\ee
Plugging the expression above for the ``frozen" partition function into \eqref{localRatLemmaProof} we get
\be
f_{\bm{1}^{r}-\bP\mid\bP}(w_1, \dots, w_r\mid x,y)=(s^2/x^2)^{|\bP|}(s^2/x^2;q)_{r-|\bP|}(\epsilon^2/y^2;q)_{|\bP|}\ T_{\widetilde{\tau}_{\bP}^{-1}}\(\prod_{i=1}^r\frac{1}{1-\lambda x^{-2} w_i}\prod_{i=1}^{|\bP|}\frac{1-\lambda s^{-2}w_i}{1-\lambda y^{-2} w_i}\).
\ee

The next step is based on the deformed Yang Baxter equation \eqref{defhsYB}, used in the following form:
\begin{equation}
\label{localRatProofDefYB}
\tikzbase{1.2}{-3}{
	\draw[fused]
	(-1,-1) node[below,scale=0.7] {\color{black} $\bI$} -- (-1,0.5) -- (0.5,2) node[right,scale=0.6] {\color{black} $\bL$};
	\draw[fused] 
	(0.5,-1) node[below,scale=0.7] {\color{black} $\bJ$} -- (0.5,0.5)  -- (-1,2) node[left,scale=0.6] {\color{black} $\bK$};
	\draw[unfused] 
	(-2,-0.5) node[left, scale=0.7] {\color{black} $i$} -- (0,-0.5) -- (1.7,-0.5) node[right ,scale=0.6] {\color{black} $0$};
	\node[above right] at (-1,-0.5) {\tiny{ ${\(\frac{w\lambda}{xs};\frac{s}{x}\)}$}};
	\node[above right] at (0.5,-0.5) {\tiny{ ${\(\frac{w\lambda}{y\epsilon};\frac{\epsilon}{y}\)}$}};
	\node[right] at (-0.25,1.25) {\ \tiny{$\({\frac{\epsilon}{x},\frac{\epsilon}{y}}\)$}};
}\quad
=
\quad
\tikzbase{1.2}{-3}{
	\draw[fused]
	(-1,-1) node[below,scale=0.7] {\color{black} $\bI$} -- (0.5,0.5) -- (0.5,2) node[right,scale=0.6] {\color{black} $\bL$};
	\draw[fused] 
	(0.5,-1) node[below,scale=0.7] {\color{black} $\bJ$} -- (-1,0.5)  -- (-1,2) node[left,scale=0.6] {\color{black} $\bK$};
	\draw[unfused] 
	(-2,1) node[left, scale=0.7] {\color{black} $i$} -- (0,1) -- (1.7,1) node[right ,scale=0.6] {\color{black} $0$};
	\node[above right] at (-1,1) {\tiny{ ${\(\frac{w\lambda}{ys};\frac{s}{y}\)}$}};
	\node[above right] at (0.5,1) {\tiny{ ${\(\frac{w\lambda}{x\epsilon};\frac{\epsilon}{x}\)}$}};
	\node[right] at (-0.25,-0.25) {\ \tiny{$\({\frac{\epsilon}{x},\frac{\epsilon}{y}}\)$}};
}
\end{equation}
Here the tilted vertices have the weights $W_{\epsilon/x,\epsilon/y}$, while the other vertices have $w$-weights with indicated parameters. Introducing a new operator $\mathcal R$ with coefficients
\be
\langle \bK,\bL\mid \mathcal R \mid\bI,\bJ\rangle=W_{\epsilon/x, \epsilon/y}(\bJ, \bI; \bK, \bL),
\ee
we can rewrite \eqref{localRatProofDefYB} as
\be
\mathcal R\ \C_i(w\mid \Xi_{x,y}, \S_{x,y})=\C_i(w\mid \Xi_{y,x}, \S_{y,x})\mathcal R.
\ee
Using this exchange relation, we obtain
\begin{multline*}
f_{\bm{1}^r|\bm 0}(w_1, \dots, w_r\mid y,x)=\langle \bm{1}^r, \bm 0|\ \C_{r}(w_1\mid\Xi_{y,x},\S_{y,x})\dots\C_{1}(w_r\mid\Xi_{y,x},\S_{y,x})\mathcal R\ |\bm 0, \bm 0\rangle\\
=\langle \bm{1}^r, \bm 0|\ \mathcal R\ \C_{r}(w_1\mid\Xi_{x,y},\S_{x,y})\dots\C_{1}(w_r\mid\Xi_{x,y},\S_{x,y})\ |\bm 0, \bm 0\rangle=\sum_{\bP\leq \bm{1}^r}\frac{(x^2/y^2;q)_{|\bP|}}{(\epsilon^2/y^2;q)_{|\bP|}}f_{\bm{1}^r-\bP|\bm \bP}(w_1, \dots, w_r\mid x,y),
\end{multline*}
where in the last equation we have explicitly evaluated the action of $\mathcal R$ on $\langle \bm{1}^r, \bm 0|$. Plugging the explicit expressions for the functions $f_{\bm{1}^r-\bP|\bm \bP}$ gives
\begin{multline*}
(s^2/y^2;q)_{r}\prod_{i=1}^r\frac{1}{1-\lambda y^{-2} w_i}\\
=\sum_{\bP\leq \bm{1}^r}(s^2/x^2)^{|\bP|}(s^2/x^2;q)_{r-|\bP|}(x^2/y^2;q)_{|\bP|}\ T_{\widetilde{\tau}_{\bP}^{-1}}\(\prod_{i=1}^r\frac{1}{1-\lambda x^{-2} w_i}\prod_{i=1}^{|\bP|}\frac{1-\lambda s^{-2}w_i}{1-\lambda y^{-2} w_i}\).
\end{multline*}
The remainder of the proof consists of algebraic manipulations bringing the identity above to the desired form. First, we set $x=t$ and $y=1$. Then, recalling the discussion from Section \ref{permSect}, the summation over $\bP \leq \bm{1}^r$ is equivalent to the summation over subsets $\mathcal I\subset \{1,2,\dots,r\}$, which is in turn equivalent to the summation over pairs $(p,\pi)$, where $p\leq r$ and $\pi\in S^{r-p|p}$. Moreover, tracing back these equivalences, one can readily see that $\bP$ corresponds to the pair $(|\bP|, \tau_{\bP})$, so we get
\be
\prod_{i=1}^r\frac{1}{1-\lambda w_i}=\sum_{p=0}^r\sum_{\tau\in S^{r-p|p}}(s^2/t^2)^{p}\frac{(s^2/t^2;q)_{r-p}(t^2;q)_{p}}{(s^2;q)_{r}}\ T_{\widetilde{\tau}^{-1}}\(\prod_{i=1}^r\frac{1}{1-\lambda t^{-2} w_i}\prod_{i=1}^{p}\frac{1-\lambda s^{-2}w_i}{1-\lambda w_i}\).
\ee
The claim follows since the action of the Hecke algebra commutes with the multiplication by $\prod_{i=1}^r(1-\lambda t^{-2} w_i)^{-1}$ and $\tau\in S^{r-p\mid p}$ if and only if $\widetilde{\tau}\in S^{p\mid r-p}$.
\end{proof}
\begin{rem}\label{spinHLRemark} \normalfont
The functions $f_{\bm{1}^r-\bP|\bP}$ used throughout the proof of Proposition \ref{localRelationRat} are particular cases of the \emph{non-symmetric spin Hall-Littlewood functions} introduced in \cite{BW18}. The first part of the proof partially reproduces the exchange relations from \cite[Section 5]{BW18}, while the second part can be interpreted as a new recurrence relation based on the deformed Yang-Baxter relation.
\end{rem}


\section{$q$-moments of height function}\label{qHahnSec} In this section we describe and prove the first main result of this work, namely, the integral expression for $q$-moments of the colored height functions of the colored diagonally inhomogeneous $q$-Hahn vertex model.

\subsection{The model.}\label{themodel} As noted before, the weights $W_{t,s}(\bA,\bB;\bC,\bD)$ are stochastic, so they can be used to interpret a vertex as a random sampling of an outgoing configuration given an incoming one. Now we  extend this sampling construction by considering a grid of vertices with a certain choice of the spin parameters of the weights $W_{t,s}(\bA,\bB;\bC,\bD)$. Then, given an incoming boundary conditions, we can sample a random configuration of the grid, that is, an assignment of a random composition to each lattice edge. Let us describe the construction in more detail.

Fix the number of colors $n$ and consider a square grid of size $N\times N$. We identify the grid with the subset of $\mathbb Z_{\geq 0}^2$ formed by intersections of $N$ rows $\mathbb Z\times\{j\}$ directed to the right and $N$ columns $\mathbb Z\times\{i\}$ directed upwards, where $1\leq i,j\leq N$. We treat this grid as a vertex model with vertices $(i,j)$ for $i,j\in[1,N]$, vertical edges $(i, j)\to(i,j+1)$ for $i\in [1,N], j\in[0,N]$ and horizontal edges $(i, j)\to(i+1,j)$ for $i\in[0,N], j\in[1,N]$. A \emph{configuration} $\Sigma$ of the model is an assignment of compositions to the edges listed above. We use $\bA^{(i,j)}$ (resp. $\bB^{(i,j)}$) to denote the composition of the vertical (resp. horizontal) edge starting at $(i,j)$.

\begin{figure}
\tikz{0.9}{
	\foreach\y in {2,...,5}{
		\draw[fused] (1,\y*1.5-1) -- (9.5,\y*1.5-1);
	}
	\foreach\x in {1,...,4}{
		\draw[fused] (2*\x,1) -- (2*\x,8);
	}
	
	\node[above right] at (2,6.5) {\tiny $\sqrt{\frac{\lambda_3}{\kappa_4}}, \sqrt{\frac{\lambda_3}{\sigma_1}}$};
	\node[above right] at (4,6.5) {\tiny $\sqrt{\frac{\lambda_2}{\kappa_4}}, \sqrt{\frac{\lambda_2}{\sigma_2}}$};
	\node[above right] at (6,6.5) {\tiny $\sqrt{\frac{\lambda_1}{\kappa_4}}, \sqrt{\frac{\lambda_1}{\sigma_3}}$};
	\node[above right] at (8,6.5) {\tiny $\sqrt{\frac{\lambda_0}{\kappa_4}}, \sqrt{\frac{\lambda_0}{\sigma_4}}$};
	
	\node[above right] at (2,5) {\tiny $\sqrt{\frac{\lambda_2}{\kappa_3}}, \sqrt{\frac{\lambda_2}{\sigma_1}}$};
	\node[above right] at (4,5) {\tiny $\sqrt{\frac{\lambda_1}{\kappa_3}}, \sqrt{\frac{\lambda_1}{\sigma_2}}$};
	\node[above right] at (6,5) {\tiny $\sqrt{\frac{\lambda_0}{\kappa_3}}, \sqrt{\frac{\lambda_0}{\sigma_3}}$};
	
	\node[above right] at (2,3.5) {\tiny $\sqrt{\frac{\lambda_1}{\kappa_2}}, \sqrt{\frac{\lambda_1}{\sigma_1}}$};
	\node[above right] at (4,3.5) {\tiny $\sqrt{\frac{\lambda_0}{\kappa_2}}, \sqrt{\frac{\lambda_0}{\sigma_2}}$};
	
	\node[above right] at (2,2) {\tiny $\sqrt{\frac{\lambda_{0}}{\kappa_1}}, \sqrt{\frac{\lambda_0}{\sigma_1}}$};

	\node[above] at (1,2) {\small $\bB^{(0,1)}$};
	\node[above] at (1,3.5) {\small $\bB^{(0,2)}$};
	\node[above] at (1,5) {\small $\bB^{(0,3)}$};
	\node[above] at (1,6.5) {\small $\bB^{(0,4)}$};
	\node[below] at (8,1) {\small$\bm{0}$};
	\node[below] at (6,1) {\small$\bm{0}$};
	\node[below] at (4,1) {\small$\bm{0}$};
	\node[below] at (2,1) {\small$\bm{0}$};
}\quad
\begin{tikzpicture}[xscale=0.8, yscale=0.9, baseline={([yshift=0]current bounding box.center)}]
	\foreach\y in {2,...,5}{
		\draw[black!10!white, line width=4 pt] (1,\y*1.5-1) -- (9,\y*1.5-1);
	}
	\foreach\x in {1,...,4}{
		\draw[black!10!white,line width=4pt] (2*\x,1) -- (2*\x,7.5);
	}	
	\draw[blue, line width = 0.8pt] (1,1.95) -- (2.05,1.95) --  (2.05, 3.5) -- (4, 3.5) -- (4, 5) -- (6, 5) -- (6, 6.5) -- (8, 6.5) -- (8,7.5);
	\draw[blue, line width = 0.8pt] (1,2.05) -- (1.95, 2.05) -- (1.95, 3.5) -- (2.1, 3.5) -- (2.1, 5) -- (2.06, 5) -- (2.06, 6.45) -- (4.06,6.45) -- (4.06,7.5);
	\draw[green!80!black, line width = 0.8pt] (1, 3.43) -- (2.03, 3.43) -- (2.03, 4.96) -- (4.05, 4.96) -- (4.05, 6.5) -- (6, 6.5) -- (6, 7.5);
	\draw[green!80!black, line width = 0.8pt] (1, 3.5) -- (1.97, 3.5) -- (1.97, 5.04) -- (3.95, 5.04) -- (3.95, 6.5) -- (4, 6.5) -- (4, 7.5);
	\draw[green!80!black, line width = 0.8pt] (1, 3.57) -- (1.9, 3.57) -- (1.9, 5) -- (1.94, 5) -- (1.94, 6.5) -- (2.05, 6.5) -- (2.05, 7.5);
	\draw[green!80!black, line width = 0.8pt] (1, 5) -- (2, 5) -- (2, 6.55) -- (3.94, 6.55) -- (3.94, 7.5);
	\draw[red, line width = 0.8pt] (1, 6.5) -- (1.95, 6.5) -- (1.95, 7.5);
	
	\node[left, blue] at (1,2) {\tiny $1$};
	\node[left, green!80!black] at (1,3.5) {\tiny $2$};
	\node[left, green!80!black] at (1,5) {\tiny $2$};
	\node[left, red] at (1,6.5) {\tiny $3$};
	
	\node[black] at (1.3, 1.5) {$0$};
	\node[black] at (1.3, 2.75) {$0$};
	\node[black] at (1.3, 4.25) {$3$};
	\node[black] at (1.3, 5.75) {$4$};
	\node[black] at (1.3, 7) {$5$};
	
	\node[black] at (3, 1.5) {$0$};
	\node[black] at (3, 2.75) {$0$};
	\node[black] at (3, 4.25) {$0$};
	\node[black] at (3, 5.75) {$2$};
	\node[black] at (3, 7) {$3$};

	\node[black] at (5, 1.5) {$0$};
	\node[black] at (5, 2.75) {$0$};
	\node[black] at (5, 4.25) {$0$};
	\node[black] at (5, 5.75) {$0$};
	\node[black] at (5, 7) {$1$};

	\node[black] at (7, 1.5) {$0$};
	\node[black] at (7, 2.75) {$0$};
	\node[black] at (7, 4.25) {$0$};
	\node[black] at (7, 5.75) {$0$};
	\node[black] at (7, 7) {$0$};
	
	\node[black] at (8.5, 1.5) {$0$};
	\node[black] at (8.5, 2.75) {$0$};
	\node[black] at (8.5, 4.25) {$0$};
	\node[black] at (8.5, 5.75) {$0$};
	\node[black] at (8.5, 7) {$0$};
	
\end{tikzpicture}
\caption{\label{qHahnFigure} Left: parameters of the $q$-Hahn vertex weights. Right: a configuration of the $q$-Hahn vertex model depicted using colored paths, along with the values of the height function $h_{\geq 2}$ for this configuration. Here $\bI=(1,2,1)$.}
\end{figure}
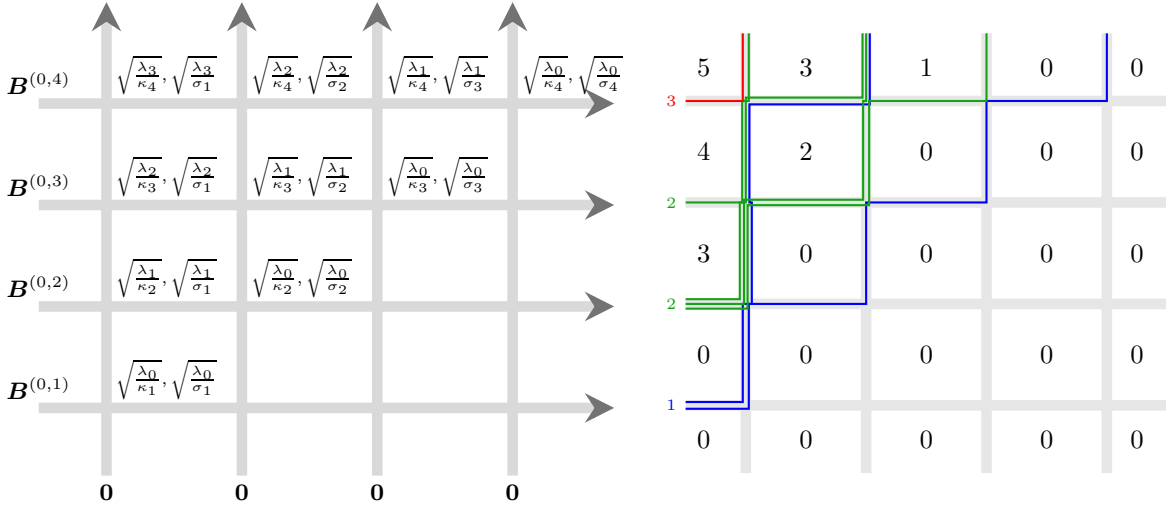

Assume that we are given $q\in (0,1)$, a composition $\bI$ such that $|\bI|=N$ and three families of real parameters $(\mu_0, \mu_1, \mu_2, \dots)$, $(\kappa_1, \kappa_2, \dots)$, $(\dots,\lambda_{-1},\lambda_0,\lambda_1,\lambda_2,\dots)$ satisfying
\begin{equation}
\label{paramqHahnRestr}
\lambda_d<\kappa_j<\mu_i,\qquad i\in\mathbb Z_{\geq 0};\  j\in\mathbb{Z}_{>0};\ d\in\mathbb Z.
\end{equation}
We construct a random configuration $\Sigma$ of the model according to the following sampling procedure:
\begin{itemize}
\item The compositions $\bA^{(i,0)}$ of all the bottom incoming vertical edges $(i,0)\to(i,1)$ are set to $\bm 0$.
\item To construct compositions of the incoming left horizontal edges we first sample a collection of independent random nonnegative integers $(b_1, b_2, \dots, b_N)$ with distributions
\begin{equation}
\label{leftProb}
\mathbb P(b_j=b)=(\kappa_j/\mu_0)^b\frac{(\lambda_j/\kappa_j;q)_b}{(q;q)_b}\frac{(\kappa_j/\mu_0;q)_\infty}{(\lambda_j/\mu_0;q)_\infty}, \qquad b\in\mathbb Z_{\geq 0}.
\end{equation}
Then the composition $\bB^{(0,j)}$ of the $j$th incoming horizontal edge $(0,j)\to (1,j)$ has the form $\bB^{(0,j)}=(0,\dots, 0, b_j, 0,\dots, 0)$, where all positions but $B_c^{(0,j)}=b_j$ are equal to $0$. The non-vanishing color $c$ is determined as the unique color satisfying $I_{[1,c-1]}<j\leq I_{[1,c]}$. In other words, the left incoming edges of the first $I_1$ rows have color $1$, the next $I_2$ rows have color $2$ and so on.
\item The compositions of the remaining edges are sampled sequentially in $N^2$ steps, which correspond to the vertices $(i,j)$ of the model and are performed in lexicographical order with respect to $(i,j)$, starting from the vertex $(1,1)$.
\item At each step we have a vertex $(i,j)$ with already sampled compositions $\bA^{(i,j-1)}$ and $\bB^{(i-1,j)}$ of the incoming edges $(i,j-1)\to (i,j)$ and $(i-1,j)\to(i,j)$, due to the order of the steps. Then we treat the vertex weights $W_{\sqrt{\lambda_{j-i}/\kappa_j},\sqrt{\lambda_{j-i}/\mu_i}}$ as probabilities for a stochastic sampling algorithm, transforming the incoming compositions into the outgoing ones $\bA^{(i,j)},\bB^{(i,j)}$:
\begin{multline}
\label{vertexSampling}
\mathbb P(\bA^{(i,j)}=\bC, \bB^{(i,j)}=\bD \mid \bA^{(i,j-1)}=\bA, \bB^{(i-1,j)}=\bB)=W_{\sqrt{\frac{\lambda_{j-i}}{\kappa_j}}, \sqrt{\frac{\lambda_{j-i}}{\mu_i}}}(\bA,\bB;\bC,\bD)\\
=\1_{\bA+\bB=\bC+\bD}\ (\kappa_j/\mu_i)^{|\bD|}\frac{(\kappa_j/\mu_i;q)_{|\bA|-|\bD|}(\lambda_{j-i}/\kappa_j;q)_{|\bD|}}{(\lambda_{j-i}/\mu_i;q)_{|\bA|}}q^{\sum_{i<j}D_i(A_j-D_j)}\prod_{i=1}^n\binom{A_i}{D_i}_q.
\end{multline}
\end{itemize}
Note that all the probabilities are nonnegative for $q\in (0,1)$ and $\mu_i,\kappa_j,\lambda_{d}$ satisfying \eqref{paramqHahnRestr}. Since the weights at vertex $(i,j)$ depend on $\sigma_i, \kappa_j$ and $\lambda_{j-i}$, we treat these parameters as attached to column $i=const$, row $j=const$ and diagonal $j-i=const$ respectively, see Figure \ref{qHahnFigure} for a depiction of the parameters of the weights in the model. 

In the description of the model we have also used that \eqref{leftProb} gives a well-defined probability distribution. This follows from the identity
\begin{equation}\label{qbinom}
\sum_{k\geq 0}x^k\frac{(y;q)_k}{(q;q)_k}=\frac{(yx;q)_\infty}{(x;q)_{\infty}},
\end{equation}
which can be proved by setting $y=q^{-n}, n\in\mathbb Z_{\geq0}$ and reducing the claim to the $q$-binomial theorem.

Finally, note that the vertex sampling is trivial outside of the region $i<j$: since the weight $W_{t,s}(\bA,\bB;\bC,\bD)$ vanishes unless $\bA\geq\bD$, a vertex $(i,j)$ with the incoming bottom configuration $\bA^{(i,j-1)}=\bm 0$ is forced to have a deterministic sampling with $\bB^{(i,j)}=\bm 0$ and $\bA^{(i,j)}=\bB^{(i-1,j)}$, that is, the top outgoing configuration is equal to the left incoming one. These rules force all vertices $(i,j)$ with $i>j$ to have edge labels $\bm 0$ around them, while the sampling at the vertices $(i,i)$ is deterministic. In particular, the model does not depend on the parameters $\{\lambda_{d}\}_{d\leq 0}$, thus we can freely omit them from now on. 

\begin{rem}\label{originRem}\normalfont
The $q$-Hahn model described in this section originates from the work \cite{BK21}, where the one-colored case of this model was used to construct \emph{inhomogeneous spin $q$-Whittaker} functions. Existence of an integrable vertex model with diagonal parameters not attached to its lines is surprising: usually integrability is caused by the Yang-Baxter equation, where parameters are attached to the lines of the model. Actually, the $q$-Hahn model without diagonal parameters was already known and it was previously obtained from the six-vertex model directly via fusion, see \cite{BP16}, \cite{BW17}.

However, the ordinary $q$-Hahn model without diagonal parameters was in a certain sense ``incomplete", which was indicated in \cite{BW17}. Namely, the $q$-Hahn model can be used to construct spin $q$-Whittaker functions, which are dual in the sense of the dual Cauchy identity to the spin Hall-Littlewood functions, constructed using higher spin six-vertex model. The latter functions have a natural inhomogeneous version, coming from a model with three families of parameters (one family is attached to rows and two families are attached to columns of the higher spin six-vertex model), but the corresponding inhomogeneous version of the spin $q$-Whittaker functions turned out to be unreachable using the ordinary $q$-Hahn model. This was remedied with the discovery of the deformed Yang-Baxter equations in \cite{BK21}, which have naturally led to the vertex model construction for the inhomogeneous spin $q$-Whittaker functions featuring the additional diagonal parameters.
\end{rem}

\subsection{The colored height functions and the formula for $q$-moments.}  The vertex model just introduced has a collection of natural observables, called \emph{colored height functions}. We denote them by $h_{\geq c}^{(x,y)}(\Sigma)$, where $\Sigma$ is a configuration of the model, $c\in[1,n]$ is an integer representing a color and $(x,y)\in\(\mathbb Z_{\geq 0}+\frac{1}{2}\)^2$ is a point treated as a facet of the grid between columns $x\pm\frac{1}{2}$ and rows $y\pm\frac{1}{2}$.

Informally, the height function $h_{\geq c}^{(x,y)}$ indicates the number of paths of color $\geq c$ passing below $(x,y)$. More precisely, for a given fixed configuration $\Sigma$ we define the corresponding height functions following two local rules:
\begin{itemize}
\item If $\bA^{(x-\frac{1}{2},y-\frac{1}{2})}$ denotes the label of the vertical edge $(x-\frac{1}{2},y-\frac{1}{2})\to(x-\frac{1}{2}, y+\frac{1}{2})$  then
\begin{equation}
\label{localrulesheight1}
h^{(x-1,y)}_{\geq c}(\Sigma)=h^{(x,y)}_{\geq c}(\Sigma) + A_{[c, n]}^{(x-\frac{1}{2},y-\frac{1}{2})}.
\end{equation}
\item If $\bB^{(x-\frac{1}{2},y+\frac{1}{2})}$ denotes the label of the horizontal edge $(x-\frac{1}{2},y+\frac{1}{2})\to(x+\frac{1}{2}, y+\frac{1}{2})$  then
\begin{equation}
\label{localrulesheight2}
h^{(x,y +1)}_{\geq c}(\Sigma)=h^{(x,y)}_{\geq c}(\Sigma) + B_{[c, n]}^{(x-\frac{1}{2},y+\frac{1}{2})}.
\end{equation}
\end{itemize}
Due to the conservation law around each vertex, these local rules give a well-defined height function $h^{(x,y)}_{\geq c}(\Sigma)$, see Figure \ref{qHahnFigure}. The resulting height functions are unique up to a global additive shift; we fix the normalization by requiring 
\be
h^{(\frac{1}{2},\frac{1}{2})}_{\geq c}(\Sigma)=0
\ee
for any color $c\in [1,n]$ and configuration $\Sigma$.

To simplify expressions in the following sections we also introduce the following notation for the multi-point height function: for a sequence of colors $\bc=(c_1, \dots, c_k)$, a pair of sequences of half integers $\bx=(x_1, \dots, x_k), \by=(y_1, \dots, y_k)$ and a configuration $\Sigma$ we set
\be
\cH_{\geq \bc}^{(\bx,\by)}(\Sigma):=\sum_{i=1}^k h_{\geq c_{i}}^{(x_i,y_i)}(\Sigma).
\ee

\subsection{The integral expression for $q$-moments.} The main result of this section expresses a certain observable of the $q$-Hahn model as a nested contour integral. The observable in question is denoted by $\Q_{\geq \tau.\bc}^{(\bx,\by)}(\Sigma)$ and for a configuration of the model $\Sigma$, an ordered sequence of colors $\bc=(c_1, \dots, c_k)$, a permutation $\tau\in S_k$ and a pair of sequences of half integers $\bx=(x_1, \dots, x_k), \by=(y_1, \dots, y_k)$ it is defined by
\be
\Q_{\geq \tau.\bc}^{(\bx,\by)}(\Sigma):=q^{\cH_{\geq \tau.\bc}^{(\bx,\by)}(\Sigma)}=\prod_{a=1}^kq^{h^{(x_a,y_a)}_{\geq c_{\tau^{-1}(a)}}(\Sigma)},
\ee
where we follow the notation
\be
\tau.\bc=\(c_{\tau^{-1}(1)}, c_{\tau^{-1}(2)}, \dots, c_{\tau^{-1}(k)}\).
\ee

To formulate the result we also need to specify integration contours. They are denoted by $\Gamma_1, \Gamma_2, \dots, \Gamma_k$ and are assumed to satisfy the following conditions:
\begin{itemize}
\item Every contour $\Gamma_a$ is a union of simple positively oriented closed contours in $\mathbb C$;
\item The points from $\{\mu^{-1}_i\}_{i\geq 0}$ and $\{\kappa^{-1}_j\}_{j\geq 1}$ are inside every contour $\Gamma_a$, while $0$ and the points from $\{\lambda^{-1}_d\}_{d\geq 1}$ are outside of the contours;
\item For any $a<b$ the contour $\Gamma_b$ encircles both contours $\Gamma_a$ and $q\Gamma_a$. Here $q\Gamma_a$ denotes the contour obtained by multiplying all points of $\Gamma_a$ by $q$.
\end{itemize}
See Figure \ref{GammaFigure} for a possible configuration of the contours. Note that the conditions on the parameters \eqref{paramqHahnRestr} guarantee the existence of such contours.

\begin{figure}
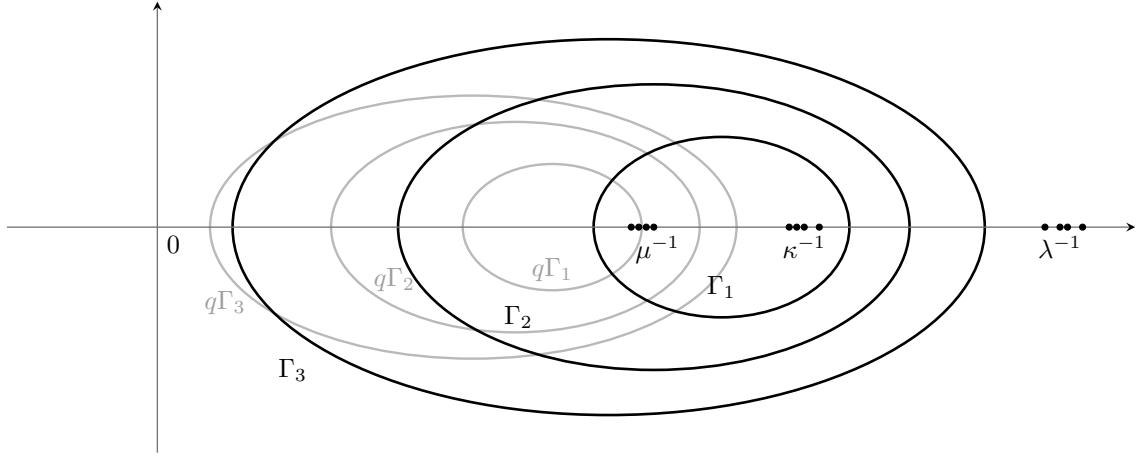

\tikz{1}{

\draw[line width=0.9pt, draw=dgray!50] (5.25,0) ellipse (1.19cm and 0.84cm);
\draw[line width=0.9pt, draw=dgray!50] (4.76,0) ellipse (2.45cm and 1.4cm);
\draw[line width=0.9pt, draw=dgray!50] (4.2,0) ellipse (3.5cm and 1.75cm);

\draw[line width=1pt, draw=black] (7.5,0) ellipse (1.7cm and 1.2cm);
\draw[line width=1pt, draw=black] (6.6,0) ellipse (3.4cm and 1.9cm);
\draw[line width=1pt, draw=black] (6,0) ellipse (5cm and 2.5cm);

\node[] at (7.5,-0.8) {$\Gamma_1$};
\node[] at (4.8,-1.2) {$\Gamma_2$};
\node[] at (1.8,-1.9) {$\Gamma_3$};

\node[text=dgray!70] at (5.25,-0.56) {$q\Gamma_1$};
\node[text=dgray!70] at (3.15,-0.7) {$q\Gamma_2$};
\node[text=dgray!70] at (0.9,-1) {$q\Gamma_3$};

\draw[line width=0.5pt, draw=dgray, arrows={->[scale=4]}] (-2,0) -- (13,0);
\draw[line width=0.5pt, draw=dgray, arrows={->[scale=4]}] (0,-3) -- (0,3);
\foreach \point in {(8.6,0), (8.5,0), (8.4,0), (8.8, 0)}
\draw[fill=black] \point circle (0.04);
\node[below] at (8.6,0) {$\kappa^{-1}$};
\foreach \point in {(12,0), (12.1,0), (12.3,0), (11.8, 0)}
\draw[fill=black] \point circle (0.04);
\node[below] at (12,0) {$\lambda^{-1}$};
\foreach \point in {(6.5,0), (6.6,0), (6.4,0), (6.3, 0)}
\draw[fill=black] \point circle (0.04);
\node[below] at (6.65,0) {$\mu^{-1}$};

\node[below right] at (0,0) {$0$};

}
\caption{\label{GammaFigure}A possible of configuration of the contours $\Gamma_i$}
\end{figure}

\begin{theo}\label{qHahnResultTheo} Let $\Sigma$ be a random configuration of the $q$-Hahn model with the weights defined by parameters $\{\mu_i\}_{i\geq 0}$, $\{\kappa_j\}_{j\geq 1}$, $\{\lambda_d\}_{d\geq 1}$ and the left boundary condition given by a composition $\bI$ as in Section \ref{themodel}. Then for any sequences $\bc=(c_1, \dots, c_k)$, $\bx=(x_1, \dots, x_k)$ and $\by=(y_1, \dots, y_k)$ satisfying
\be
x_1 \leq x_2\leq\dots\leq x_k,\qquad y_1\geq y_2\geq\dots\geq y_k, \qquad x_a\leq y_a, \qquad x_a,y_a\in\mathbb Z_{\geq 0}+\frac{1}{2};
\ee
\be
c_1\leq c_2\leq\dots\leq c_k,\qquad c_a\in\mathbb Z_{\geq 0};
\ee
and for any permutation $\tau\in S_k$ the following holds:
\begin{multline}
\label{qHahnIntExpression}
\E\left[\Q^{\bx,\by}_{\geq\tau.\bc}(\Sigma)\right]=\frac{(-1)^kq^{\frac{k(k-1)}{2}-l(\tau)}}{(2\pi\i)^k}\oint_{\Gamma_1}\cdots\oint_{\Gamma_k}\prod_{a<b}\frac{w_b-w_a}{w_b-qw_a}\ T_\tau\(\prod_{a=1}^k\prod_{i=1}^{l_a}\frac{1-\lambda_iw_a}{1-\kappa_iw_a}\)\\
\times \prod_{a=1}^k\(\prod_{i=0}^{i<x_a}\frac{1}{1-\mu_iw_a}\prod_{j=1}^{j<y_a}(1-\kappa_jw_a)\prod_{d=1}^{d\leq y_a-x_a}\frac{1}{1-\lambda_dw_a}\)\frac{dw_a}{w_a},
\end{multline}
where the integration contours $\Gamma_a$ are specified before the theorem, $T_\tau$ denotes the action of the Hecke algebra defined by \eqref{Tdef} and the integers $l_a$ are defined by $l_a=I_{[1,c_a-1]}$.
\end{theo}

The proof of the theorem occupies the remainder of the section, but before it we explain how Theorem \ref{qHahnResultTheo} simplifies in the one-colored case, when $n=1$. Then all colors in $\bc$ are equal to $1$, the colors of the incoming left boundary edges are defined by $\bI=(N)$ and the operators $T_\tau$ act on a constant $1$, multiplying it by $q^{l(\tau)}$. Since the only nontrivial height function in the color-blind case is $h_{\geq1}$, we omit $\bc$ and write just $\Q^{\A,\B}(\Sigma)$ instead of $\Q^{\A,\B}_{\geq(1,\dots, 1)}(\Sigma)$ in the one-colored case:
\begin{cor}
Let $\Sigma$ be a random configuration sampled from the one-colored $q$-Hahn model with parameters $\{\mu_i\}_{i\geq 0}, \{\kappa_j\}_{j\geq 1}, \{\lambda_d\}_{d\geq 1}$. Then for any sequences $\bx=(x_1, \dots, x_k)$ and $\by=(y_1, \dots, y_k)$ satisfying
\be
x_1 \leq x_2\leq\dots\leq x_k,\qquad y_1\geq y_2\geq\dots\geq y_k, \qquad x_a\leq y_a, \qquad x_a, y_b\in\mathbb Z_{\geq 0}+\frac{1}{2},
\ee
 the following holds:
\begin{multline}
\label{oneColorqHahnIntExpression}
\E\left[\Q^{\bx,\by}(\Sigma)\right]=\frac{(-1)^kq^{\frac{k(k-1)}{2}}}{(2\pi\i)^k}\oint_{\Gamma_1}\cdots\oint_{\Gamma_k}\prod_{a<b}\frac{w_b-w_a}{w_b-qw_a}\\
 \times \prod_{a=1}^k\(\prod_{i=0}^{i<x_a}\frac{1}{1-\mu_iw_a}\prod_{j=1}^{j<y_a}(1-\kappa_jw_a)\prod_{d=1}^{d\leq y_a-x_a}\frac{1}{1-\lambda_dw_a}\)\frac{dw_a}{w_a},
\end{multline}
where the integration contours $\Gamma_a$ are simple positively oriented curves satisfying
\begin{itemize}
\item The points from $\{\mu^{-1}_i\}_{i\geq 0}$ are inside every contour $\Gamma_a$, while $0$ and the points from $\{\lambda^{-1}_d\}_{d\geq 1}$ are outside of the contours;
\item For any $a<b$ the contour $\Gamma_b$ encircles both contours $\Gamma_a$ and $q\Gamma_a$. 
\end{itemize}
\end{cor}
Note that, contrary to the colored case, in the one-colored situation the integrand has no singularities at $\kappa_j^{-1}$, so we need no restrictions on the positions of $\kappa^{-1}_j$ with respect to the contours $\Gamma_a$.

\begin{rem}\normalfont
When all parameters $\{\lambda_d\}_{d\geq1}$ are equal and there is no inhomogeneity assigned to diagonals, the $q$-Hahn model can be obtained from the \emph{colored higher-spin six-vertex model} using \emph{stochastic fusion}. In this case Theorem \ref{qHahnResultTheo} can be obtained from analogous integral expressions for the $q$-moments of the higher-spin model: for one-color case this was done in \cite{BP16}, while the colored situation is covered by \cite{BW20}, \cite{BK20}. But it seems that the general case with arbitrary parameters $\{\lambda_d\}_{d\geq1}$ cannot be reached with fusion.
\end{rem}

\subsection{Explicit integral computations.} The proof of Theorem \ref{qHahnResultTheo} consists of two ingredients: explicit computations of the contour integrals in certain degenerate cases, and recurrence relations given in Section \ref{localSect}. Here we focus on the former.

We fix the contours $\Gamma_a$ as in the theorem, and let $\bw=(w_1, \dots, w_k)$ denote the collection of the integration variables. We need the following fact describing the operators adjoint to $T_\tau$ with respect to the scalar product
\be
\langle F(\bw), G(\bw)\rangle_k:=\oint_{\Gamma_1}\cdots\oint_{\Gamma_k}\prod_{a<b}\frac{w_b-w_a}{w_b-qw_a} F(\bw)G(\bw)\prod_{a=1}^k\frac{dw_a}{w_a}.
\ee
\begin{prop}[\cite{BW18}]\label{selfAdjProp}
Assume that $F(w_1, \dots, w_k)=F(\bw)$ and $G(w_1, \dots, w_k)=G(\bw)$ are rational functions having only singularities of the form $w_a=\mu_i^{-1}$, $w_a=\kappa_j^{-1}$ or $w_a=\lambda_d^{-1}$. Then for any permutation $\tau$ the following relation holds:
\begin{equation}
\label{selfAdj}
\oint_{\Gamma_1}\cdots\oint_{\Gamma_k}\prod_{a<b}\frac{w_b-w_a}{w_b-qw_a} T_\tau(F(\bw))G(\bw)\prod_{a=1}^k\frac{dw_a}{w_a}=\oint_{\Gamma_1}\cdots\oint_{\Gamma_k}\prod_{a<b}\frac{w_b-w_a}{w_b-qw_a} F(\bw)T_{\tau^{-1}}(G(\bw))\prod_{a=1}^k\frac{dw_a}{w_a},
\end{equation}
or, equivalently,
\be
\langle T_\tau F(\bw), G(\bw)\rangle_k=\langle F(\bw), T_{\tau^{-1}}G(\bw)\rangle_k.
\ee
\end{prop}
\begin{proof}
The complete proof is given in \cite[Proposition 8.1.3]{BW18}, see also \cite[Proposition 4.1]{BK20}. The main idea is to prove that the generators $T_i$ are self-adjoint, which is enough since $T_i$ generate the Hecke algebra $\mathcal H_k$. Each $T_i$ is a sum of a multiplication by a constant and $\frac{w_{i+1}-qw_i}{w_{i+1}-w_i}\mathfrak s_i$. The former is clearly self-adjoint, while for the latter note that the multiplication by $\frac{w_{i+1}-qw_i}{w_{i+1}-w_i}$ removes the singularity of the integrand at $w_{i+1}=qw_i$, allowing to exchange the contours $\Gamma_i$ and $\Gamma_{i+1}$ without changing the integral. Performing a change of variables swapping $w_i$ and $w_{i+1}$, one can see that  $\frac{w_{i+1}-qw_i}{w_{i+1}-w_i}\mathfrak s_i$ is also self-adjoint.
\end{proof}

The next statement describes the situation when one of the integrals can be easily computed, reducing the number of integration variables by one. In the context of Theorem \ref{qHahnResultTheo} we will use this computation when $x_k=y_k$ and the integral has no residue at $w_k=\lambda_d^{-1}$ for any $d$, see Step 2 in Section \ref{mainproofsubsection}.

\begin{prop} \label{integralReduction}Assume that $f_1, \dots, f_k$ and $g_1, \dots, g_k$ are rational functions in one variable satisfying 
\begin{itemize}
\item For every $a\in[1,k]$ the function $f_k(w)g_a(w)$ is holomorphic outside of $\Gamma_k$;
\item For every $a\in[1,k]$ the poles of $f_a(w)$ and $g_a(w)$ are inside of $\{\mu_i^{-1}\}_{i\geq 0}\cup\{\kappa_j^{-1}\}_{j\geq 1}\cup\{\lambda_d^{-1}\}_{d\geq 1}$;
\item For every $a\in [1,k]$ we have $f_a(0)=g_a(0)=1$.
\end{itemize}
Then
\begin{multline*}
\frac{(-1)^kq^{\frac{k(k-1)}{2}-l(\tau)}}{(2\pi\i)^k}\oint_{\Gamma_1}\cdots\oint_{\Gamma_k}\prod_{a<b}\frac{w_b-w_a}{w_b-qw_a}T_{\tau}\(\prod_{a=1}^kf_a(w_a)\)\ \prod_{a=1}^kg_a(w_a)\prod_{a=1}^k\frac{dw_a}{w_a}\\
=\frac{(-1)^{k-1}q^{\frac{(k-1)(k-2)}{2}-l(\tau')}}{(2{\pi}\i)^{k-1}}\oint_{\Gamma_1}\cdots\oint_{\Gamma_{k-1}}\prod_{a<b\leq k-1}\frac{w_b-w_a}{w_b-qw_a}T_{\tau'}\(\prod_{a=1}^{k-1}f_a(w_a)\)\ \prod_{a=1}^{k-1}g_{\rho(a)}(w_a)\prod_{a=1}^{k-1}\frac{dw_a}{w_a}
\end{multline*}
where $\rho=\sigma_{k-1}\sigma_{k-2}\dots\sigma_{\tau(k)}$ is the cycle sending $\tau(k), \tau(k)+1, \dots, k$ to $\tau(k)+1, \tau(k)+2, \dots,  k, \tau(k)$ respectively, while leaving other elements intact, and the permutation $\tau'$ is defined by $\tau'=\rho^{-1}\tau$.
\end{prop}
\begin{proof}
The proof is similar to \cite[Lemma 4.2.1]{BK20}. Note that $T_{\tau}=T_{\rho}T_{\tau'}$, so by Proposition \ref{selfAdjProp} the left-hand side can be rewritten as
\begin{equation}
\label{lemmaProof}
\frac{(-1)^kq^{\frac{k(k-1)}{2}-l(\tau)}}{(2\pi\i)^k}\oint_{\Gamma_1}\cdots\oint_{\Gamma_k}\prod_{a<b}\frac{w_b-w_a}{w_b-qw_a}T_{\tau'}\(\prod_{a=1}^kf_a(w_a)\)\ T_{\rho^{-1}}\(\prod_{a=1}^kg_a(w_a)\)\prod_{a=1}^k\frac{dw_a}{w_a}.
\end{equation}
Now we are taking the integral with respect to $w_k$ by computing its residues outside of $\Gamma_k$. Note that $\tau'(k)=k$, hence
\be
T_{\tau'}\(\prod_{a=1}^kf_a(w_a)\)=f_k(w_k)T_{\tau'}\(\prod_{a=1}^{k-1}f_a(w_a)\).
\ee
By the conditions of the claim the function
\be
f_k(w_k) T_{\rho^{-1}}\(\prod_{a=1}^kg_a(w_a)\)=\(\prod_{a=1}^{k-1}f_k(w_a)\)^{-1} T_{\rho^{-1}}\(\prod_{a=1}^kf_k(w_a)g_a(w_a)\)
\ee
is holomorphic with respect to $w_k$ outside of $\Gamma_k$, so the only non-vanishing residue of the integrand in \eqref{lemmaProof} outside of $\Gamma_k$ is the residue at $w_k=0$. Computing it, we obtain
\be
\frac{(-1)^{k-1}q^{\frac{(k-1)(k-2)}{2}-l(\tau)}}{(2\pi\i)^{k-1}}\oint_{\Gamma_1}\cdots\oint_{\Gamma_{k-1}}\prod_{a<b\leq k-1}\frac{w_b-w_a}{w_b-qw_a}T_{\tau'}\(\prod_{a=1}^{k-1}f_a(w_a)\)\ \restr{ \( T_{\rho^{-1}}\prod_{a=1}^kg_a(w_a)\)}{w_k=0} \prod_{a=1}^{k-1}\frac{dw_a}{w_a}.
\ee
To finish the proof, note that for any function $G(\bw)$ nonsingular at $w_{a}=0$ for every $a$ we have
\be
\restr{\(T_i G(\bw)\)}{w_{i+1}=0}=\restr{\(\frac{(q-1)w_{i+1}}{w_{i+1}-w_i}G(\bw)+\frac{w_{i+1}-qw_i}{w_{i+1}-w_i}G(\sigma_i. \bw)\)}{w_{i+1}=0}=q\restr{G(\sigma_i.\bw)}{w_{i+1}=0},
\ee
where $\sigma_i$ is a simple transposition acting on the variables $\bw$ by permutation. A straightforward induction yields 
\begin{multline*}
\restr{T_{\rho^{-1}}\(\prod_{a=1}^kg_a(w_a)\)}{w_k=0}=\restr{T_{k-1}T_{k-2}\dots T_{\tau(k)}\(\prod_{a=1}^kg_a(w_a)\)}{w_k=0}\\
=\restr{q^{l(\rho)}\prod_{a=1}^kg_a(w_{\rho^{-1}(a)})}{w_k=0}=q^{l(\rho)}\prod_{a=1}^{k-1}g_{\rho(a)}(w_{a}).
\end{multline*}
\end{proof}

Another integral we need to explicitly compute by taking residues is the right-hand side of \eqref{qHahnIntExpression} when $x_1=\dots=x_k=\frac{1}{2}$. For convenience, set $p_i:=y_i-\frac{1}{2}$ and recall that $l_a$ corresponds to the position of the color $c_a$ on the left boundary.

\begin{prop} \label{initialCondition} Fix $N>0$. For any integer sequences $l_1, l_2, \dots, l_k$ and $p_1, p_2, \dots, p_k$ such that
\be
0\leq l_1\leq l_2\leq\dots\leq l_k\leq N,\qquad N\geq p_1\geq p_2\geq \dots\geq p_k\geq 0,
\ee
and any permutation $\tau$ we have
\begin{multline}
\label{initialConditionEq}
\prod_{j=1}^N\frac{(\kappa_j/\mu_0;q)_{r_j}}{(\lambda_j/\mu_0;q)_{r_j}}=\frac{(-1)^kq^{\frac{k(k-1)}{2}-l(\tau)}}{(2\pi\i)^k}\oint_{\Gamma_1}\dots\oint_{\Gamma_k}\prod_{a<b}\frac{w_b-w_a}{w_b-qw_a}\ T_\tau\(\prod_{a=1}^k\prod_{j=1}^{l_a}\frac{1-\lambda_jw_a}{1-\kappa_jw_a}\)\\
\times \prod_{a=1}^k\prod_{j=1}^{p_a}\frac{1-\kappa_jw_a}{1-\lambda_jw_a}\ \prod_{a=1}^k\frac{dw_a}{w_a(1-\mu_0 w_a)},
\end{multline}
where for every integer $j\in [1, N]$ we set
\be
r_j:=\#\{a\in[1,k] \mid l_{\tau^{-1}(a)}<j\leq p_a\},
\ee
that is, $r_j$ is the number of intervals $[l_{\tau^{-1}(a)}+1, p_a]$ containing $j$. 
\end{prop}
\begin{proof}
The proof is by induction on $k$ and $N$, where at each step we either reduce $k$ or $N$ by $1$, keeping $N,k\geq 0$. The base case $k=0$ is trivial.

Assume that $k> 0$ and let $\mathbf l=(l_1, \dots, l_k)$ and $\mathbf p=(p_1, \dots, p_k)$. Set 
\be
F^{\mathbf l}(\bw)=\prod_{a=1}^k\prod_{j=1}^{l_a}\frac{1-\lambda_jw_a}{1-\kappa_jw_a},\qquad G^{\mathbf p}(\bw)=\prod_{a=1}^k\prod_{j=1}^{p_a}\frac{1-\kappa_jw_a}{1-\lambda_jw_a}.
\ee
We have three cases:

\emph{Case 1: $l_k<N$ and $p_1<N$:} Then for any $a$ we have $l_a, p_a<N$, so $r_N=0$, $N>0$ and we can freely reduce $N$ by $1$ without changing the claim.

\emph{Case 2: $l_k=N$:} Here we can compute the integral with respect to $w_k$ by applying Proposition \ref{integralReduction} for the functions
\be
f_a=\prod_{j=1}^{l_a}\frac{1-\lambda_jw}{1-\kappa_jw}, \qquad g_a=\frac{1}{1-\mu_0w}\prod_{j=1}^{p_a}\frac{1-\kappa_jw}{1-\lambda_jw}.
\ee
Since $l_k=N\geq p_a$ for any $a$, after the multiplication by $f_k(w)$ the poles of $g_a(w)$ at $w=\lambda_j^{-1}$ vanish, so the function $f_k(w)g_a(w)$  has no nonzero residues outside of $\Gamma_k$ and Proposition \ref{integralReduction} implies that
\begin{multline*}
\frac{(-1)^kq^{k(k-1)/2-l(\tau)}}{(2\pi\i)^k}\oint_{\Gamma_1}\cdots\oint_{\Gamma_k}\prod_{a<b}\frac{w_b-w_a}{w_b-qw_a}T_\tau\(F^{\mathbf l}(\bw)\)G^{\mathbf p}(\bw)\prod_{a=1}^k\frac{dw_a}{w_a(1-\mu_0w_a)}\\
=\frac{(-1)^{k-1}q^{(k-1)(k-2)/2-l(\tau')}}{(2\pi\i)^{k-1}}\oint_{\Gamma_1}\cdots\oint_{\Gamma_{k-1}}\prod_{a<b\leq k-1}\frac{w_b-w_a}{w_b-qw_a}T_{\tau'}\(F^{{\mathbf l}'}(\bw')\)G^{{\mathbf p}'}({\bw}')\prod_{a=1}^{k-1}\frac{dw_a}{w_a(1-\mu_0w_a)},
\end{multline*}
where  the permutation ${\tau}'$ as defined by $\tau=\rho\tau'$ for the same cycle $\rho$ as in Proposition \ref{integralReduction} and
\be
{\mathbf w}'=(w_1, \dots, w_{k-1}),\qquad {\mathbf l}'=(l_1, \dots, l_{k-1}),
\ee
\be
{\mathbf p}'=(p_{\rho(1)}, \dots, p_{\rho(k-1)})=(p_{1},\dots, p_{\tau(k)-1}, p_{\tau(k)+1}, \dots, p_{k}).
\ee
Hence, the right-hand sides of \eqref{initialConditionEq} for the triples $(\mathbf l, \mathbf p, \tau)$ and $({\mathbf l}', {\mathbf p}', {\tau}')$ coincide. 

At the same time, the families of intervals $\{[l_{\tau^{-1}(a)}+1, p_a]\}_{a\in[1,k]}$ and $\{[l_{{\tau}'^{-1}(a)}+1, {p}_{\rho(a)}]\}_{a\in[1,k-1]}$ differ only by the empty interval $[l_k+1, p_{\tau(k)}]$, implying that the left-hand sides of the claim for $({{\mathbf l}},{\mathbf p},{\tau})$ and $({{\mathbf l}}', {\mathbf p}', {\tau}')$ coincide as well. Thus, we can reduce $k$ by $1$ and the induction step hold.

\emph{Case 3: $l_k<N$ and $N=p_1=p_2=\dots=p_r>p_{r-1}$:} To tackle this case we start computing the integrals with respect to $w_1, \dots, w_r$, taking residues inside $\Gamma_1, \dots, \Gamma_r$. Note that possible simple poles at $w_1=\kappa^{-1}_j, \dots, w_r=\kappa^{-1}_j$ coming from $F^{\mathbf l}(\bw)$ are cancelled by $\prod_{a=1}^r\prod_{j=1}^N(1-\kappa_jw_a)$ from $G^{\mathbf p}(\bw)$, so all non-vanishing residues inside $\Gamma_1, \dots, \Gamma_r$ come from the singularities of 
\be
\prod_{a=1}^k\frac{1}{w_a(1-\mu_0w_a)}\prod_{a<b}\frac{w_b-w_a}{w_b-qw_a}.
\ee
Taking into account the geometry of the contours, the integral with respect to $w_1$ has only one residue inside $\Gamma_1$, which is at $w_1=\mu_0^{-1}$. Computing this residue for the expression above, we obtain
\be
\restr{\frac{1}{-\mu_0w_1}\prod_{a=2}^{k}\frac{1}{w_a(1-\mu_0w_a)}\prod_{a<b}\frac{w_b-w_a}{w_b-qw_a}}{w_1=\mu_0^{-1}}=-q^{-k+1}\prod_{a=2}^{k}\frac{1}{w_a(1-q^{-1}\mu_0w_a)}\prod_{2\leq a<b}\frac{w_b-w_a}{w_b-qw_a}.
\ee
Hence, after taking the residue at $w_1=\mu_0^{-1}$, the integral with respect to $w_2$ again has only one non-vanishing residue inside $\Gamma_2$, this time at $w_2=q\mu^{-1}_0$. Continuing this argument, we see that the only non-vanishing contribution in the computation of the right-hand side of \eqref{initialConditionEq} comes from the residue at $w_1=\mu_0^{-1}, w_2=q\mu_0^{-1}, \dots, w_r=q^{r-1}\mu_0^{-1}$. Note that 
\be
\restr{G^{\mathbf p}(\bw)}{\substack{w_a=q^{a-1}\mu_0^{-1}\\ a\in [1,r]}}=\frac{(\kappa_N/\mu_0;q)_r}{(\lambda_N/\mu_0;q)_r}\restr{G^{{\mathbf p}'}(\bw)}{\substack{w_a=q^{a-1}\mu_0^{-1}\\ a\in [1,r]}},
\ee
where ${\mathbf p}'=(p_1-1, p_2-1, \dots, p_r-1, p_{r+1}, \dots, p_k)$. Hence
\begin{multline*}
\frac{(-1)^kq^{k(k-1)/2-l(\tau)}}{(2\pi\i)^k}\oint_{\Gamma_1}\cdots\oint_{\Gamma_k}\prod_{a<b}\frac{w_b-w_a}{w_b-qw_a}T_\tau\(F^{\mathbf l}(\bw)\)G^{\mathbf p}(\bw)\prod_{a=1}^k\frac{dw_a}{w_a(1-\mu_0w_a)}\\
=\frac{(\kappa_N/\mu_0;q)_r}{(\lambda_N/\mu_0;q)_r}\frac{(-1)^kq^{k(k-1)/2-l(\tau)}}{(2\pi\i)^k}\oint_{\Gamma_1}\cdots\oint_{\Gamma_k}\prod_{a<b}\frac{w_b-w_a}{w_b-qw_a}T_\tau\(F^{\mathbf l}(\bw)\)G^{{\mathbf p}'}(\bw)\prod_{a=1}^k\frac{dw_a}{w_a(1-\mu_0w_a)},
\end{multline*}
At the same time, note that $r_N=r$ for the triple of data $({\mathbf l}, \mathbf p, \tau)$, while the numbers $r_i$ for $i<N$ are the same for both triples $({\mathbf l}, \mathbf p, \tau)$ and $(\mathbf l, \mathbf p', \tau)$. So, factoring out the term $\frac{(\kappa_N/\mu_0;q)_r}{(\lambda_N/\mu_0;q)_r}$, we reduce the claim for  $({\mathbf l}, \mathbf p, \tau)$ to the claim for $({\mathbf l}, \mathbf p', \tau)$, and the latter one allows to freely reduce $N$ by $1$, finishing the step of induction. 
\end{proof}

\subsection{Proof of Theorem \ref{qHahnResultTheo}.}\label{mainproofsubsection} The main idea of the proof is to use induction, applying the local relations from Section \ref{localSect} to reduce the coordinates $\bx$ and $\by$. This process has to eventually reach either the case $x_k=y_k$, when we can use Proposition \ref{integralReduction} to reduce $k$,  or the case $x_1=x_2=\dots=x_{k}=\frac{1}{2}$, in which the claim holds by Proposition \ref{initialCondition}. Throughout the whole argument the parameters of the model, the contours $\Gamma_a$ and the sequence $\bc$ are fixed.

\paragraph*{{\bfseries Step 1.\ }}\ Fix $\bx,\by,\tau$ as in the claim and assume that $\frac{1}{2}<x_k<y_k$. In this main step we use the local relations to show that the claim for $(\bx,\by,\tau)$ follows from the claims for $(\bx',\by',\tau')$ with smaller coordinates $x'_a\leq x_a, y'_b\leq y_b$, $\by'\neq\by$.  

More precisely, we start by finding integers $m\in\mathbb Z_{\geq 0}, r\in\mathbb Z_{>0}$ such that
\be
x_m<x_{m+1}=\dots=x_{m+r}\leq x_{m+r+1},\qquad y_m\geq y_{m+1}=\dots=y_{m+r}> y_{m+r+1},
\ee
\be
y_{m+1}>x_{m+1}>\frac{1}{2},
\ee
where $m+r\leq k$ and if $m=0$ or $m+r=k$ we omit inequalities involving $x_0,y_0$ or $x_{k+1}, y_{k+1}$ respectively. In other words, $m$ is chosen in a way that all points $(x_a,y_a)$ are either equal to the point $(x_{m+1},y_{m+1})$, strictly below it or strictly to the left of it, while $r$ is the number of the points equal to $(x_{m+1},y_{m+1})$. To show existence of such $m,r$ one can take $m$ equal to the maximal index such that $x_m<x_{k}$ (if $x_1=x_k$ take $m=0$) while $r$ is the maximal index such that $y_{m+r}=y_{m+1}$. For any $a\leq b$ let $\e^{[a,b]}\in\mathbb Z^k$ denote the vector with $1$ entries at positions $a, a+1, \dots, b$, and $0$s elsewhere:
\be
\e^{[a,b]}=\sum_{i=a}^b \e^i=(\dots, 0, 0, 1, 1, \dots, 1, 0, 0 \dots).
\ee
It turns out that the claim for $\bx,\by$ and arbitrary $\tau$ follows from the claims for $\bx-\e^{[m+1,m+p]}, \by-\e^{[m+1,m+r]}$ with $p=0, \dots, r$. To show it we proceed in three sub-steps: first we simplify the permutation $\tau$ and then apply the local relation to both sides of \eqref{qHahnIntExpression}.

\paragraph*{{\bfseries Step 1.1.}}\ For now we want to show that it is enough to consider permutations $\tau$ which are $[m+1, m+r]$-ordered:
\be
\tau^{-1}(m+1)<\tau^{-1}(m+2)<\dots<\tau^{-1}(m+r).
\ee
For the left-hand side of \eqref{qHahnIntExpression} note that for any permutation $\rho\in S_{[m+1,m+r]}$ permuting only numbers $m+1, \dots, m+r$ we have
\be
\Q^{\bx,\by}_{\geq\tau.\bc}=\Q^{\rho^{-1}.\bx, \rho^{-1}.\by}_{\geq\tau.\bc}=\Q^{\bx, \by}_{\geq\rho\tau.\bc},
\ee
thus the left-hand side depends only on the right coset of $\tau$ in $S_{[m+1,m+r]}\backslash S_{k}$. On the other hand, we can rewrite the right-hand side of \eqref{qHahnIntExpression} as 
\be
\frac{(-1)^kq^{\frac{k(k-1)}{2}-l(\tau)}}{(2\pi\i)^k}\oint_{\Gamma_1}\cdots\oint_{\Gamma_k}\prod_{a<b}\frac{w_b-w_a}{w_b-qw_a}\ G(\bw)\ T_{\tau^{-1}}\( F^{\bx,\by}(\bw)\) \frac{dw_a}{w_a},
\ee
where
\begin{equation}
\label{GFnotationTheoProof}
G(\bw)=\prod_{a=1}^k\prod_{i=1}^{l_a}\frac{1-\lambda_iw_a}{1-\kappa_iw_a},
\end{equation}
\be
F^{\bx,\by}(\bw)=\prod_{a=1}^k\(\prod_{i=0}^{i<x_a}\frac{1}{1-\mu_iw_a}\prod_{j=1}^{j<y_a}(1-\kappa_jw_a)\prod_{d=1}^{d\leq y_a-x_a}\frac{1}{1-\lambda_dw_a}\).
\ee
Note that $F^{\bx,\by}(\bw)$ is symmetric in $w_{m+1}, w_{m+2},\dots, w_{m+r}$, hence for any $\rho\in S_{[m+1, m+r]}$ we have
\begin{equation}
\label{tmpFident}
q^{-l(\rho)}T_{\rho^{-1}} F^{\bx,\by}(\bw)=F^{\bx,\by}(\bw),
\end{equation}
since $T_i f(w_i, w_{i+1})=qf(w_i, w_{i+1})$ for any $f$ symmetric in $w_{i},w_{i+1}$. Multiplying \eqref{tmpFident} on the left by $T_{\tau_{min}}^{-1}$ for a minimal permutation $\tau_{min}$ of a coset in $S_{[m+1,m+r]}\backslash S_{k}$, one can readily see that the right-hand side of \eqref{qHahnIntExpression} also depends only on the left coset of $\tau$, rather than $\tau$ itself. So, we can assume that $\tau$ is a minimal representative of a left coset in  $S_{[m+1,m+r]}\backslash S_{k}$, or, equivalently, $\tau$ is $[m+1, m+r]$-ordered.

\paragraph*{{\bfseries Step 1.2.}}\ Now, assuming that $\tau$ is $[m+1, m+r]$-ordered, we want to apply an appropriate local relation to the left-hand side of \eqref{qHahnIntExpression}. Let $v=(x_{m+1}-\frac{1}{2},y_{m+1}-\frac{1}{2})$ denote the vertex immediately to the south-west of the facet $(x,y):=(x_{m+1}, y_{m+1})=\dots=(x_{m+r}, y_{m+r})$. It has vertex weights $W_{\sqrt{\lambda/\kappa},\sqrt{\lambda/\mu}}$, where 
\be
\mu:=\mu_{x-\frac{1}{2}},\qquad \kappa:=\kappa_{y-\frac{1}{2}},\qquad \lambda:=\lambda_{y-x}.
\ee  
Now we want to change the order in which we perform the sampling of the model: first we sample the outgoing configurations for the vertices $(i,j)$ either below $v$ or to the left of $v$, obtaining a configuration $\Sigma_{\swarrow v}$, and then we sample the remaining configuration $\Sigma$. Since all our samplings are independent, their order does not matter as long as each vertex is sampled after the vertices directly below and to the left of it, so the resulting configuration is distributed in the same way as in the initial model. But note that $\Sigma_{\swarrow v}$ determines the value of the height functions at the facets $(x_{a}, y_{a})$ for $a\leq m$, because they are to the left of $v$, and at the facets $(x_{a}, y_{a})$ for $m+r< a$ because they are below $v$. 

At the same time, the height functions at $(x,y)$ are determined only by $\Sigma_{\swarrow v}$ and the labels of the edges around $v$, which are determined by the sampling at $v$. Let $(\bA,\bB)$ denote the compositions of the incoming bottom and left edges of $v$, determined by $\Sigma_{\swarrow_v}$, and let $(\bC, \bD)$ denote the compositions of the outgoing top and right edges of $v$. Then the local rule \eqref{localrulesheight2} from the definition of the height functions implies that
\begin{equation}
\label{useLocalHeight}
\sum_{a=m+1}^{m+r}h^{(x,y)}_{\geq c_{\tau^{-1}(a)}}(\Sigma)=\sum_{a=m+1}^{m+r}h^{(x,y-1)}_{\geq c_{\tau^{-1}(a)}}(\Sigma_{\swarrow v}) +\sum_{a=m+1}^{m+r}D_{[c_{\tau^{-1}(a)},n]},
\end{equation}
and hence
\be
\Q_{\geq \tau.\bc}^{\bx,\by}(\Sigma)=\Q_{\geq \tau.\bc}^{\bx,\by-\e^{[m+1, m+r]}}(\Sigma_{\swarrow v})\Q_{\geq \tau.\bc[m+1, m+r]}(\bD),
\ee
where we use the notation $\Q_{\geq \bc}(\bD)$ from \eqref{defoflocalQ} and we set
\be
\tau.\bc[m+1, m+r]=(c_{\tau^{-1}(m+1)}, \dots, c_{\tau^{-1}(m+r)}).
\ee
Taking expectations we obtain
\be
\E[\Q_{\geq \tau.\bc}^{\bx,\by}(\Sigma)]=\E\left[\Q_{\geq \tau.\bc}^{\bx,\by-\e^{[m+1, m+r]}}(\Sigma_{\swarrow v})\  \E[\Q_{\geq \tau.\bc[m+1, m+r]}(\bD)  \mid \Sigma_{\swarrow v}]\right].
\ee
Now we can apply Corollary \ref{localRelationSub} to the conditional expectation inside to get
\be
\E[\Q_{\geq \tau.\bc[m+1, m+r]}(\bD)  \mid \Sigma_{\swarrow v}]=\sum_{p=0}^r\sum_{\rho\in S^{p|r-p}_{[m+1,m+r]}} (\kappa/\mu)^p\frac{(\kappa/\mu;q)_{r-p}(\lambda/\kappa)_p}{(\lambda/\mu)_r}q^{l(\rho)}\Q_{\geq \rho\tau.\bc[m+1,m+p]}(\bA),
\ee
where the second sum is over $[m+1, m+p]\times [m+p+1, m+r]$-ordered permutations of $m+1, \dots, m+r$. For the same reasons as in Step 1.1, for any such permutation $\rho$ we have $\Q_{\geq \tau.\bc}^{\bx,\by-\e^{[m+1, m+r]}}=\Q_{\geq \rho\tau.\bc}^{\bx,\by-\e^{[m+1, m+r]}}$ and using the definition of the height function in a way similar to \eqref{useLocalHeight} we have
\be
\Q_{\geq \tau.\bc}^{\bx,\by-\e^{[m+1, m+r]}}(\Sigma_{\swarrow v})\Q_{\geq \rho\tau.\bc[m+1, m+p]}(\bA)=\Q_{\geq \rho\tau.\bc}^{\bx-\e^{[m+1,m+p]},\by-\e^{[m+1, m+r]}}(\Sigma_{\swarrow v}).
\ee
So the local relation from Corollary \ref{localRelationSub} gives the following relation on the left-hand side of \eqref{qHahnIntExpression}:
\begin{equation}
\label{nonlocalQ}
\E[\Q_{\geq \tau.\bc}^{\bx,\by}(\Sigma)]=\sum_{p=0}^r\sum_{\rho\in S^{p|r-p}_{[m+1,m+r]}} (\kappa/\mu)^p\frac{(\kappa/\mu;q)_{r-p}(\lambda/\kappa;q)_p}{(\lambda/\mu;q)_r}q^{l(\rho)}\ \E\left[\Q^{\bx-\e^{[m+1,m+p]},\by-\e^{[m+1, m+r]}}_{\geq \rho\tau.\bc}(\Sigma)\right].
\end{equation}

\paragraph*{{\bfseries Step 1.3.}}\ Continuing with the same setup, we now want to apply the other local relation to the right-hand side of \eqref{qHahnIntExpression}. Using the notation from Step 1.1, namely \eqref{GFnotationTheoProof}, we rewrite the right-hand side as 
\be
\frac{(-1)^kq^{\frac{k(k-1)}{2}-l(\tau)}}{(2\pi\i)^k}\oint_{\Gamma_1}\cdots\oint_{\Gamma_k}\prod_{a<b}\frac{w_b-w_a}{w_b-qw_a}\ G(\bw)\ T_{\tau^{-1}}\( F^{\bx,\by}(\bw)\) \frac{dw_a}{w_a}.
\ee
Note that
\be
F^{\bx,\by}(\bw)=\(\prod_{a=m+1}^{m+r}\frac{1-\kappa w_a}{1-\lambda w_a}\) F^{\bx,\by-\e^{[m+1,m+r]}}(\bw) ,
\ee
so we can apply Proposition \ref{localRelationRat} with $t=\sqrt{\lambda/\kappa}, s=\sqrt{\lambda/\mu}$ to obtain
\be
\prod_{a=m+1}^{m+r}\frac{1-\kappa w_a}{1-\lambda w_a}=\sum_{p=0}^r\sum_{\rho\in S_{[m+1, m+r]}^{p | r-p}} (\kappa/\mu)^p\frac{(\kappa/\mu;q)_{r-p} (\lambda/\kappa;q)_{p}}{(\lambda/\mu;q)_r}T_{\rho^{-1}}\(\prod_{a=m+1}^{m+p}\frac{1-\mu w_a}{1-\lambda w_a}\).
\ee
Since $F^{\bx,\by-\e^{[m+1,m+r]}}(\bw)$ is symmetric in $w_{m+1}, \dots, w_{m+r}$ and
\be
F^{\bx-\e^{[m+1, m+p]},\by-\e^{[m+1,m+r]}}(\bw)=\(\prod_{a=m+1}^{m+p}\frac{1-\mu w_a}{1-\lambda w_a}\) F^{\bx,\by-\e^{[m+1,m+r]}}(\bw) ,
\ee
we get
\be
T_{\tau^{-1}}F^{\bx,\by}(\bw)=\sum_{p=0}^r\sum_{\rho\in S_{[m+1, m+r]}^{p | r-p}} (\kappa/\mu)^p\frac{(\kappa/\mu;q)_{r-p} (\lambda/\kappa;q)_{p}}{(\lambda/\mu;q)_r}T_{(\rho\tau)^{-1}} F^{\bx-\e^{[m+1, m+p]},\by-\e^{[m+1,m+r]}}(\bw).
\ee
Let $\mathcal I^{\bx,\by}_{\tau}$ denote the right-hand side of \eqref{qHahnIntExpression}. Then, the application of the local relation above gives
\be
\mathcal I_{\tau}^{\bx,\by}=\sum_{p=0}^r\sum_{\rho\in S_{[m+1, m+r]}^{p | r-p}} (\kappa/\mu)^p\frac{(\kappa/\mu;q)_{r-p} (\lambda/\kappa;q)_{p}}{(\lambda/\mu;q)_r}q^{l(\rho)}\mathcal{I}_{\rho\tau}^{\bx-\e^{[m+1,m+p]},\by-\e^{[m+1, m+r]}},
\ee
which looks identical to the analogous relation for the left-hand side \eqref{nonlocalQ}. Thus, the identity \eqref{qHahnIntExpression} for $(\bx,\by,\tau)$ follows from $(\bx-\e^{[m+1,m+p]},\by-\e^{[m+1, m+r]}, \rho\tau)$ for  various $p\in [0,r]$ and $\rho\in S_{[m+1, m+r]}^{p | r-p}$.

\paragraph*{{\bfseries Step 2.}}\ Applying Step 1 we can reduce the coordinates $\bx,\by$ until either $x_1=\dots=x_k=\frac{1}{2}$ or $x_k=y_k$. Here we deal with the latter case, assuming that $x_k=y_k$ and showing that in this situation we can reduce $k$ by $1$ and replace the data $(\bx,\by,\bc,\tau)$ by 
\be
\bx'=(x_1, \dots, x_{k-1}),\qquad \by'=(y_1, \dots, y_{k-1}),
\ee 
\be
\bc=(c_{\rho(1)}, \dots, c_{\rho(k-1)})=\rho^{-1}.\bc[1,k-1],\qquad \tau'=\tau\rho,
\ee
where $\rho$ is the cycle sending $\tau^{-1}(k), \tau^{-1}(k)+1, \dots, k$ to $\tau^{-1}(k)+1, \dots, k, \tau^{-1}(k)$ respectively.

For the left-hand side of \eqref{qHahnIntExpression} note that for any color $c$ we have
\be
h^{(x_k,y_k)}_{\geq c}(\Sigma)=0\qquad \text{for}\ x_k=y_k 
\ee
due to the behavior of the model: all edges adjacent to vertices $(i,j)$ for $i>j$ have compositions $\bm 0$. This means that we can freely remove this height function from $\Q_{\geq\tau.\bc}^{\bx,\by}$, obtaining
\be
\Q_{\geq\tau.\bc}^{\bx,\by}(\Sigma)=\Q_{\geq\tau'.\bc'}^{\bx',\by'}(\Sigma).
\ee
for any possible configuration $\Sigma$. Taking expectation we conclude that the left-hand sides for $(\bx,\by,\bc,\tau)$ and $(\bx',\by',\bc',\tau')$ are equal.

For the right-hand side of \eqref{qHahnIntExpression} we are going to apply Proposition \ref{integralReduction}. Set
\be
f_a(w)=\prod_{i=0}^{i<x_a}\frac{1}{1-\mu_i w}\prod_{j=1}^{j<y_a}(1-\kappa_j w_a)\prod_{d=1}^{d\leq y_a-x_a}\frac{1}{1-\lambda_d w}, \qquad g_a(w)=\prod_{j=1}^{l_a}\frac{1-\lambda_j w}{1-\kappa_j w}
\ee
Since $x_k=y_k$, the function $f_k(w)$ has no singularities outside of $\Gamma_k$, as well as the functions $g_a(w)$ for any $a\in [1,k]$. Thus, the conditions of Proposition \ref{integralReduction} hold and we have
\begin{multline*}
\frac{(-1)^kq^{\frac{k(k-1)}{2}-l(\tau)}}{(2\pi\i)^k}\oint_{\Gamma_1}\cdots\oint_{\Gamma_k}\prod_{a<b}\frac{w_b-w_a}{w_b-qw_a}\prod_{a=1}^kg_a(w_a)\  T_{\tau^{-1}}\(\prod_{a=1}^kf_a(w_a)\) \prod_{a=1}^k\frac{dw_a}{w_a}\\
=\frac{(-1)^{k-1}q^{\frac{(k-1)(k-2)}{2}-l(\tau')}}{(2{\pi}\i)^{k-1}}\oint_{\Gamma_1}\cdots\oint_{\Gamma_{k-1}}\prod_{a<b\leq k-1}\frac{w_b-w_a}{w_b-qw_a}\prod_{a=1}^{k-1}g_{\rho(a)}(w_a)\ T_{\tau'^{-1}}\(\prod_{a=1}^{k-1}f_a(w_a)\)\prod_{a=1}^{k-1}\frac{dw_a}{w_a}
\end{multline*}
where $\rho$ and $\tau'$ are exactly the permutations defined above. Plugging back the expressions for $f_a, g_a$, we get the right-hand side of \eqref{qHahnIntExpression}  for $(\bx',\by',\tau',\bc')$, as desired.

\paragraph*{{\bfseries Step 3.}}\ Repeating Steps 1 and 2, at some point we either reach $k=0$, when the claim is trivial, or the situation when $x_1=x_2=\dots=x_k=\frac{1}{2}$. In the latter case Proposition \ref{initialCondition} implies that the right-hand side of \eqref{qHahnIntExpression} is equal to
\be
\prod_{j=1}^N\frac{(\kappa_j\mu_0^{-1};q)_{r_j}}{(\lambda_j\mu_0^{-1};q)_{r_j}}, \qquad r_j=\#\left\{a\in[1,k] \mid j\in\left[l_{\tau^{-1}(a)}+1, y_a-\frac{1}{2}\right]\right\}.
\ee
At the same time, recall that the compositions of the incoming edges are mono-colored, and the color of the incoming edge on row $j$ is $\geq c_{\tau^{-1}(a)}$ if and only if $j> l_{{\tau^{-1}(a)}}$.  Recall that $b_j=|\bB^{(0,j)}|$ denote the random integers used to define the compositions at the left boundary and distributed according to \eqref{leftProb}, so we get for $\bx=\(\frac{1}{2},\dots, \frac{1}{2}\)$
\begin{multline*}
\E[\Q_{\tau.\bc}^{\bx; \by}]= \sum_{b_1, b_2,\dots\geq 0}q^{\sum_{a=1}^k\sum^{j<y_{\tau(a)}}_{j=l_{a}+1} b_j}\prod_{j\geq 1}(\kappa_j/\mu_0)^{b_j}\frac{(\lambda_j/\kappa_j;q)_{b_j}}{(q;q)_{b_j}}\frac{(\kappa_j/\mu_0;q)_\infty}{(\lambda_j/\mu_0;q)_\infty}\\
=\prod_{i\geq 1}\sum_{b\geq 0}q^{r_jb}(\kappa_j/\mu_0)^{b}\frac{(\lambda_j/\kappa_j;q)_{b}}{(q;q)_{b}}\frac{(\kappa_j/\mu_0;q)_\infty}{(\lambda_j/\mu_0;q)_\infty}=\prod_{j\geq 1}\frac{(q^{r_j}\lambda_j/\mu_0;q)_\infty}{(q^{r_j}\kappa_j/\mu_0;q)_\infty}\frac{(\kappa_j/\mu_0;q)_\infty}{(\lambda_j/\mu_0;q)_\infty},
\end{multline*}
which is equal to the value of the right-hand side, concluding the base case. Note that for the last identity we have used \eqref{qbinom}.
\qed


\section{Shift invariance}\label{shiftSect} In this section we show how the integral expressions from Theorem \ref{qHahnResultTheo} lead to certain symmetries of the diagonally inhomogeneous $q$-Hahn model, which are similar to the shift-invariance symmetries of the stochastic colored six-vertex model from \cite{BGW19} and \cite{Gal20}.  Our argument is identical to \cite[Section~7]{BK20}.

To state the result we need a finer description of the height functions. For simplicity, consider a $q$-Hahn vertex model with incoming colors along the left boundary defined by $\bI=(1,1,1,\dots)$, that is, in the incoming boundary composition at row $j$ only $j$th entry is nonzero. Then, for each color $c$ we can define a cutoff point $(\frac{1}{2}, c-\frac{1}{2})$, which separates the rows with the incoming colors on the left $\geq c$  and $<c$. Since such cutoff points uniquely determine the corresponding colors, instead of using a color $c$ and a point $(x,y)$ to identify a colored height function $h^{(x,y)}_{\geq c}$  we can equivalently use a pair of points $(\frac{1}{2}, c-\frac{1}{2})$ and $(x,y)$. It turns out that the latter pair of points plays an important role in visualizing the  behavior of the height function.

For a color $c\in\mathbb Z_{\geq 1}$ and a point $(x,y)\in\(\mathbb Z_{\geq 0}+\frac{1}{2}\)^2$ such that $y-x\geq c$ define the following integer intervals:
\be
\Row_{\geq c}^{(x,y)}:=\left[c, y-\frac{1}{2}\right], \quad \Col_{\geq c}^{(x,y)}:=\left[0, x-\frac{1}{2}\right], \quad \Diag_{\geq c}^{(x,y)}:=[c, y-x].
\ee
Here $[a,b]$ denotes an integer interval $a, a+1, \dots, b$. In other words, $\Row,\Col,\Diag$ are exactly the labels of the rows, columns\footnote{With the exception of the artificially added $0$th column.} and diagonals passing between the points $(\frac{1}{2}, c-\frac{1}{2})$ and $(x,y)$, possibly containing $(x,y)$ but avoiding $(\frac{1}{2}, c-\frac{1}{2})$. The additional constraint $y-x\geq c$ listed above is not restrictive, since  $h_{\geq c}^{(x,y)}\equiv 0$ for $y-x<c$.

We start with the informal statement of the shift-invariance symmetry, which is rigorously described in Proposition \ref{shiftTheo}. Given a vector of height functions $\(h_{\geq c_1}^{(x_{\tau(1)}, y_{\tau(1)})},\dots, h_{\geq c_k}^{(x_{\tau(k)}, y_{\tau(k)})}\)$ with $x_a,y_a, c_a$ ordered as before in Theorem \ref{qHahnResultTheo}, the joint $k$-dimensional distribution is determined solely by
\begin{itemize}
\item The collections of parameters attached to the rows $\Row_{\geq c_a}^{(x_{\tau(a)},y_{\tau(a)})}$, columns $\Col_{\geq c_a}^{(x_{\tau(a)},y_{\tau(a)})}$ and diagonals $\Diag_{\geq c_a}^{(x_{\tau(a)},y_{\tau(a)})}$, corresponding to the height functions,

\item The relative order $\tau$ of evaluation and cutoff points of the height functions.
\end{itemize}
The data in the first part corresponds to the one-dimensional distributions of $h^{(x,y)}_{\geq c}$, which are determined by the listed parameters in view of Theorem \ref{qHahnResultTheo}. Since the one-dimensional distributions are determined by the joint $k$-dimensional distributions, one can expect this data to be necessary. On the contrary, the second part corresponds to the relative position of the height functions with respect to each other on the plane, and this data is not \emph{a priori} contained in the joint distribution. Moreover, we actually expect it to be redundant, see Remark \ref{shiftRem}.

\begin{figure}
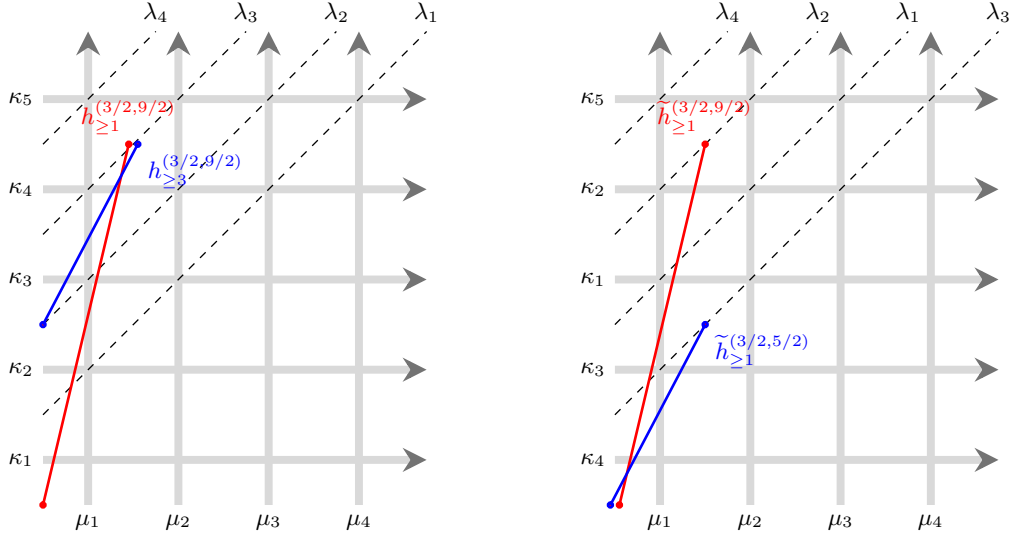

\centerline{
\tikz{0.6}{
	\foreach\y in {1,...,5}{
		\draw[lgray, line width=3pt, arrows={-Stealth[scale=0.9,length=10,width=10,dgray]}] (1,2*\y) -- (9.5,2*\y);
	}
	\foreach\x in {1,...,4}{
		\draw[lgray, line width=3pt, arrows={-Stealth[scale=0.9,length=10,width=10,dgray]}] (2*\x,1) -- (2*\x,11.5);
	}
	\draw[black, line width = 0.5pt, dashed] (1,3) -- (9.5, 11.5) node[above] {\small $\lambda_1$};
	\draw[black, line width = 0.5pt, dashed] (1,5) -- (7.5, 11.5) node[above] {\small $\lambda_2$};
	\draw[black, line width = 0.5pt, dashed] (1,7) -- (5.5, 11.5) node[above] {\small $\lambda_3$};
	\draw[black, line width = 0.5pt, dashed] (1,9) -- (3.5, 11.5) node[above] {\small $\lambda_4$};
	\node[left] at (1,2) {\small $\kappa_1$};
	\node[left] at (1,4) {\small $\kappa_2$};
	\node[left] at (1,6) {\small $\kappa_3$};
	\node[left] at (1,8) {\small $\kappa_4$};
	\node[left] at (1,10) {\small $\kappa_5$};
	\node[below] at (8,1) {\small$\mu_4$};
	\node[below] at (6,1) {\small$\mu_3$};
	\node[below] at (4,1) {\small$\mu_2$};
	\node[below] at (2,1) {\small$\mu_1$};
	
	\draw[red, fill=red, line width = 1pt] (1,1) circle (1.5pt) -- (2.9, 9) circle (1.5pt) node[above] {\small $h_{\geq 1}^{(3/2,9/2)}$};
	\draw[blue, fill=red, line width = 1pt] (1,5) circle (1.5pt) -- (3.1, 9) circle (1.5pt) node[below right] {\small $h_{\geq 3}^{(3/2,9/2)}$};
}
\qquad\qquad
\tikz{0.6}{
	\foreach\y in {1,...,5}{
		\draw[lgray, line width=3pt, arrows={-Stealth[scale=0.9,length=10,width=10,dgray]}] (1,2*\y) -- (9.5,2*\y);
	}
	\foreach\x in {1,...,4}{
		\draw[lgray, line width=3pt, arrows={-Stealth[scale=0.9,length=10,width=10,dgray]}] (2*\x,1) -- (2*\x,11.5);
	}
	\draw[black, line width = 0.5pt, dashed] (1,3) -- (9.5, 11.5) node[above] {\small $\lambda_3$};
	\draw[black, line width = 0.5pt, dashed] (1,5) -- (7.5, 11.5) node[above] {\small $\lambda_1$};
	\draw[black, line width = 0.5pt, dashed] (1,7) -- (5.5, 11.5) node[above] {\small $\lambda_2$};
	\draw[black, line width = 0.5pt, dashed] (1,9) -- (3.5, 11.5) node[above] {\small $\lambda_4$};
	\node[left] at (1,2) {\small $\kappa_4$};
	\node[left] at (1,4) {\small $\kappa_3$};
	\node[left] at (1,6) {\small $\kappa_1$};
	\node[left] at (1,8) {\small $\kappa_2$};
	\node[left] at (1,10) {\small $\kappa_5$};
	\node[below] at (8,1) {\small$\mu_4$};
	\node[below] at (6,1) {\small$\mu_3$};
	\node[below] at (4,1) {\small$\mu_2$};
	\node[below] at (2,1) {\small$\mu_1$};
	
	\draw[red, fill=red, line width = 1pt] (1.1,1) circle (1.5pt) -- (3, 9) circle (1.5pt) node[above] {\small $\widetilde{h}_{\geq 1}^{(3/2,9/2)}$};
	\draw[blue, fill=red, line width = 1pt] (0.9,1) circle (1.5pt) -- (3, 5) circle (1.5pt) node[below right] {\small $\widetilde{h}_{\geq 1}^{(3/2,5/2)}$};
}}
\caption{\label{shiftInvarianceFigure} An example of the shift invariance:  by Proposition \ref{shiftTheo} the joint distribution of the height functions $(h_{\geq 1}^{(3/2,9/2)}, h^{(3/2,9/2)}_{\geq 3})$ on the left is equal to the joint distribution of  $(\widetilde h_{\geq 1}^{(3/2,9/2)}, \widetilde h^{(3/2,5/2)}_{\geq 1})$ on the right. Here the two models have the same row, column and diagonal parameters, but with different orderings corresponding to $\psi_{\Row}, \psi_{\Col}, \psi_{\Diag}$. To illustrate assumptions of Proposition \ref{shiftTheo}, note that for $h^{(3/2,9/2)}_{\geq 3}$ on the left we have rows $\Row^{(3/2,9/2)}_{\geq 3}=\{3,4\}$ with the parameters $\{\kappa_3,\kappa_4\}$, while for $\widetilde h^{(3/2,5/2)}_{\geq 1}$ on the right $\Row^{(3/2,5/2)}_{\geq 1}=\{1,2\}$ with the same parameters $\{\kappa_3,\kappa_4\}$.}
\end{figure}

\begin{prop}\label{shiftTheo} Let $\{\mu_i\}_{i\geq 0}$, $\{\kappa_j\}_{j\geq 1}$, $\{\lambda_d\}_{d\geq 1}$ be parameters of a $q$-Hahn model, and let $\bc=(c_1, \dots, c_k)$, $\bx=(x_1, \dots, x_k)$ and $\by=(y_1, \dots, y_k)$ be sequences satisfying
\be
x_1 \leq x_2\leq\dots\leq x_k,\qquad y_1\geq y_2\geq\dots\geq y_k, \qquad x_a, y_a\in\mathbb Z_{\geq 0}+\frac{1}{2};
\ee
\be
c_1\leq c_2\leq\dots\leq c_k,\qquad c_a\in\mathbb Z_{\geq 0}.
\ee
Additionally, let $\tau\in S_k$ be a permutation such that $y_{\tau(a)}-x_{\tau(a)}\geq c_{a}$. Similarly, let  $\{\widetilde{\mu}_i\}_{i\geq 0}$, $\{\widetilde{\kappa}_j\}_{j\geq 1}$, $\{\widetilde{\lambda}_d\}_{d\geq 1}$, $\widetilde{\bc}=(\widetilde{c}_1, \dots, \widetilde{c}_k)$, $\widetilde{\bx}=(\widetilde{x}_1, \dots, \widetilde{x}_k)$, $\widetilde{\by}=(\widetilde{y}_1, \dots, \widetilde{y}_k)$ and $\widetilde{\tau}=\tau$ be another collection of data satisfying identical constraints, which we treat as another $q$-Hahn model. 

Assume that there exist bijections $\psi_{\Row}, \psi_{\Diag}:\mathbb Z_{\geq 1}\to \mathbb Z_{\geq 1}, \psi_{\Col}:\mathbb Z_{\geq 0}\to \mathbb Z_{\geq 0}$ satisfying 
\be
\psi_{\Row}\(\Row_{\geq c_{a}}^{(x_{\tau(a)}, y_{\tau(a)})}\)=\Row_{\geq \widetilde{c}_{a}}^{(\widetilde{x}_{\tau(a)}, \widetilde{y}_{\tau(a)})},\quad \psi_{\Col}\(\Col_{\geq c_{a}}^{(x_{\tau(a)}, y_{\tau(a)})}\)=\Col_{\geq \widetilde{c}_{a}}^{(\widetilde{x}_{\tau(a)}, \widetilde{y}_{\tau(a)})},
\ee
\be
\psi_{\Diag}\(\Diag_{\geq c_{a}}^{(x_{\tau(a)}, y_{\tau(a)})}\)=\Diag_{\geq \widetilde{c}_{a}}^{(\widetilde{x}_{\tau(a)}, \widetilde{y}_{\tau(a)})},
\ee
for any $a=1, \dots, k$ and 
\be
\widetilde{\kappa}_j=\kappa_{\psi^{-1}_{\Row}(j)}, \quad \widetilde{\mu}_i=\mu_{\psi^{-1}_{\Col}(i)},\quad \widetilde{\lambda}_d=\lambda_{\psi^{-1}_{\Col}(d)}.
\ee
for any $i\geq 0, j\geq1, d\geq 1$. Then the following random vectors are equal in finite-dimensional distributions:
\be
\(h_{\geq c_1}^{(x_{\tau(1)}, y_{\tau(1)})},\dots, h_{\geq c_k}^{(x_{\tau(k)}, y_{\tau(k)})}\)=\(\widetilde{h}_{\geq \widetilde{c}_1}^{(\widetilde{x}_{\tau(1)}, \widetilde{y}_{\tau(1)})},\dots, \widetilde{h}_{\geq \widetilde{c}_k}^{(\widetilde{x}_{\tau(k)}, \widetilde{y}_{\tau(k)})}\),
\ee
where $h$ and $\widetilde{h}$ denote the height functions of the two $q$-Hahn models. 
\end{prop}
See Figure \ref{shiftInvarianceFigure} for an example of the setup of Proposition \ref{shiftTheo}. We are not going to give a complete proof of Proposition \ref{shiftTheo}, since it is identical to the proof of \cite[Theorem 7.1]{BK20}. Below we just outline the main steps of that argument.

\begin{proof}[Sketch of the proof] The proof is based on the following observation: the joint $k$-dimensional distribution of a random vector $(h_1, \dots, h_k)$ of nonnegative real numbers is completely determined by the $q$-moments $\E\left[q^{a_1h_1+\dots+a_kh_k}\right]$ for $a_1, \dots, a_k\in\mathbb Z_{\geq 0}$. Repeating in the statement of Proposition \ref{shiftTheo} the first height function for $a_1$ times, the next height function for $a_2$ times and so on, one readily see that it is enough to prove 
\be
\E\left[\Q^{(\bx,\by)}_{\geq \tau.\bc}(\Sigma)\right]=\widetilde{\E}\left[\Q^{(\widetilde{\bx},\widetilde{\by})}_{\geq \tau.\widetilde{\bc}}(\Sigma)\right],
\ee
where $\widetilde{\E}$ denotes expectation in the second model. Applying Theorem \ref{qHahnResultTheo} to both sides, the problem is reduced to matching the right-hand sides of the integral expression \eqref{qHahnIntExpression}.

To match integral expressions, we expand the operator $T_\tau$ as a linear combination of the permutation operators acting on the variables $\bw$:
\be
T_\tau=\sum_{\rho\preceq \tau}Z_{\tau}^\rho(\bw)\rho, \qquad \rho f(w_1, \dots, w_k)=f(w_{\rho(1)}, \dots, w_{\rho(k)}),
\ee
where $Z_{\tau}^\rho(\bw)$ are rational functions in $\bw$ and the inequality $\rho\preceq \tau$ is taken with respect to the strong Bruhat order, see \cite[Proposition 3.1]{BK20}. Then, to match the integral expressions, it is enough to prove for any $\rho\preceq \tau$ that
\begin{multline}\label{tmpshift}
\oint\cdots\oint\prod_{a<b}\frac{w_b-w_a}{w_b-qw_a}\ Z_\tau^\rho(\bw)\prod_{a=1}^k\prod_{i=1}^{c_{a}-1}\frac{1-\lambda_iw_{\rho(a)}}{1-\kappa_iw_{\rho(a)}}\prod_{i=0}^{i<x_a}\frac{1}{1-\mu_iw_a}\prod_{j=1}^{j<y_a}(1-\kappa_jw_a)\prod_{d=1}^{d\leq y_a-x_a}\frac{1}{1-\lambda_dw_a}\frac{dw_a}{w_a}\\
=\oint\cdots\oint\prod_{a<b}\frac{w_b-w_a}{w_b-qw_a}\ Z_\tau^\rho(\bw)\prod_{a=1}^k\prod_{i=1}^{\widetilde{c}_{a}-1}\frac{1-\widetilde{\lambda}_iw_{\rho(a)}}{1-\widetilde{\kappa}_iw_{\rho(a)}}\prod_{i=0}^{i<\widetilde{x}_a}\frac{1}{1-\widetilde{\mu}_iw_a}\prod_{j=1}^{j<\widetilde{y}_a}(1-\widetilde{\kappa}_jw_a)\prod_{d=1}^{d\leq \widetilde{y}_a-\widetilde{x}_a}\frac{1}{1-\widetilde{\lambda}_dw_a}\frac{dw_a}{w_a}.
\end{multline}
When $\rho=\tau$  both sides of \eqref{tmpshift} are equal by the assumptions of the claim: for the left-hand side the factors in the product over $a$ can be rewritten as
\begin{equation}
\label{shiftterm}
\prod_{a=1}^k\prod_{i\in\Col^{(x_a, y_a)}_{\geq c_{\overline{\rho}(a)}}}\frac{1}{1-\mu_iw_a}\prod_{j\in\Row^{(x_a,y_a)}_{\geq c_{\overline{\rho}(a)}}}(1-\kappa_jw_a)\prod_{d\in\Diag^{(x_a,y_a)}_{\geq c_{\overline{\rho}(a)}}}\frac{1}{1-\lambda_dw_a},
\end{equation}
where $\overline{\rho}=\rho^{-1}$, and the statement of Proposition \ref{shiftTheo} essentially just tells that the expression above is equal to the analogous expression for the right-hand side of \eqref{tmpshift} when $\rho=\tau$.

For general $\rho$ one can closely follow the proof of \cite[Lemma 7.1.3]{BK20} and show that for any $\rho\preceq\tau$ one of the following situations holds:
\begin{itemize}
\item For some $a$ we have $y_{\rho(a)}-x_{\rho(a)}< c_{\rho(a)}$ and $\widetilde{y}_{\rho(a)}-\widetilde{x}_{\rho(a)}<\widetilde{c}_{\rho(a)}$, that is, the corresponding height functions $h^{(x_{\rho(a)}, y_{\rho(a)})}_{\geq c_a}$ and $\widetilde{h}^{(\widetilde{x}_{\rho(a)}, \widetilde{y}_{\rho(a)})}_{\geq \widetilde{c}_a}$ are degenerate and equal to $0$. Then both sides of \eqref{tmpshift} have no nonzero residues with respect to $w_a$ outside of the integration contours, and hence vanish.

\item For every $a$ we have $y_{\rho(a)}-x_{\rho(a)}\geq c_{\rho(a)}$ and $\widetilde{y}_{\rho(a)}-\widetilde{x}_{\rho(a)}\geq\widetilde{c}_{\rho(a)}$. Then one can verify that the conditions of Proposition \ref{shiftTheo} still hold when $\tau$ is replaced by $\rho$ and, in particular, the expressions of the form \eqref{shiftterm} match for $\rho$ as well. 
\end{itemize}
\end{proof}

\begin{rem}\label{shiftRem}\normalfont
This section is based on \cite[Section 7]{BK20}, where the shift-invariance for the \emph{stochastic colored six-vertex model} is derived directly from the integral expressions for the $q$-moments of the height functions. The shift-invariance for the six-vertex model was initially introduced in \cite{BGW19} and in full generality it was proved in \cite{Gal20}, where \emph{flip symmetries} were used as a basic block for all other symmetries. Moreover, it was shown in \cite{Gal20} that the relative position of the height functions, encoded by $\tau$ in our notation, is irrelevant and joint distributions of height functions are determined solely by their one-dimensional distributions. Since we don't know how the $q$-Hahn model with diagonal parameters can be constructed from the six-vertex model, the more general invariance result from \cite{Gal20} cannot be applied to our model, but we suspect it to be true.
\end{rem}


\section{Reduction to Beta polymer}\label{reductionSect}
In this section we follow \cite{BGW19} in order to reduce the results from Section \ref{qHahnSec} to the \emph{Beta polymer} model. See also \cite{BW20} and \cite{BK20} for similar arguments.
\subsection{Continuous model.} As before, we treat $q$-Hahn vertex model from Section \ref{qHahnSec} as a sequence of independent stochastic samplings corresponding to the vertices, depending on the parameters $\mu_i, \kappa_j,\lambda_d$. We slightly simplify the model by assuming that the incoming colors from the left are different and are encoded by the composition $\bI=(1,1,1,\dots)$. We consider the following limit regime:
\begin{equation}
\label{limitRegime}
q=e^{-\ve},\qquad \mu_i=q^{-\sigma_i},\qquad \kappa_i=q^{-\rho_i},\qquad \lambda_i=q^{-\omega_i};\qquad \ve\to0,
\end{equation}
where $\sigma_i, \rho_j,\omega_d$ are new real parameters satisfying
\begin{equation}\label{conditionzero}
\omega_d<\rho_j<\sigma_i,\qquad \text{for\ any\ } i,j,d. 
\end{equation}
It turns out that the limiting model in this regime is described in terms of the Beta distribution. Below we present without proofs the exact description, referring to \cite{BGW19} for the general case, and to \cite{BC15b} for the one-colored case.
\begin{prop} In the regime \eqref{limitRegime} the random variable $\exp(-\ve b_j)$ for $b_j$ distributed as in \eqref{leftProb} converges weakly as $\ve\to 0$ to $\Beta(\sigma_0-\rho_j, \rho_j-\omega_j)$.
\end{prop}
\begin{proof}
See \cite[Proposition 6.16]{BGW19} or \cite[Lemma 2.3]{BC15b}.
\end{proof}
\begin{prop} The stochastic sampling rule for $(\bC,\bD)$ given $(\bA,\bB)$ defined by the vertex weights $W_{\sqrt{\lambda_{j-i}/\kappa_j},\sqrt{\lambda_{j-i}/\mu_i}}(\bA,\bB;\bC,\bD)$ weakly converges in the limit regime
\be
q=e^{-\ve},\qquad \mu_i=q^{-\sigma_i},\qquad \kappa_i=q^{-\rho_i},\qquad \lambda_i=q^{-\omega_i};\qquad \ve\to0,
\ee 
\be
\ve\bA\to {\bm \alpha},\qquad \ve\bB\to {\bm \beta},\qquad \ve\bC\to {\bm {\gamma}}\qquad, \ve\bD\to {\bm {\delta}},
\ee
to the following sampling procedure for the outgoing masses ${\bm \gamma}=(\gamma_1, \dots, \gamma_n)$ and ${\bm \delta}=(\delta_1, \dots, \delta_n)$ given the incoming masses ${\bm \alpha}=(\alpha_1, \dots, \alpha_n)$ and ${\bm \beta}=(\beta_1, \dots, \beta_n)$: We take a Beta-distributed random variable $\eta\sim \Beta(\sigma_i-\rho_j,\rho_j-\omega_{j-i})$ and define $\delta_k$, $0\leq\delta_k\leq\alpha_k$, for all $k$ through
\be
\exp(-\delta_{\geq k})=\exp(-\alpha_{\geq k}) + (1-\exp(-\alpha_{\geq k}))\eta, \qquad k=1,2,\dots,
\ee
where we set $\delta_{\geq k}=\sum_{l\geq k}\delta_l$ and $\alpha_{\geq k}=\sum_{l\geq k}\alpha_l$. The remaining masses $\gamma_k$  are defined by
\be
\gamma_k=\alpha_k+\beta_k-\delta_k, \qquad k=1,2,\dots,
\ee
so the conservation law holds for each color.
\end{prop}
\begin{proof}
See \cite[Corollary 6.21]{BGW19}. For $n=1$ this is also proved in  \cite[Lemma 2.4]{BC15b}.
\end{proof}

The two statements above result in the following model, serving as the limit under \eqref{limitRegime} of the inhomogeneous $q$-Hahn vertex model and depending on the parameters $\sigma_i, \rho_j, \omega_d$:
\begin{itemize}

\item Each lattice edge of the positive quadrant $\mathbb Z_{\geq 1}^2$ has a random real-valued vector ${\bm u}=(u_1, u_2, u_3, \dots)$ attached to it. Coordinate $u_i\geq 0$ is interpreted as the \emph{mass} of color $i$;

\item To sample a configuration of the model we first sample independent Beta-distributed variables $\eta_{i,j}\sim \Beta(\sigma_i-\rho_j, \rho_j-\omega_{j-i})$ for all $ i\geq 0, j\geq 1$;

\item For the edges entering the quadrant from below all masses are set to $0$;

\item For the edge entering the quadrant at row $j$ only the mass $u_j$ is non-zero and it is equal to $-\ln(\eta_{0,j})$;

\item Given the incoming masses of colors ${\bm \alpha},{\bm \beta}$ entering into a vertex $(i,j)$ from the bottom and from the left, the outgoing masses ${\bm\gamma}, {\bm \delta}$ exiting the vertex to the top and to the right are defined through
\be
\exp(-\delta_{\geq k})=\exp(-\alpha_{\geq k})+(1-\exp(-\alpha_{\geq k}))\eta_{i,j}, \qquad k=1,2,\dots,
\ee
\be
\gamma_k=\alpha_k+\beta_k-\delta_k,\qquad k=1,2,\dots,
\ee
where $\delta_{\geq k}=\sum_{l\geq k}\delta_l$ and $\alpha_{\geq k}=\sum_{l\geq k}\alpha_l$.

\item This definition implies that only colors $\leq j$ pass through vertex $(i,j)$, guaranteeing that at any vertex we are working only with finite sequences of masses. 
\end{itemize}

We can define the colored height functions $\widetilde{h}_{\geq c}^{(x,y)}$ of the continuous model in the same manner as for the $q$-Hahn model: we set $\widetilde{h}_{\geq c}^{(x,\frac{1}{2})}=0$ and 
\begin{itemize}
\item If the vertical edge $(x-\frac{1}{2},y-\frac{1}{2})\to(x-\frac{1}{2}, y+\frac{1}{2})$ has masses ${\bm u}=(u_1, u_2,\dots)$, then
\be
\widetilde{h}^{(x-1,y)}_{\geq c}=\widetilde{h}^{(x,y)}_{\geq c} + u_{\geq c}.
\ee
\item If the horizontal edge $(x-\frac{1}{2},y+\frac{1}{2})\to(x+\frac{1}{2}, y+\frac{1}{2})$ has masses ${\bm u}=(u_1, u_2, \dots)$, then
\be
\widetilde{h}^{(x,y +1)}_{\geq c}=\widetilde{h}^{(x,y)}_{\geq c} + u_{\geq c}.
\ee
\end{itemize}

The following generalization of \cite[Corollary 6.22]{BGW19} straightforwardly follows from the statements above and clarifies in what sense the continuous model is the limit of the $q$-Hahn model.

\begin{cor}\label{limitCor}
Let $h^{(x,y)}_{\geq c}[\ve]$ denote the height function of the $q$-Hahn model in the regime \eqref{limitRegime}. Then in finite-dimensional distributions 
\be
\lim_{\ve\to 0}\ve h^{(x,y)}_{\geq c}[\ve]=\widetilde{h}_{\geq c}^{(x,y)}.
\ee
\end{cor}

\subsection{Beta polymer.} Fix the same collection of parameters $\{\sigma_i, \rho_j, \omega_{d}\}$ as before. For $(x,y)\in\mathbb Z^2_{\geq 0}$\footnote{From now on $x,y$ will usually be integers, rather than half-integers as before.} and $r\in\mathbb Z_{\geq0}$ such that $r\leq y-x$ define the \emph{delayed inhomogeneous Beta polymer} partition function $\Z_{x,y}^{(r)}$ as follows:
\begin{itemize}
\item Let $\{\eta_{i,j}\}_{i\geq 0, j\geq 0}$ be a collection of independent Beta distributed variables
\be
\eta_{i,j}\sim \Beta(\sigma_i-\rho_j, \rho_j-\omega_{j-i});
\ee

\item For all $t,r\geq0$ set
\be
\Z_{t,r+t}^{(r)}\equiv1,\qquad \Z_{0, r+t}^{(r)}=\eta_{0,r+1}\eta_{0,r+2}\dots\eta_{0,r+t};
\ee

\item For general $(x,y)\in\mathbb Z^2_{\geq 1}$ such that $x>y+r$ the partition function $\Z^{(r)}_{x,y}$ is defined by the recurrence relation
\be
\Z^{(r)}_{x,y}=\eta_{x,y}\Z^{(r)}_{x,y-1}+(1-\eta_{x,y})\Z^{(r)}_{x-1,y-1}.
\ee
\end{itemize}
The model just described is an inhomogeneous generalization of the Beta polymer model from \cite{BC15b}, with three families of the parameters. See Section \ref{introSect} for a graphical description of $\Z^{(r)}_{x,y}$.


It turns out that the continuous vertex model described earlier can be identified with the delayed inhomogeneous Beta polymer model, in the way identical to the homogeneous situation \cite{BGW19}: repeating the proof of \cite[Proposition 7.1]{BGW19} verbatim one can verify that for $(x,y)\in\mathbb Z^2_{\geq 0}$ the height functions $\widetilde{h}_{\geq c}^{(x+\frac{1}{2},y+\frac{1}{2})}$ of the continuous vertex model as random variables can be defined by the recurrence relation
\be
\exp\(-\widetilde{h}_{\geq c}^{(x+\frac{1}{2},y+\frac{1}{2})}\)=\eta_{x,y}\exp\(-\widetilde{h}_{\geq c}^{(x+\frac{1}{2},y-\frac{1}{2})}\)+(1-\eta_{x,y})\exp\(-\widetilde{h}_{\geq c}^{(x-\frac{1}{2},y-\frac{1}{2})}\)
\ee 
and the initial conditions
\be
\widetilde{h}_{\geq c}^{(x+\frac{1}{2},x+\frac{1}{2})}=0, \qquad \widetilde{h}_{\geq c}^{(\frac{1}{2},y+\frac{1}{2})}=\sum_{j=c}^{y} -\ln(\eta_{0,j})
\ee
for exactly the same choice of $\eta_{i,j}$. Comparing the recurrence relations and initial conditions for $\widetilde{h}_{\geq c}^{(x+\frac{1}{2},y+\frac{1}{2})}$ and $\Z_{x,y}^{(r)}$ one obtains
\begin{cor}\label{qHahnContVMCor} The finite dimensional distributions of
\be
\left\{\Z_{x,y}^{(r)}\right\}_{x,r\geq 0; y\geq x+r}
\ee
coincide with the finite dimensional distributions of 
\be
\left\{\exp\(-\widetilde{h}_{\geq r+1}^{(x+\frac{1}{2}, y+\frac{1}{2})}\)\right\}_{x,r\geq 0; y\geq x+r}.
\ee
\end{cor}

As a result of the discussion in this section, we have realized the Beta polymer partition functions as certain limits of the height functions of the $q$-Hahn model. Thus, Theorem \ref{qHahnResultTheo} can be used to obtain an integral expression for the joint moments of the partitions functions $\{\Z_{i,j}^{(r)}\}$. To formulate the result, consider a degenerated representation $ t_\tau$ of the Hecke algebra, acting on functions in $v_1, \dots, v_k$ and defined by the generators
\be
t_i=1+\frac{v_{i+1}-v_{i}-1}{v_{i+1}-v_{i}}(\mathfrak s_i-1).
\ee

\begin{prop} \label{delayedBetaIntExpressionProp}
For any sequences $(r_1, \dots, r_k)$, $(x_1, \dots, x_k)$ and $(y_1, \dots, y_k)$ satisfying
\be
x_1 \leq x_2\leq\dots\leq x_k, \qquad y_1\geq y_2\geq\dots\geq y_k, \qquad, x_a\leq y_a, \qquad x_a,y_b\in\mathbb Z_{\geq 0};
\ee
\be
r_1\leq r_2\leq\dots\leq r_k,\qquad r_a\in\mathbb Z_{\geq 0};
\ee
and for any permutation $\tau\in S_k$ the following holds:
\begin{multline}
\label{delayedBetaIntExpression}
\E\left[\Z_{x_1,y_1}^{\(r_{\tau^{-1}(1)}\)} \Z_{x_2,y_2}^{\(r_{\tau^{-1}(2)}\)}\dots \Z_{x_k,y_k}^{\(r_{\tau^{-1}(k)}\)}\right]=\oint_{\mathcal S_1}\cdots\oint_{\mathcal S_k}\prod_{a<b}\frac{v_b-v_a}{v_b-v_a-1}\ t_\tau\(\prod_{a=1}^k\prod_{j=1}^{r_a}\frac{v_a-\omega_j}{v_a-\rho_j}\)\\
\times \prod_{a=1}^k\(\prod_{j=1}^{j\leq y_a}(v_a-\rho_j)\prod_{i=0}^{i\leq x_a}\frac{1}{v_a-\sigma_i}\prod_{d=1}^{d\leq y_a-x_a}\frac{1}{v_a-\omega_d}\)\frac{dv_a}{2\pi\i},
\end{multline}
where the integration contours $\mathcal S_a$ encircle the points $\{\sigma_i\}_{i}$ and $\{\rho_j\}_j$ as well as the contours $\mathcal S_b, \mathcal S_b+1$ for any $b<a$, while the points $\{\omega_d\}_d$ are outside of the contours. Here $\mathcal S_b+1$ denotes the contour $\mathcal S_b$ shifted to the right by $1$. 
\end{prop}
\begin{proof}
We take the limit \eqref{limitRegime} in the integral expression from Theorem \ref{qHahnResultTheo}, applied to the height functions $h^{(x_a+\frac{1}{2}, y_a+\frac{1}{2})}_{\geq r_{\tau^{-1}(a)}+1}$ and assuming $\bI=(1,1,\dots)$. Set 
\be
\bx+\frac{1}{2}=\(x_1+\frac{1}{2}, \dots, x_k+\frac{1}{2}\),\quad \by+\frac{1}{2}=\(y_1+\frac{1}{2}, \dots, y_k+\frac{1}{2}\), \quad \bc=(r_1+1, \dots, r_k+1).
\ee

For the left-hand side of \eqref{qHahnIntExpression} we can apply Corollaries \ref{limitCor} and \ref{qHahnContVMCor} to obtain
\be
\E[\Q^{\bx+\frac{1}{2}, \by+\frac{1}{2}}_{\geq \tau.\bc}]\to \E\left[\Z_{x_1,y_1}^{\(r_{\tau^{-1}(1)}\)} \Z_{x_2,y_2}^{\(r_{\tau^{-1}(2)}\)}\dots \Z_{x_k,y_k}^{\(r_{\tau^{-1}(k)}\)}\right],
\ee
since the finite-dimensional distributions of $q^{h^{(x+\frac{1}{2},y+\frac{1}{2})}_{\geq r+1}}=\exp(-\ve h^{(x+\frac{1}{2},y+\frac{1}{2})}_{\geq r+1})$ converge to the finite-dimensional distributions of $\exp(-\widetilde{h}_{\geq r+1}^{(x+\frac{1}{2},y+\frac{1}{2})})\sim \Z^{(r)}_{x,y}$ as $\ve\to 0$.

For the right-hand side of \eqref{qHahnIntExpression} we perform a change of coordinates $w_a=1-\ve v_a$. One can readily see that for fixed contours $\mathcal S_1, \dots, \mathcal S_k$ satisfying the conditions of the claim and a sufficiently small $\ve$ the contours $1-\ve \mathcal S_a$ satisfy the conditions for the contours $\Gamma_a$, and so in the right-hand side after a change of contours we have an integral over fixed compact contours, with integrand depending on $\ve$. Point-wise we have
\be
\frac{(-1)^kq^{\frac{k(k-1)}{2}-l(\tau)}}{(2\pi\i)^k}\prod_{a=1}^k\frac{dw_a}{w_a}=\ve^k  \prod_{a=1}^k\frac{dv_a}{2\pi\i}+O(\ve^{k+1}),
\ee
\be
\prod_{a<b}\frac{w_b-w_a}{w_b-qw_a}\to\prod_{a<b}\frac{v_b-v_a}{v_b-v_a-1},
\ee
\be
T_{i}=q+\frac{w_{i+1}-qw_i}{w_{i+1}-w_i}(\mathfrak{s}_i-1)\to t_i=1+\frac{v_{i+1}-v_i-1}{v_{i+1}-v_i}(\mathfrak{s}_i-1),
\ee
\be
1-\mu_iw_a=\ve(v_a-\sigma_i)+O(\ve^2),\quad 1-\kappa_iw_a=\ve(v_a-\rho_i)+O(\ve^2),\quad 1-\lambda_iw_a=\ve(v_a-\omega_i)+O(\ve^2).
\ee
One can readily check that the integrand is uniformly bounded for fixed contours $\mathcal S_a$ and sufficiently small $\ve$, so by dominated convergence the claim follows.
\end{proof}

As a particular case of Proposition \ref{delayedBetaIntExpressionProp}, we have an integral expression for the moments of a single partition function $\Z^{(r)}_{x,y}$ of the inhomogeneous Beta polymer. Note that, up to a shift of parameters, one dimensional distributions of $\Z^{(r)}_{x,y}$ depend only on $x$ and $y-r$ (this is the vector between end points of paths inside the polymer), so without loss of generality we can set $r=0$ and consider the moments of $\Z_{x,y}:=\Z^{(0)}_{x,y}$:

\begin{cor}
For any $(x,y)\in\mathbb Z_{\geq 0}^2$ such that $y\geq x$ and any $k\in\mathbb Z_{\geq 0}$ we have
\begin{equation}
\label{BetaMometsIntExpression}
\E\left[\(\Z_{x,y}\)^{k}\right]=\oint_{\mathcal S_1}\cdots\oint_{\mathcal S_k}\prod_{a<b}\frac{v_b-v_a}{v_b-v_a-1}\prod_{a=1}^k\(\prod_{i=1}^{i\leq y}(v_a-\rho_i)\prod_{i=0}^{i\leq x}\frac{1}{v_a-\sigma_i}\prod_{i=1}^{i\leq y-x}\frac{1}{v_a-\omega_i}\)\frac{dv_a}{2\pi\i},
\end{equation}
where the integration contours $\mathcal S_a$ encircle $\sigma_i$, leave $\omega_d$ outside and $\mathcal S_b$ encircles both $\mathcal S_a$ and $\mathcal S_a+1$ for $a<b$.
\end{cor}


\section{Laplace transform for inhomogeneous Beta polymer partition function}\label{fredSect}
From now on we focus on the partition function $\Z_{x,y}=\Z_{x,y}^{(0)}$ of the inhomogeneous Beta polymer, with parameters $\sigma_i, \rho_j, \omega_d$. Our goal is Theorem \ref{mainBetaResult}, which partially establishes the Tracy-Widom large-scale behavior and which is proved in next section. Our result generalizes the corresponding result from \cite{BC15b}, moreover, our proof follows the same outline: in this section we find a suitable expression for the Laplace transform of $\Z_{x,y}$ in terms of a Fredholm determinant, which is asymptotically analyzed in next section to establish the limit theorem. 

Our starting point is the integral expression \eqref{BetaMometsIntExpression} for the moments of $\Z_{x,y}$. Fix $x,y$ and the parameters of the model $\sigma_i, \rho_j,\omega_d$, and assume that
\begin{equation}\label{conditionone}
\max_i \sigma_i-\min_i \sigma_i<1.
\end{equation}
Since our aim is an integral expression for the Laplace transform of $\Z_{x,y}$ which involves all the moments, it is convenient to replace the integration over shifted contours $\mathcal S_1, \dots, \mathcal S_k$ in \eqref{BetaMometsIntExpression} by an integration over a single contour $\C$. This is done with the help of the following statement:

\begin{prop} Suppose that contours $\mathcal S_1, \dots, \mathcal S_k$ and a rational function $f(v)$ satisfy
\begin{itemize}
\item The contours $S_a$ are positively-oriented simple closed curves in $\mathbb C$;
\item For any $a<b$ the contour $\mathcal S_b$ encircles both contours $\mathcal S_a$ and $\mathcal S_a+1$;
\item The contour $\mathcal S_1+1$ is completely outside of $\mathcal S_1$; 
\item All the poles of the function $f(v)$ are either in the interior of $\mathcal S_1$ or in the exterior of $\mathcal S_k$.
\end{itemize}
Then 
\begin{multline}
\label{contourShift}
\oint_{\mathcal S_1}\cdots\oint_{\mathcal S_k}\prod_{a<b}\frac{v_b-v_a}{v_b-v_a-1}\prod_{a=1}^kf(v_a)\frac{dv_a}{2\pi\i}\\
=\sum_{\lambda\vdash k}\frac{k!}{m_1!m_2!\dots}\oint_{\mathcal S_1}\cdots\oint_{\mathcal S_1}\det\left[\frac{1}{v_i+\lambda_i-v_j}\right]_{i,j=1}^{l(\lambda)}\prod_{a=1}^{l(\lambda)} f(v_a)f(v_{a}+1)\dots f(v_a+\lambda_a-1)\frac{dv_a}{2\pi\i},
\end{multline}
where the sum in the right-hand side is over partitions $\lambda$ of $k$, $l(\lambda)$ denotes the number of nonzero parts of $\lambda$ and $m_i$ denotes the number of parts of $\lambda$ equal to $i$. 
\end{prop}
\begin{proof}[Rough idea of the proof] The statement above follows a type of deduction called \emph{contour shift argument}. The main idea is to shrink the contours $\mathcal S_2, \dots, \mathcal S_k$ in the left-hand side to the smallest contour $\mathcal S_1$, picking up the residues at $v_b=v_a+1$ for $b>a$. Due to the conditions on the function $f$ and the contours, all non-zero residues along this deformation come from the poles of $\prod_{a<b}\frac{v_b-v_a}{v_b-v_a-1}$, and a careful bookkeeping produces the right-hand side.
 
 A detailed explanation of the full argument can be found in \cite[Proposition 3.2.1]{BC11} and \cite[Proposition 7.4]{BCPS13}, where more general $q$-deformed versions of the argument are presented. To adopt those claims to our setting one have to repeat the argument similar to the proof of Proposition \ref{delayedBetaIntExpressionProp}, see \cite[Proposition 6.2.7]{BC11} and \cite[Proposition 3.6]{BC15b}.
\end{proof}

Let 
\be
f(z)=\prod_{i=1}^y(z-\rho_i)\prod_{i=0}^x\frac{1}{z-\sigma_i}\prod_{i=1}^{y-x}\frac{1}{z-\omega_i}
\ee
and choose $\C=\mathcal S_1$ to be a sufficiently small contour around the points $\sigma_i$ such that the shifted contour $\C+1$ and the points $\omega_d$ are completely outside of $\C$. Such contour exists due to the conditions \eqref{conditionzero},\eqref{conditionone}.

\begin{cor} With the notation above, we have
\begin{equation}
\label{BetaMomentsUnifiedIntExpression}
\E\left[\(\Z_{x,y}\)^{k}\right]=\sum_{\lambda\vdash r}\frac{k!}{m_1!m_2!\dots}\oint_{\C}\cdots\oint_{\C}\det\left[\frac{1}{v_i+\lambda_i-v_j}\right]_{i,j=1}^{l(\lambda)}\prod_{a=1}^{l(\lambda)} f(v_a)f(v_{a}+1)\dots f(v_a+\lambda_a-1)\frac{dv_a}{2\pi\i}.
\end{equation}
\end{cor}
\begin{proof} Direct application of \eqref{contourShift}.
\end{proof}

Now we want to take a sum over $k$ in \eqref{BetaMomentsUnifiedIntExpression} to obtain the Laplace transform. But to resolve possible convergence issues we need to estimate the integrand for large $\lambda$. Set
\be
g(z)=\prod_{i=0}^x\Gamma(z-\sigma_i)\prod_{j=1}^y\frac{1}{\Gamma(z-\rho_j)}\prod_{d=1}^{y-x}\Gamma(z-\omega_d),
\ee
where $\Gamma(z)$ is the gamma function. Then
\be
f(z)f(z+1)f(z+2)\dots f(z+n-1)=\frac{g(z)}{g(z+n)}.
\ee
\begin{lem}\label{assymptg}
For a fixed $\delta>0$ and any $z\in\mathbb C$ such that $\arg(z-\sigma)\in[-\pi+\delta;\pi-\delta]$  we have as $z\to \infty$
\be
\ln g(z)=\(z-\Delta-\frac{1}{2}\)\ln z -z + \frac{1}{2}\ln(2\pi)+O\(z^{-1}\)
\ee
where $\Delta=\sum_{i=0}^x\sigma_i-\sum_{j=1}^y\rho_j+\sum_{d=1}^{y-x}\omega_d$.
\end{lem}
\begin{proof}
We use the following useful asymptotic expansion from \cite{Luk69}[(1),(4), pp. 31-32], valid uniformly for $\arg(z)\in[-\pi+\delta/2;\pi-\delta/2]$ and $|z|\to\infty$
\begin{equation}
\label{Gammaass}
\ln\Gamma(z)=\(z-\frac{1}{2}\)\ln z -z + \frac{1}{2}\ln(2\pi)+O\(z^{-1}\),
\end{equation}
and, more generally, for a fixed $a\in \mathbb C$ and sufficiently large $z$ such that $\arg(z)\in[-\pi+\delta/2;\pi-\delta/2]$
\be
\ln\Gamma(z+a)=\(z+a-\frac{1}{2}\)\ln z -z + \frac{1}{2}\ln(2\pi)+O\(z^{-1}\).
\ee
Here $\ln$ denotes the principal branch of the natural logarithm. Plugging this expression into the definition of $g(z)$ we get
\begin{multline*}
\ln g(z)=\sum_{i=0}^x\(\(z-\sigma_i-\frac{1}{2}\)\ln(z) -z+\frac{1}{2}\ln(2\pi)\) - \sum_{j=1}^y\(\(z-\rho_j-\frac{1}{2}\)\ln(z) -z+\frac{1}{2}\ln(2\pi)\) \\
+\sum_{d=1}^{y-x}\(\(z-\omega_d-\frac{1}{2}\)\ln(z) -z+\frac{1}{2}\ln(2\pi)\)+ O\(z^{-1}\),
\end{multline*}
implying the claim.
\end{proof}

\begin{prop} \label{preFredholmProp}For any $u\in\mathbb C$ we have
\begin{equation}
\label{preFredholm}
\E\left[ e^{u\Z_{x,y}}\right]=\det(I+\widetilde{K_u})_{L^2(\mathbb Z_{\geq 1}\times \mathcal C)}:=\sum_{l\geq 0}\frac{1}{l!}\sum_{n_1, \dots, n_l\geq 1}\oint_{\C}\frac{dv_1}{2\pi\i}\cdots\oint_{\C}\frac{dv_l}{2\pi\i}\det\left[\widetilde K_u(v_a, n_a; v_b, n_b)\right]_{a,b=1}^l,
\end{equation}
where $\det(I+K)$ denotes the Fredholm determinant of $K$, and $\widetilde{K_u}$ is the integral operator with the kernel
\be
\widetilde{K_u}(v,n;v',n')=\frac{f(v)f(v+1)\dots f(v+n-1)u^n}{v+n-v'}=\frac{g(v)}{g(v+n)}\frac{u^n}{v+n-v'}.
\ee
\end{prop}
\begin{proof}
The proof essentially repeats \cite[Proposition 3.2.8]{BC11}. First we rewrite \eqref{BetaMomentsUnifiedIntExpression} as
\be
\frac{u^k\E\left[\(\Z_{x,y}\)^k \right]}{k!}=\sum_{\lambda\vdash k}\frac{1}{m_1!m_2!\dots}\oint_{\C}\frac{dv_1}{2\pi\i}\dots\oint_{\C}\frac{dv_{l(\lambda)}}{2\pi\i}\det\left[\widetilde{K}_u(v_a, \lambda_a; v_b, \lambda_b)\right]_{a,b=1}^{l(\lambda)}
\ee
and then change the summation over partitions to a summation over arbitrary sequences, absorbing the multinomial coefficient:
\begin{equation}
\label{preFredholmtemp}
\frac{u^k\E\left[\(\Z_{x,y}\)^k \right]}{k!}=\sum_{l=0}^k\frac{1}{l!}\sum_{\substack{n_1, \dots, n_l\geq 1\\n_1+\dots+n_l=k}}\oint_{\C}\frac{dv_1}{2\pi\i}\dots\oint_{\C}\frac{dv_l}{2\pi\i}\det\left[\widetilde{K}_u(v_a, n_a; v_b, n_b)\right]_{a,b=1}^l.
\end{equation}
Now we want to take the sum over $k$, but we need to be careful about possible convergence issues. For the left-hand side we have $|\Z_{x,y}|\leq 1$ almost surely, so the sum converges absolutely for any $u$ and by dominated convergence gives the Laplace transform:
\be
\E\left[e^{u\Z_{x,y}}\right]=\sum_{k\geq 0}\frac{u^k\E\left[\Z_{x,y}^k \right]}{k!}.
\ee
For the right-hand side we can use the conditions on the parameters $\sigma_i, \rho_j,\omega_d$ and on the compact contour $\C$ to give a uniform in $(v,v',n)\in\C\times\C\times\mathbb Z_{\geq 1}$ upper-bound
\be
\left|\frac{g(v)}{v+n-v'}\right|<\Lambda_1.
\ee
At the same time, by Lemma \ref{assymptg} we have
\be
\frac{1}{g(v+n)}=\exp\(-\(v+n-\Delta-\frac{1}{2}\)\ln(n) +n - \frac{1}{2}\ln(2\pi)+O\(n^{-1}\)\),
\ee
where $\Delta=\sum_{i=0}^x\sigma_i-\sum_{j=1}^y\rho_j+\sum_{d=1}^{y-x}\omega_d$ and the constants of $O\(n^{-1}\)$ can be chosen independently of $v\in\C$, since $\C$ is bounded. Hence for some constants $\Lambda_2,\ve>0$ and any $n\in\mathbb Z_{\geq 1}$ we have
\be
\frac{1}{|g(v+n)|}<\Lambda_2 n^{-\ve n}.
\ee
Gathering all pieces, we obtain
\be
|\widetilde{K}_u(v, n; v', n')|<\Lambda_1\Lambda_2n^{-\ve n}|u|^n
\ee
and hence, by Hadamard's inequality,
\be 
\left|\oint_{\C}\frac{dv_1}{2\pi\i}\dots\oint_{\C}\frac{dv_l}{2\pi\i}\det\left[\widetilde{K}_u(v_a, n_a; v_b, n_b)\right]_{a,b=1}^l\right|\leq \Lambda_1\Lambda_2 \xi^l l^{l/2}|u|^{{n_1+\dots n_l}}\prod_{i=1}^l n_i^{-\ve n_i},
\ee
where $\xi>0$ is the length of the contour $\C$. Since
\be
\sum_{l\geq 0}\sum_{\substack{n_1, \dots, n_l}\geq 1} \frac{\xi^l l^{l/2}}{l!} |u|^{{n_1+\dots n_l}}\prod_{i=1}^l n_i^{-\ve n_i} <\infty
\ee
for any $u\in\mathbb C$,  the right-hand side of \eqref{preFredholmtemp} absolutely converges.
\end{proof}

So, we have obtained an expression for the Laplace transform in terms of a Fredholm determinant over $L^2(\mathbb Z_{\geq 1}\times \C)$. But the summation over $\mathbb Z_{\geq 1}$ is inconvenient for the asymptotical analysis, so we replace it by another integration using Mellin-Barnes integral formula.  To proceed we add another assumption, namely, we assume that there exists a vertical line $\L=\{z \mid \Re[z]=h\}$ for some $h$ such that the contours $\C$ and $\C+1$ are separated by $\mathcal L$. Note that this is still possible because of \eqref{conditionone}.

\begin{prop}
\label{MB}
For any $u\in\mathbb C$ such that $\arg(-u)\in \(-\frac{\pi}{2};\frac{\pi}{2}\)$ we have
\be
\sum_{n\geq 1} \frac{g(v)}{g(v+n)}\frac{u^n}{v+n-v'}=\frac{1}{2\pi\i}\int_{\L}\frac{-\pi}{\sin(\pi (z-v))}(-u)^{z-v}\frac{g(v)}{g(z)}\frac{dz}{z-v'},
\ee
where $\L$ is a vertical line $\{z\mid \Re[z]=h\}$ directed upwards and for $(-u)^{z-v}=e^{(z-v)\ln(-u)}$ we use the principal branch of the natural logarithm.
\end{prop}
\begin{proof} The statement and the proof are similar to \cite[Lemma 3.7]{BC15b}, but our restrictions on $u$ are different.

Since the contour $\C$ is compact, there exists $N_0\in\mathbb Z_{>0}$ such that the contour $\C$ is entirely inside the rectangle formed by the lines $\mathcal L, \mathcal L-1, \mathbb R+N_0\i, \mathbb R-N_0\i$. For $N>N_0$ let $\mathcal H_N$ denote the rectangle formed by the lines $\mathcal L$, $\mathcal L+N$, $\mathbb R+N\i, \mathbb R-N\i$, oriented \emph{clockwise}. Then the integral
\be
\frac{1}{2\pi\i}\int_{\mathcal H_N}\frac{-\pi}{\sin(\pi (z-v))}(-u)^{z-v}\frac{g(v)}{g(z)}\frac{dz}{z-v'}
\ee
has nonzero residues only at the poles of $\frac{\pi}{\sin(\pi (z-v))}$, since all the residues of $\frac{g(v)}{g(z)}$ are to the left of the line $\mathcal L$. Computing the residues at $z=v+n$ for $n=1,\dots, N$ we get
\be
\sum_{n\geq 1}^N \frac{g(v)}{g(v+n)}\frac{u^n}{v+n-v'}=\frac{1}{2\pi\i}\int_{\mathcal H_N}\frac{-\pi}{\sin(\pi (z-v))}(-u)^{z-v}\frac{g(v)}{g(z)}\frac{dz}{z-v'}.
\ee
Now we want to take $N\to\infty$ in both sides. For the left-hand side, using the same upper-bound on $K(v,n;v'n')$ as in Proposition \ref{preFredholmProp}, we have
\begin{equation}
\label{upperboundKtilde}
\sum_{n\geq 1} \left|\frac{g(v)}{g(v+n)}\frac{u^n}{v+n-v'}\right|<\sum_{n\geq 1} \Lambda_1\Lambda_2n^{-\ve n}|u|^n<\infty.
\end{equation}
so the absolute convergence holds for any $u\in\mathbb C$.

For the right-hand side we need to show that the integral along the symmetric difference of $\mathcal L$ and $\mathcal H_N$ converges to $0$. This symmetric difference $\mathcal D_N$ consists of the line segments sequentially connecting  the points $h-\infty \i, h-N\i, h+N-N\i, h+N+N\i, h+N\i, h+\infty \i$.

Clearly, $\frac{1}{z-v'}$ for $z\in\mathcal D_N$ can be bounded uniformly in $N$, while $\frac{\pi}{\sin(\pi (z-v))}$ decays exponentially in $\Im[z]$:
\be
\left|\frac{-\pi}{\sin(\pi (z-v))}\right|<\Lambda_3e^{-\pi |\Im(z)|}, \qquad z\in\mathcal D_N,
\ee
where $\Lambda_3$ does not depend on $N$. To bound $\frac{1}{g(z)}$ we use Lemma \ref{assymptg} to get
\be
\frac{1}{|g(z)|}=\(1+O(z^{-1})\)\frac{|e^zz^{-z+\Delta+\frac{1}{2}}|}{\sqrt{2\pi}}<\Lambda_4 |z|^{\Delta+\frac{1}{2}} (|z|/e)^{-\Re[z]} e^{\frac{\pi}{2}|\Im[z]|}
\ee
for a constant $\Lambda_4$ and $z\in\mathcal D_N$ for sufficiently large $N$. Finally, due to the condition on $u$, we have
\be
|(-u)^z|<|u|^{\Re[z]}e^{\(\frac{\pi}{2}-\delta\)|\Im[z]|}
\ee
for some $\delta>0$. Combining everything together we see that for some constant $\Lambda_5>0$ the integrand is bounded by 
\be
\Lambda_5 e^{-\delta |\Im[z]|}|z|^{-\Re[z]} (|u|e)^{\Re[z]} |z|^{\Delta+\frac{1}{2}}
\ee
which is enough to conclude that the integral over $\mathcal D_N$ tends to $0$.
\end{proof}

\begin{prop}\label{fredholm} For any $u\in\mathbb C$ such that $\arg(-u)\in (-\frac{\pi}{2};\frac{\pi}{2})$ we have
\be
\E\left[ e^{u\Z_{x,y}}\right]=\det(I+{K_u})_{L^2(\mathcal C)}=\sum_{l\geq 0}\frac{1}{l!}\oint_{\C}\frac{dv_1}{2\pi\i}\cdots\oint_{\C}\frac{dv_l}{2\pi\i}\det\left[K_u(v_a; v_b)\right]_{a,b=1}^l,
\ee
where the integral operator ${K_u}$ is defined by the kernel
\be
{K_u}(v;v')=\frac{1}{2\pi\i}\int_{\mathcal L}\frac{-\pi}{\sin(\pi (z-v))}(-u)^{z-v}\frac{g(v)}{g(z)}\frac{dz}{z-v'}.
\ee
Here $\C$ is a sufficiently small contour containing $\sigma_i$ for all $i$, and $\L$ is a vertical line separating $\C$ and $\C+1$.
\end{prop}
\begin{proof}
From Proposition \ref{preFredholmProp} we have
\be
\E\left[ e^{u\Z_{x,y}}\right]=\sum_{l\geq 0}\frac{1}{l!}\sum_{\substack{n_1, \dots, n_l\geq 1}}\oint_{\C}\frac{dv_1}{2\pi\i}\dots\oint_{\C}\frac{dv_l}{2\pi\i}\det\left[\widetilde{K}_u(v_a, n_a; v_b, n_b)\right]_{a,b=1}^l.
\ee
Due to the dominated convergence, which is valid since the contour $\mathcal C$ is compact and we have an integrable converging upper-bound \eqref{upperboundKtilde}, we can exchange the second summation and the integration. Since $\widetilde{K}_u(v, n; v', n')$ does not depend on $n'$, we can use use the same absolute upper-bound \eqref{upperboundKtilde} and multi-linearity of the determinant with respect to rows to get
\be
\sum_{\substack{n_1, \dots, n_l\geq 1}}\det\left[\widetilde{K}_u(v_a, n_a; v_b, n_b)\right]_{a,b=1}^l=\det \left[\sum_{n_a\geq 1} \frac{g(v_a)}{g(v_a+n_a)}\frac{u^{n_a}}{v_a+n_a-v_b'}\right]_{a,b=1}^l .
\ee
Then the claim follows from Proposition~\ref{MB}.
\end{proof}


\section{Limit theorem} \label{limitSect}

This section is devoted to the partial result about the Tracy-Widom large-scale limit of the partition function $\Z_{x,y}$ of the inhomogeneous Beta polymer model, when $x,y\to \infty$ at a constant ratio and the parameters of the model have finite number of values repeated with certain frequencies. Our approach follows \cite{BC15b}: we perform a steepest descent analysis of the Fredholm determinant from Proposition \ref{fredholm}.

\subsection{Overview} Fix $k$ and consider three families of parameters $\{\sigma_1,\dots, \sigma_k\}, \{\rho_1,\dots, \rho_k\}$ and $\{\omega_1, \dots, \omega_k\}$. Let $\{\alpha_1, \dots, \alpha_k\}$, $\{\beta_1, \dots, \beta_k\}, \{\gamma_1, \dots, \gamma_k\}$ be collections of corresponding ``frequencies" satisfying 
\be
\sum_{i=1}^k{\alpha_i}=\sum_{j=1}^k{\beta_j}=\sum_{d=1}^k{\gamma_d}=1, \qquad \alpha_i,\beta_j,\gamma_d\in \mathbb R_{\geq 0}.
\ee

For a fixed slope $x/y$ we consider the partition function $\Z_{\lfloor{xt}\rfloor,\lfloor{yt}\rfloor}$ as $t\to\infty$, where parameters of the model $\{\widetilde{\sigma}_0, \dots, \widetilde{\sigma}_{\lfloor{xt}\rfloor}\}$, $\{\widetilde{\rho}_1, \dots, \widetilde{\rho}_{\lfloor{yt}\rfloor}\}$, $\{\widetilde{\omega}_1, \dots, \widetilde{\omega}_{\lfloor{yt}\rfloor-\lfloor{xt}\rfloor}\}$\footnote{In earlier sections $\sigma_i, \rho_j,\omega_d$ were used to denote the parameters of individual columns, rows and diagonals. To lighten the expressions in this section it is convenient for us to change the notation by using $\widetilde{\sigma}_i,\widetilde{\rho}_j,\widetilde{\omega}_d$ for individual parameters while $\sigma_i, \rho_j,\omega_d$ denote their possible values; this abuse of notation should not cause confusion.} are defined in the following way: For the column parameters, exactly $\alpha_i^{[t]} xt $ out of $\lfloor xt\rfloor+1$ parameters $\widetilde\sigma_i$ are equal to $\sigma_i$, where $\alpha_i^{[t]}$ is a sequence satisfying
\be
\alpha_i^{[t]}-\alpha_i=O(t^{-1})\qquad \text{as}\ t\to\infty.
\ee
Similarly, the row parameters $\rho_j$ and the diagonal parameters $\omega_d$ are repeated for $\beta_j^{[t]}yt$ and $\gamma^{[t]}_d(y-x)t$ times respectively, where
\be
\beta^{[t]}_j-\beta_j=O(t^{-1}),\qquad \gamma^{[t]}_d-\gamma_d=O(t^{-1}) \qquad t\to\infty.
\ee

Motivated by the homogeneous setting \cite{BC15b} we are looking for a result of the form
\begin{equation}\label{generallimiteq}
\lim_{t\to\infty}\P\(\frac{\ln\Z_{\floor{xt},\floor{yt}}+It}{c t^{1/3}}\leq r\)=F_{GUE}(r),
\end{equation}
where $F_{GUE}(p)$ is the GUE Tracy-Widom distribution \cite{TW92}. A probabilistic interpretation, which one can find more natural, is given by identification of $\Z_{\floor{xt},\floor{yt}}$ with the distribution of a random walk in a random Beta environment. From this point of view, we are looking at large deviations of the random walk, studying fluctuations of a random rate function (depending on the random environment) from the expected limit $I(x/y)$ depending on the slope. See \cite{BC15b} for more details.   

To prove a limit relation \eqref{generallimiteq} it is enough to set $u^{[t]}=-\exp(tI-t^{1/3}cr)$ and show that
\be
\lim_{t\to\infty}\E\left[e^{u^{[t]}\Z_{\floor{xt},\floor{yt}}}\right]=F_{GUE}(r),
\ee
because, by \cite[Lemma 4.1.39]{BC11}\footnote{Following the notation of \cite[Lemma 4.1.39]{BC11}, we need to set $f(x)=e^{-e^{xct^{1/3}}}$.} we have
\be
\lim_{t\to\infty}\P\(\frac{\ln\Z_{\floor{xt},\floor{yt}}+It}{c t^{1/3}}\leq r\)=\lim_{t\to\infty}\E\left[e^{u^{[t]}\Z_{\floor{xt},\floor{yt}}}\right]
\ee
given that the limit is a continuous distribution function. So, the problem is reduced to studying the Laplace transform, which by Proposition \ref{fredholm} has an expression in terms of a Fredholm determinant. The latter can be analyzed using \emph{the steepest descent method}: rewrite the kernel $K_{u^{[t]}}$ as
\be
K_{u^{[t]}}(v,v')=\frac{1}{2\pi\i}\int_{\mathcal L}\frac{-\pi}{\sin(\pi (z-v))}\exp\(t(h^{[t]}(z)-h^{[t]}(v))- crt^{1/3}(z-v)\)\frac{dz}{z-v'},
\ee
where
\be
h^{[t]}(z)=Iz-\sum_{i=1}^kx\alpha^{[t]}_i\ln\Gamma(z-\sigma_i)+\sum_{i=1}^ky\beta_i^{[t]}\ln\Gamma(z-\rho_i)-\sum_{i=1}^k(y-x)\gamma_i^{[t]}\ln\Gamma(z-\omega_i).
\ee
As $t\to\infty$ the function $h^{[t]}(z)$ converges to the function
\be
h(z)=Iz-\sum_ix\alpha_i\ln\Gamma(z-\sigma_i)+\sum_iy\beta_i\ln\Gamma(z-\rho_i)-\sum_i(y-x)\gamma_i\ln\Gamma(z-\omega_i),
\ee
which, for the right choice of $I$, has a second order critical point $\theta$ depending on the slope $y/x$ in a nontrivial way. However, the dependence of the slope $x/y$ on $\theta$ can be easily described, suggesting the following parametrization by $\theta\in(\max_i\sigma_i,\infty)$:
\begin{equation}
\label{thetaparametrizationcollection}
\begin{gathered}
x=x(\theta)=\sum_i\beta_i\Psi_1(\theta-\rho_i)-\sum_i\gamma_i\Psi_1(\theta-\omega_i),\\
y=y(\theta)=\sum_i\alpha_i\Psi_1(\theta-\sigma_i)-\sum_i\gamma_i\Psi_1(\theta-\omega_i),\\
I=I(\theta)=x(\theta)\sum_i\alpha_i\Psi(\theta-\sigma_i)-y(\theta)\sum_i\beta_i\Psi(\theta-\rho_i)+ (y(\theta)-x(\theta))\sum_i\gamma_i\Psi(\theta-\omega_i),\\
c^3=c(\theta)^3=-\frac{x(\theta)}{2}\sum_i\alpha_i\Psi_2(\theta-\sigma_i)+\frac{y(\theta)}{2}\sum_i\beta_i\Psi_2(\theta-\rho_i)-\frac{y(\theta)-x(\theta)}{2}\sum_i\gamma_i\Psi_2(\theta-\omega_i),
\end{gathered}
\end{equation}
where $\Psi(z)$ denotes the digamma function and $\Psi_k(z)$ denotes the $k$th polygamma function:
\be
\Psi(z)=\frac{d}{dz}\ln\Gamma(z),\qquad \Psi_{k}(z)=\frac{d^{k+1}}{dz^{k+1}}\ln\Gamma(z).
\ee

\begin{prop}\label{goodParam}
1) The function $x(\theta)/y(\theta)$ is increasing in $\theta\in(\max_i \sigma_i, \infty)$, with the limit values
\be
\lim_{\theta\to\max_i\sigma_i}\frac{x(\theta)}{y(\theta)}=0,\qquad \lim_{\theta\to\infty}\frac{x(\theta)}{y(\theta)}=\frac{\sum_i\beta_i\rho_i-\sum_{i}\gamma_i\omega_i}{\sum_i\alpha_i\sigma_i-\sum_i\gamma_i\omega_i}.
\ee

2) $c(\theta)$ can be chosen to be positive.
\end{prop}
Proposition \ref{goodParam} is proved in Appendix \ref{app:num}.

With this parametrization we have $h(z)-h(\theta)=\frac{c^3}{3}(z-\theta)^3 +O((z-\theta)^4)$ for $z$ close to $\theta$, and the general philosophy of the steepest descent method suggests that asymptotical behavior of the Fredholm determinant is dominated by its behavior in a small neighborhood of $\theta$, which, after certain manipulation, can be shown to produce $F_{GUE}(r)$. 

\begin{conj}\label{conjecture} For any $\theta\in(\max(\sigma), \infty)$ we have
\be
\lim_{t\to\infty}\P\(\frac{\ln\Z_{\floor{xt},\floor{yt}}+It}{c t^{1/3}}\leq r\)=F_{GUE}(r),
\ee
with $x=x(\theta), y=y(\theta), I=I(\theta), c=c(\theta)$ defined above in terms of $\theta$.
\end{conj}

Unfortunately, the general case turns out to be technically challenging, and we do not pursue it here. In this section we prove Conjecture \ref{conjecture} for parameters $\sigma_i,\rho_i,\omega_i,\theta$ satisfying the following restrictions:
\begin{ass}\label{assumptions}
We assume that
\begin{itemize}
\item All $\sigma_i$ are equal to $0$, while all $\rho_i$ are equal to $-1$:
\be
\sigma_1=\sigma_2=\dots=\sigma_k=0,\qquad \rho_1=\rho_2=\dots=\rho_k=-1.
\ee
\item $\theta\in \(0, \frac{1}{2}\)$.
\end{itemize}
\end{ass}

\begin{theo}\label{mainBetaResult} Under Assumption \ref{assumptions} we have 
\be
\lim_{t\to\infty}\P\(\frac{\ln\Z_{\floor{xt},\floor{yt}}+It}{c t^{1/3}}\leq r\)=F_{GUE}(r).
\ee
\end{theo}

\begin{rem}\normalfont
Theorem \ref{mainBetaResult} generalizes \cite[Theorem 5.2]{BC15b} by adding a degree of freedom corresponding to the diagonal parameters $\omega_i$, while in \cite{BC15b} they were additionally assumed to be equal to $-2$. Returning to the Beta polymer model, Theorem \ref{mainBetaResult} corresponds to the model with Beta distributions $\Beta(1,-\widetilde{\omega}-1)$ with parameters $\widetilde{\omega}<-1$ constant along the diagonals and having finitely many possible values $\omega_1, \dots, \omega_k$. 
\end{rem}

The restrictions from Assumption \ref{assumptions} come from two technical difficulties. Namely, to justify the steepest descent argument one needs to
\begin{itemize}
\item Find contours $\C_\theta,\L_\theta$ passing through $\theta$ and satisfying $\Re[h(z)-h(v)]\leq 0$ for $z\in \L_\theta, v\in\C_\theta$, with equality only if $v=z=\theta$;
\item Prove that one can deform the initial contours $\C,\L$ to $\C_\theta,\L_\theta$ without changing the limit of the Fredholm determinant.
\end{itemize}
For the first step a suitable descent contour $\L_{\theta}$ for $\Re[h(z)]$ valid for any parameters $\sigma_i,\rho_i,\omega_i$ was found in \cite{BC15b} and it is given by a vertical line passing through $\theta$. At the same time, finding a descent contour for $-\Re[h(v)]$ is much more involved for general parameters: we know how to find it only if $\sigma_i-\rho_j=1$ for any $i,j$ and $\theta-\max_i\sigma_i<1$, forcing the first part of Assumption \ref{assumptions}\footnote{To simplify the assumptions we have also used the invariance of the model under simultaneous additive shifts of all the parameters $\sigma_i,\rho_i,\omega_i$.}. Under these restrictions another contour from \cite{BC15b} turns out to be sufficient, we denote it by $\C_\theta$  and it is equal to a circle around $0$ of radius $\theta$. However, numerical computations and related results suggest that a descent contour for $-\Re[h(v)]$ should exist in general: see \cite{BR19} for an example of a significantly more involved argument used to lift the analogous assumption in a different but related setting.

The second step is complicated by potential singularities of the Fredholm determinant along the deformation of $\C, \L$ to $\C_\theta,\L_\theta$. One can hope to estimate the residues at these singularities and show that their contribution asymptotically vanishes, but the required argument is technically challenging. The second part of Assumption \ref{assumptions} allows us to avoid this issue, since for $\theta<\frac{1}{2}$ the contours can be deformed without crossing any singularities of the relevant integrands.

The remainder of Section \ref{limitSect} is devoted to the proof of Theorem \ref{mainBetaResult}. As clarified above, it is enough to prove
\begin{equation}
\label{laplacelimiteq}
\lim_{t\to\infty}\E\left[e^{u^{[t]}\Z_{\floor{xt},\floor{yt}}}\right]=F_{GUE}(r),
\end{equation}
where $u^{[t]}=-\exp(tI-t^{1/3}cr)$. The proof is structured as follows: first we describe the needed properties of the function $h(z)$, as well as specify estimates on convergence of $h^{(t)}(z)$ to $h(z)$. Then we prove various bounds on the kernel $K_u$ from Proposition \ref{fredholm}. Finally, to prove Theorem~\ref{mainBetaResult} we reduce the computation of the limit to a small neighborhood of $\theta$, perform a change of variables to get rid of lower degree terms and use standard algebraic manipulations to establish \eqref{laplacelimiteq}.

\subsection{$h$-functions and descent contours} Here we list the needed properties of the functions $h(z)$ and $h^{[t]}(z)$, which are defined by
\be
h^{[t]}(z)=Iz-\sum_{i=1}^kx\alpha^{[t]}_i\ln\Gamma(z-\sigma_i)+\sum_{i=1}^ky\beta_i^{[t]}\ln\Gamma(z-\rho_i)-\sum_{i=1}^k(y-x)\gamma_i^{[t]}\ln\Gamma(z-\omega_i),
\ee
\be
h(z)=Iz-\sum_{i=1}^kx\alpha_i\ln\Gamma(z-\sigma_i)+\sum_{i=1}^ky\beta_i\ln\Gamma(z-\rho_i)-\sum_{i=1}^k(y-x)\gamma_i\ln\Gamma(z-\omega_i),
\ee
with $x,y,I$ defined in terms of $\theta$ as before.

We start with the properties of $h(z)$. Define the contour $\L_\theta$ as a vertical line passing through $\theta$, oriented upwards:
\be
\L_\theta=\{\theta+b\i\mid b\in \mathbb R\}.
\ee
Under Assumption \ref{assumptions} we also define the contour
\be
\C_\theta=\{\theta e^{i\phi}\mid\phi\in [-\pi,\pi]\}
\ee 
oriented counter-clockwise. See Figure \ref{BigContoursFigure} for a depiction of the contours.

\begin{figure}
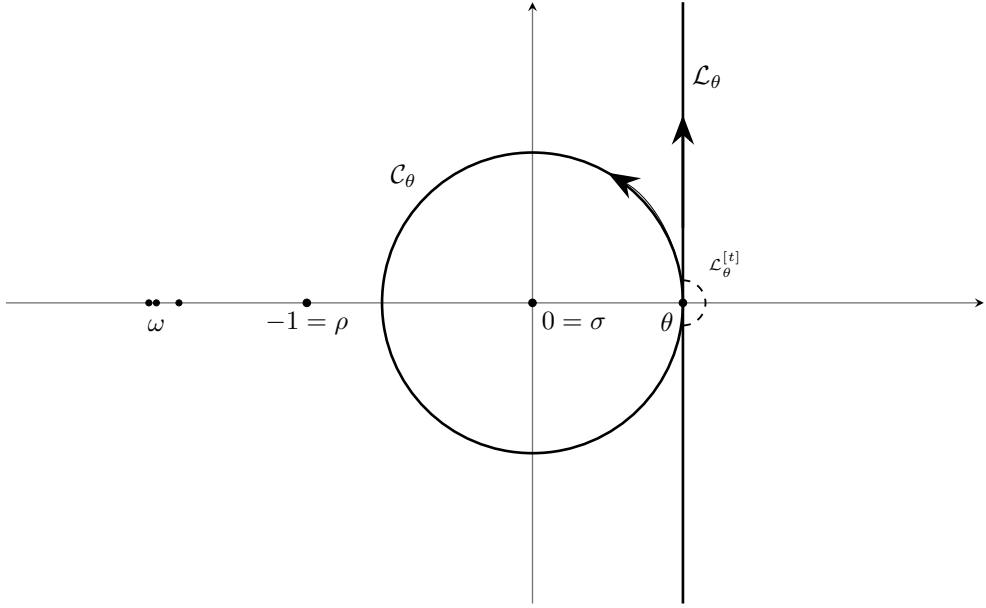

\centerline{
\tikz{1}{

\draw[line width=0.5pt, draw=dgray, arrows={->[scale=4]}] (-7,0) -- (6,0);
\draw[line width=0.5pt, draw=dgray, arrows={->[scale=4]}] (0,-4) -- (0,4);

\draw[fill=black] (0,0) circle (0.05);
\node[below right] at (0,0) {$0=\sigma$};

\draw[fill=black] (-3,0) circle (0.05);
\node[below] at (-3,0) {$-1=\rho$};

\draw[fill=black] (2,0) circle (0.05);
\node[below left] at (2,0) {$\theta$};


\draw[fill=black] (-5,0) circle (0.04);
\draw[fill=black] (-4.7,0) circle (0.04);
\draw[fill=black] (-5.1,0) circle (0.04);
\node[below] at (-5,-0.1) {$\omega$};

\draw[line width=1pt, draw=black] (0,0) circle (2);
\draw[line width=0pt, draw=black, arrows={-Stealth[scale=4]}] (2,0) arc (0:60:2);
\node[above left] at (135:2) {$\C_\theta$};

\draw[line width=0.7pt, dashed,  draw=black] (2,0) +(-90:0.3) arc (-90:45:0.3) node[above right] {\tiny $\L^{[t]}_\theta$} arc (45:90:0.3);

\draw[line width=1pt, arrows={-Stealth[scale=1.5]}] (2,-4) -- (2,2.5);
\draw[line width=1pt] (2,1) -- (2,4);
\node[right] at (2, 3) {$\L_\theta$};
}}
\caption{\label{BigContoursFigure} Contours $\C_\theta$ and $\L_\theta$, as well as the $t$-dependent deformation $\L^{[t]}_\theta$ needed to avoid singularity at $z=\theta$.  }
\end{figure}

The proofs of the following three statements are technical and are omitted from this text, they can be found in Appendix \ref{app:num}. The first statement does not require  Assumption~\ref{assumptions}.

\begin{prop}\label{LcontProp}
$\L_\theta$ is a steep descent contour for $\Re[h(z)]$, that is, $\Re[h(\theta+bi)]$ decreases for positive $b$ and increases for negative $b$.
\end{prop}
\begin{prop}\label{CcontProp}
Under Assumption \ref{assumptions} $\mathcal C_\theta$ is a steep descent contour for $-\Re[h(z)]$, that is, $-\Re[h(\theta e^{\phi \i})]$ decreases for $\phi>0$ and increases for $\phi<0$.
\end{prop}
\begin{prop}\label{AcontProp} Suppose that Assumption \ref{assumptions} is satisfied. Let $\psi(\ve)=\arg(v_{\ve}-\theta)$, where $v_{\ve}$ is the intersection point of $\mathcal C_\theta$ with the circle of radius $\ve$ around $\theta$ satisfying $\Im[v_{\ve}]>0$. Then for sufficiently small $\ve$ there exists a constant $b$ such that
\be
\Re\left[h(\theta+\ve e^{\pm i\phi})-h(\theta)\right]>b, \qquad \phi\in[\psi(\ve), 2\pi/3].
\ee
\end{prop}

The next statement reiterates the fact that $\theta$ is a critical point of $h(z)$.
\begin{lem}\label{nbhlemma}
For $z$ in a neighborhood of $\theta$ we have
\be
h(z)-h(\theta)=\frac{c^3}{3}(z-\theta)^3+O((z-\theta)^4),
\ee
where $c$ is given by \eqref{thetaparametrizationcollection}. Moreover, $c>0$.
\end{lem}
\begin{proof} Readily follows by considering the Taylor expansion of $h(z)$, taking derivatives and using the definitions of $x(\theta), y(\theta), I(\theta)$ and $c(\theta)$. The positivity of $c$ follows from Proposition \ref{goodParam}.
\end{proof}

Finally, we need the following statements regarding the convergence of $h^{[t]}(z)$ to $h(z)$ in various situations.

\begin{lem}
\label{htsmallbound}
Assume that a sequence of points $z_t\in\mathbb C$ converges to $\theta$ in a way that $h^{[t]}(z_t)$ is well defined for all $t>0$. Then
\be
\lim_{t\to\infty} t\(h^{[t]}(z_t)-h^{[t]}(\theta)\) = \lim_{t\to\infty} t\(h(z_t)-h(\theta)\),
\ee
assuming that the right-hand side exists.
\end{lem}
\begin{proof} 
Note that
\begin{multline*}
th^{[t]}(z)-th(z)=-x\sum_{i=1}^k t(\alpha_i^{[t]}-\alpha_i)\ln\Gamma(z-\sigma_i)\\
+y\sum_{i=1}^kt(\beta_i^{[t]}-\beta_i)\ln\Gamma(z-\rho_i)-(y-x)\sum_{i=1}^kt(\gamma_i^{[t]}-\gamma_i)\ln\Gamma(z-\omega_i)
\end{multline*}
Since $\alpha^{[t]}_i-\alpha_i=O(1/t)$ and $\ln\Gamma(z-\sigma_i)$ is continuous around $\theta$, we have
\be
\lim_{t\to \infty} t(\alpha_i^{[t]}-\alpha_i)\(\ln\Gamma(z_t-\sigma_i)-\ln\Gamma(\theta-\sigma_i)\)=0.
\ee
Taking sum over $i$ and repeating the same argument for $\beta_i,\beta_i^{[t]},\rho_i$ and $\gamma_i,\gamma_i^{[t]},\omega_i$ we obtain
\be
\lim_{t\to\infty} t\(h^{[t]}(z_t)-h(z_t)\)- t\(h^{[t]}(\theta)-h(\theta)\)=0,
\ee
which implies the claim.
\end{proof}
\begin{lem}
\label{htbound}
Let $K\subset \mathbb C$ be a compact subset not containing any of the parameters $\sigma_i,\rho_i,\omega_i$. Then for any $z\in K$ and $t>0$ we have
\be
|th^{[t]}(z)-th(z)|<\Lambda
\ee
for some constant $\Lambda>0$ depending only on the compact set $K$ but not on $t$.
\end{lem}
\begin{proof}
Following the lines of the proof of Lemma \ref{htsmallbound}, since $\alpha^{[t]}_i-\alpha_i=O(1/t), \beta^{[t]}_i-\beta_i=O(1/t), \gamma^{[t]}_i-\gamma_i=O(1/t)$ there exists a constant $\Lambda_1>0$ such that
\be
t(\alpha^{[t]}_i-\alpha_i), t(\beta^{[t]}_i-\beta_i), t(\gamma^{[t]}_i-\gamma_i)<\Lambda_1, \qquad i=1,\dots, k.
\ee
At the same time, since $K$ is bounded away from $\sigma_i,\rho_i,\omega_i$, there exists another constant $\Lambda_2>0$ such that
\be
|\ln\Gamma(z-\sigma_i)|,|\ln\Gamma(z-\rho_i)|,|\ln\Gamma(z-\omega_i)|<\Lambda_2,  \qquad z\in K,\  i=1,\dots, k.
\ee
Combining these upper bounds concludes the proof.
\end{proof}

\begin{lem}
\label{htbigbound}
For sufficiently large $z$ such that $\arg(z)\in [-\pi+\delta,\pi-\delta]$ for a fixed $\delta$ we have
\be
\exp(th^{[t]}(z))=\exp(th(z))\frac{z^{-\Delta^{[t]}}}{\Gamma(z)}(1+O\(z^{-1}\)),
\ee
where the constants of $O(z^{-1})$ can be chosen independently of $t$, and $\Delta^{[t]}$ is given by
\be
\Delta^{[t]}=-\sum_i xt(\alpha_i^{[t]}-\alpha_i)\sigma_i+\sum_iyt(\beta_i^{[t]}-\beta_i)\rho_i-\sum_it(y-x)(\gamma_i^{[t]}-\gamma_i)\omega_i.
\ee
\end{lem}
Note that, due to our conventions on $\alpha_i^{[t]}, \beta_i^{[t]}$ and $\gamma_i^{[t]}$, we have a uniform in $t$ bound $|\Delta^{[t]}|<\Lambda$ for a constant $\Lambda>0$ independent of $t$.
\begin{proof}
By taking logarithm the claim can be rephrased as
\begin{multline*}
-\sum_i xt(\alpha_i^{[t]}-\alpha_i)\ln\Gamma(z-\sigma_i)+\sum_iyt(\beta_i^{[t]}-\beta_i)\ln\Gamma(z-\rho_i)-\sum_it(y-x)(\gamma_i^{[t]}-\gamma_i)\ln\Gamma(z-\omega_i)\\
=-\Delta^{[t]}\ln(z)-\ln\Gamma(z)+O(z^{-1}).
\end{multline*}
Now we can use the asymptotic expansion for Gamma function, used before in Lemma \ref{assymptg}:  with the same restrictions on $z$ as in the claim and any fixed $a\in\mathbb C$ we have
\be
\ln\Gamma(z-a)=\(z-a-\frac{1}{2}\)\ln z -z + \frac{1}{2}\ln(2\pi)+O(z^{-1}).
\ee
Recall that $\sum_{i=1}^k\alpha_i=\sum_{i=1}^k\beta_i=\sum_{i=1}^k\gamma_i=1$, while for $\alpha^{[t]}_i,\beta^{[t]}_i,\gamma^{[t]}_i$ we have
\be
\sum_{i=1}^k xt\alpha^{[t]}_i=\floor{xt}+1,\quad \sum_{i=1}^k yt\beta^{[t]}_i=\floor{yt}, \quad \sum_{i=1}^k (y-x)t\gamma^{[t]}_i=\floor{yt}-\floor{xt},
\ee
with $1$ in the first equality coming from having $\floor{xt}+1$ column parameters $\{\widetilde{\sigma}_i\}_{i=0,\dots, \floor{xt}}$. Thus
\be
-\sum_{i=1}^k tx(\alpha_i^{[t]}-\alpha_i)+\sum_{i=1}^kty(\beta_i^{[t]}-\beta_i)-\sum_{i=1}^kt(y-x)(\gamma_i^{[t]}-\gamma_i)=-1.
\ee
Combining the identity above with the asymptotic expansion for the Gamma function, we obtain
\begin{multline*}
-\sum_i xt(\alpha_i^{[t]}-\alpha_i)\ln\Gamma(z-\sigma_i)+\sum_iyt(\beta_i^{[t]}-\beta_i)\ln\Gamma(z-\rho_i)-\sum_it(y-x)(\gamma_i^{[t]}-\gamma_i)\ln\Gamma(z-\omega_i)\\
=\(-z-\Delta^{[t]}+\frac{1}{2}\)\ln z + z -\frac{1}{2}\ln(2\pi)+2y\Lambda_1O(z^{-1}), 
\end{multline*}
where $\Lambda_1$ is the same constant as in the proof of Lemma \ref{htbound}. Applying the asymptotic expansion for the Gamma function to the right-hand side implies the claim.
\end{proof}

\subsection{Kernel estimates}
Recall that the substitution of $u=-\exp(tI-t^{1/3}cr)$ into the integral kernel from Proposition \ref{fredholm} leads to the following kernel:
\be
K^{[t]}(v,v')=\frac{1}{2\pi\i}\int_{\mathcal L_\theta}\frac{-\pi}{\sin(\pi (z-v))}\exp\(t(h^{[t]}(z)-h^{[t]}(v))- crt^{1/3}(z-v)\)\frac{dz}{z-v'}.
\ee
For $v,v'=\theta$ we interpret the integral above using its principal value as the contour $\mathcal L_\theta$ passes to the right of $\theta$. In other words, we slightly deform the contour $\L_\theta$  to an upwards oriented contour $\L_\theta^{[t]}$ consisting of vertical lines $\{\theta\pm Y\i\mid Y\in[t^{-1/3},\infty)\}$ and a semicircle $\{\theta+t^{-1/3}e^{\phi\i}\mid \phi\in [-\pi,\pi]\}$, see Figure \ref{BigContoursFigure}.

The aim of this section is to give asymptotical bounds on behavior of the kernel. We use the following notation. For $\ve>0$ let $\L_{\theta,\ve}^{[t]}$ and $\C_{\theta,\ve}$ denote the parts of the contours $\L_{\theta}^{[t]}$ and $\C_\theta$ contained inside the $\ve$-neighborhood of $\theta$. Define the kernel $K_{\ve}^{[t]}$ by
\be
K_{\ve}^{[t]}(v,v')=\frac{1}{2\pi\i}\int_{\mathcal L_{\theta,\ve}^{[t]}}\frac{-\pi}{\sin(\pi (z-v))}\exp\(t(h^{[t]}(z)-h^{[t]}(v))- crt^{1/3}(z-v)\)\frac{dz}{z-v'}.
\ee

Finally, define $\mathcal{K}(\ve)=\D_\ve\cup\mathcal A_\ve\cup (\C_\theta\backslash \C_{\theta,\ve})$ as a contour consisting of the upward-oriented wedge
\be
\D_\ve=\{\theta+Re^{\pm 2\pi\i/3}\mid R\in [0,\ve]\}, 
\ee
a part of the contour $\C_\theta$ outside of the $\ve$-neighborhood of $\theta$, which we denote as $\C_{\theta,\geq \ve}$, and a couple of upward-oriented arcs 
\be
\mathcal A_\ve=\{\theta+\ve e^{\phi\i}\mid\phi\in [\psi(\ve), 2\pi/3]\},
\ee
where $\psi(\ve)$ is chosen so that the result is connected. See Figure \ref{SmallContoursFigure} for a depiction of the contour $\mathcal{K}(\ve)$.

For the following statements we fix $\ve>0$.

\begin{lem} Let $v,v'\in \K(\ve)$ be a pair of points such that $\Re[h(v)-h(\theta)]\geq b$ for some fixed $b\geq 0$. Then for $t>\ve^{-3}$ we have
\label{Kreducelem}
\be
|K^{[t]}(v,v')-K^{[t]}_\ve(v,v')|<\Lambda e^{-t(a+b)}, 
\ee
where the constants $\Lambda, a>0$ depend only on $\ve$ and the parameters of the model.
\end{lem}
\begin{lem} Let $v,v'\in \K(\ve)$ be a pair of points such that $\Re[h(v)-h(\theta)]\geq b$ for a fixed $b\geq 0$. Then for $t>\ve^{-3}$ we have
\label{Kvboundlem}
\be
|K^{[t]}(v,v')|<\Lambda t^{1/3} e^{-bt},
\ee
where $\Lambda>0$ is a constant depending only on $\ve$ and $b$, but not on $v,v'$. In particular, if $b> 0$ we also have
\be
|K^{[t]}(v,v')|<\Lambda' e^{-bt/2}
\ee
for some other constant $\Lambda'>0$ depending on $\ve$ and $b$.
\end{lem}
\begin{proof}[Proof of Lemma \ref{Kreducelem}] The assumption $t>\ve^{-3}$ is needed only to ensure that the semicircle of $\L^{[t]}_\theta$ is inside the $\ve$ neighborhood. Then we need to bound the integral
\be
\int_{\mathcal L_\theta\backslash \{|z|<\ve\}}\frac{-\pi}{\sin(\pi (z-v))}\exp\(t(h^{[t]}(z)-h^{[t]}(v))- crt^{1/3}(z-v)\)\frac{dz}{z-v'}
\ee
taken over two vertical lines $\{\theta\pm Y\i|Y>\ve\}$. We denote this contour by $\L_{\theta,\geq \ve}$

Note that the contour $\K(\ve)$ is bounded away from $\L_{\theta,\geq\ve}+n$ for any $n\in\mathbb Z$, hence we have
\be
\left|\frac{-\pi}{\sin\pi(z-v)}\frac{1}{z-v'}\right|<\Lambda_1 e^{-\pi |\Im[z]|}, \qquad  z\in\L_{\theta,\geq\ve},
\ee
where $\Lambda_1>0$ is a constant depending only on $\ve$ but not on $v,v'$. So it is enough to give a sufficiently strong upper-bound on $|\exp(t(h^{[t]}(z)-h^{[t]}(v))-crt^{1/3}(z-v))|$. Due to the steep descent from Proposition \ref{LcontProp}, for some constant $a>0$ we have
\be
\Re[h(z)-h(v)]<-2a-b, \qquad z\in\mathcal L_{\theta,\geq \ve}
\ee 
But the kernel $K^{[t]}$ is defined in terms of $h^{[t]}(z)$ rather than $h(z)$ and to proceed we have to split the problem in two parts, with $|z|$ large and $z$ compactly supported.

Since $\K(\ve)$ is compact, by Proposition \ref{htbound} we have
\be
\Re[h(z)-h^{[t]}(v)]<-2a-b+\frac{\ln\Lambda_3}{t},\qquad z\in\mathcal L^{[t]}_\theta\backslash \L^{[t]}_{\theta,\ve},
\ee 
where the constant $\Lambda_3>1$ does not depend on $v$. On the other hand, by Lemma \ref{htbigbound} and the asymptotical expansion for the Gamma function \eqref{Gammaass}, for sufficiently large $M>\ve$ independent of $t,\ve$ we have
\be
\left|\exp\(t(h^{[t]}(z)\)\right|<\left|\exp\(t(h(z))\)\frac{2z^{-\Delta^{[t]}}}{\sqrt{2\pi}z^{z-1/2}e^{-z}}\right|, \qquad |z|>M, z\in\L_\theta.
\ee 
Recall that $\Delta^{[t]}$ is uniformly bounded, so set $\Delta=\inf_{t}\Delta^{[t]}>-\infty$.  Then, assuming $M>1$, for some constant $\Lambda_2'$ and any $z\in\mathcal L_\theta, |z|>M, v\in \K(\ve)$ we obtain
\be
|\exp(t(h^{[t]}(z)-h^{[t]}(v))-crt^{1/3}(z-v))|<\frac{2\Lambda_3|z|^{-\Delta}\exp(-2ta-tb-crt^{1/3}\Re[z-v])}{\sqrt{2\pi}|z^{z-1/2}|e^{-\Re[z]}}<\Lambda_2' e^{-(a+b)t}e^{\frac{2\pi}{3}|\Im[z]|},
\ee
where we have used that $|z^z|>|z|^\theta e^{-\frac{\pi}{2}|\Im[z]|}$ and $|z|<\theta+|\Im(z)|$ for $z\in\L_\theta$.

At the same time, having fixed $M$ from the previous part, for $z\in\mathcal L_{\theta,\geq \ve}, |z|<M$ we can again apply Lemma \ref{htbound} to get
\be
\Re[h^{[t]}(z)-h^{[t]}(v)]<-2a-b +\frac{\ln\Lambda_4}{t},
\ee 
for $\Lambda_4>1$. Hence, for some constant $\Lambda_2''>0$ depending only on $\ve,M$ we have
\be
|\exp(t(h^{[t]}(z)-h^{[t]}(v))-crt^{1/3}(z-v))|<\Lambda_2'' e^{-(a+b)t}, \qquad z\in\L_{\theta,\geq \ve}, |z|<M.
\ee
Setting $\Lambda_2=\max(\Lambda_2',\Lambda_2'')$ and gathering all the pieces together, we get 
\be
\left|\int_{\substack{\L_{\theta,\geq\ve}}}\frac{-\pi}{\sin(\pi (z-v))}\exp\(t(h^{[t]}(z)-h^{[t]}(v))- crt^{1/3}(z-v)\)\frac{dz}{z-v'}\right|<\Lambda_1\Lambda_2e^{-(a+b)t}\int_{\substack{\mathcal L_{\theta,\geq\ve}}}e^{-\frac{\pi}{3}|\Im[z]|}|dz|
\ee
where $\Lambda_1,\Lambda_2$ depend only on $\ve$, but not on $v,v',t$.
\end{proof}
\begin{proof}[Proof of Lemma \ref{Kvboundlem}]
By Lemma \ref{Kreducelem} it is enough to prove that
\be
|K_{\ve}^{[t]}(v,v')|<\Lambda t^{1/3} e^{-bt}.
\ee

To bound the integral $K_{\ve}^{[t]}(v,v')$ we deform the integration contour $\mathcal L^{[t]}_{\theta,\ve}$ further, changing it to the contour $\mathcal P$ consisting of the line segments $\mathcal P_1=\{\theta+Y\i\mid y\in [-\ve, -\delta t^{-\frac{1}{3}}]\cup[\delta t^{-\frac{1}{3}},\ve]\}$, and the semicircle $\mathcal P_2=\{\theta+ \delta t^{-1/3}e^{\phi \i}\mid\phi\in[-\pi/2,\pi/2]\}$, where the parameter $\delta>0$ possibly depends on $\ve$ and is chosen small enough so that $\K(\ve)$ is bounded away from $\mathcal P+n$ uniformly in $t>1$ for $n\neq 0$.\footnote{Basically, we reduce the semicircle part of $\mathcal L^{[t]}_{\theta,\ve}$ avoiding additional singularities of $\frac{1}{\sin(\pi(z-v))}$. More precisely, we want $\K(\ve)$ to be between $\mathcal P-1$ and $\mathcal P$. } Then there exists a constant $\Lambda_1$ depending only on $\ve,\delta$, such that
\be
\left |\frac{\pi}{\sin\pi(z-v)}\frac{1}{z-v'}\right|<\Lambda_1\frac{1}{|z-v||z-v'|}, \qquad v,v'\in \K(\ve),\ z\in \mathcal P
\ee
which implies, due to the geometry of $\K(\ve)$, existence of uniform in $v,v',t$ constants $\Lambda_{1,\mathcal P_1}$ and $\Lambda_{1,\mathcal P_2}$ such that
\begin{equation}
\label{imboundlemtmp}
\left |\frac{\pi}{\sin\pi(z-v)}\frac{1}{z-v'}\right|<\Lambda_{1,\mathcal P_1}\frac{1}{|\Im[z]|^2},\qquad z\in \mathcal P_1,
\end{equation}
\be
\left |\frac{\pi}{\sin\pi(z-v)}\frac{1}{z-v'}\right|<\Lambda_{1,\mathcal P_2}t^{2/3},\qquad z\in \mathcal P_2.
\ee 
To bound the exponential term in the integrand of $K^{[t]}_{\ve}$ we use Lemma \ref{htbound}, obtaining
\be
|\exp(t(h^{[t]}(z)-h^{[t]}(v))-crt^{1/3}(z-v))|<\Lambda_2 \frac{|\exp(t(h(z)-h(\theta))-crt^{1/3}(z-\theta))|}{|\exp(t(h(v)-h(\theta))-crt^{1/3}(v-\theta))|}
\ee
for a constant $\Lambda_2$ depending only on $\ve$. Since $\Re[v]<\theta$ and $\Re[h(v)-h(\theta)]\geq b$, for any $t$ we have 
\be
\frac{|\exp(t(h(z)-h(\theta))-crt^{1/3}(z-\theta))|}{|\exp(t(h(v)-h(\theta))-crt^{1/3}(v-\theta))|}\leq e^{-bt} |\exp(t(h(z)-h(\theta))-crt^{1/3}(z-\theta))|.
\ee

To finish the upper-bounds we have to consider the contours $\mathcal P_1$ and $\mathcal P_2$ separately, starting with $\mathcal P_1$. Since $\mathcal L_\theta$ is a steep descent contour for $\Re[h(z)]$ by Proposition \ref{LcontProp}, for $z\in\theta+\mathbb R\i$ we have
\be
\exp(t(h(z)-h(\theta))-rct^{1/3}(z-\theta))|\leq 1.
\ee
Then the whole integral is bounded as follows:
\begin{multline*}
\left |\int_{\mathcal P_1}\frac{-\pi}{\sin(\pi (z-v))}\exp\(t(h^{[t]}(z)-h^{[t]}(v))- cst^{1/3}(z-v)\)\frac{dz}{z-v'}\right|\\
\leq \Lambda_{1,\mathcal P_1}\Lambda_2e^{-bt}\int_{y\in [-\ve, -\delta t^{-\frac{1}{3}}]\cup[\delta t^{-\frac{1}{3}},\ve]}\frac{dy}{y^2}\leq \Lambda_3 t^{1/3}e^{-bt},
\end{multline*}
for some constant $\Lambda_3$ depending only on $\ve$.

To upper-bound the part taken over $\mathcal P_2$ we use the Taylor expansion, finding a constant $\Xi>0$ such that
\be
\left|h(z)-h(\theta) - \frac{c^3}{3}(z-\theta)^3\right|<\Xi|z-\theta|^4.
\ee
Plugging $z=\theta+\delta t^{-1/3}e^{i\phi}$ we get
\be
\exp(t(h(z)-h(\theta))-t^{1/3}cs(z-\theta))|\leq \exp\( \frac{c^3\delta^3}{3}\cos(3\phi)+\delta^4\Xi t^{-1/3}\cos(4\phi)-\delta rc\cos(\phi)\)\leq\Lambda_4
\ee
where $\Lambda_4$ is a constant depending on $\ve$. Gathering the bounds above and using that the length of $\mathcal P_2$ is equal to $\pi t^{-1/3}$ we get
\be
\left |\int_{\mathcal P_2}\frac{-\pi}{\sin(\pi (z-v))}\exp\(t(h^{[t]}(z)-h^{[t]}(v))- crt^{1/3}(z-v)\)\frac{dz}{z-v'}\right|<\Lambda_{1,\mathcal P_2}\Lambda_4t^{2/3}e^{-bt}\ \pi t^{-1/3}<\Lambda_5 t^{1/3}e^{-bt}.
\ee 
The claim follows by combining the bounds for $\mathcal P_1$ and $\mathcal P_2$, .
\end{proof}

\begin{figure}
\centerline{
\tikz{0.9}{

\draw[lgray, line width=1pt, arrows={-Stealth[scale=1.5]}] (0,-5) -- (0,3.6);
\draw[lgray, line width=1pt] (0,1) -- (0,5);

\draw[lgray, line width=1pt] (0,0) arc (0:23:10);
\draw[lgray, line width=1pt] (0,0) arc (0:-23:10);
\draw[line width=0pt, draw=lgray, arrows={-Stealth[scale=4]}] (0,0) arc (0:20:10);

\draw[line width=0.5pt, draw=dgray, arrows={->[scale=4]}] (-6,0) -- (6,0);

\draw[line width=1pt, dashed] (0,0) circle (4);

\draw[fill=black] (0,0) circle (0.05);
\node[below right] at (0,0) {$\theta$};

\draw[fill=black] (1,0) circle (0.05);
\node[below right] at (1,0) {$\theta+t^{-1/3}$};

\draw[fill=black] (4,0) circle (0.05);
\node[below right] at (4,0) {$\theta+\ve$};

\draw[draw=black, line width=1.5pt, arrows={-Stealth[scale=1.5]}] (0, 1) -- (0,3.6);
\draw[draw=black, line width=1.5pt] (0, 1) -- (0,5);
\draw[draw=black, line width=1.5pt] (0, -5) -- (0,-1);
\draw[line width=1.5pt,  draw=black] (0,0) +(-90:1) arc (-90:45:1) node[above right] {$\L^{[t]}_\theta$} arc (45:90:1);

\draw[black, line width=1.5pt, arrows={-Stealth[scale=1.5]}] (0,0) -- (120:3); 
\draw[black, line width=1.5pt] (0,0) -- (120:2) node[below left] {$\D_\ve$} -- (120:4) arc (120: 110:4) node[above left] {$\mathcal{A}_\ve$} arc (110: 101.5:4) arc (23:30:10);
\draw[black, line width=1.5pt] (0,0) -- (-120:4) arc (-120: -110:4) node[below left] {$\mathcal{A}_\ve$} arc (-110: -101.5:4) arc (-23:-30:10);

}}
\caption{\label{SmallContoursFigure} The contours $\D_\ve,\mathcal A_\ve, \K(\ve)$ and $\L^{[t]}_\ve$ in an $\ve$-neighborhood of  $\theta$. }
\end{figure}

\subsection{Proof of Theorem \ref{mainBetaResult}} Recall that, in view of Proposition \ref{fredholm}, it is enough to prove
\be
\lim_{t\to\infty}\det(I+K^{[t]})_{L^2(\C_\theta)}=F_{GUE}(r),
\ee
where we use the notation from the previous subsection. 

{\bfseries Step 1.} First we show that the limit is dominated by the behavior in a neighborhood of $\theta$. Namely, in this step we show that for sufficiently small fixed $\ve>0$ we have
\be
\lim_{t\to\infty} |\det(I+K^{[t]})_{L^2(\mathcal C_\theta)}-\det(I+K^{[t]}_{\ve})_{L^2(\mathcal D_\ve)}|=0,
\ee
where $\D_\ve=\{\theta+Re^{\pm 2\pi\i/3}\mid R\in [0,\ve]\}$.

It is more convenient to proceed in two steps, reducing first the contour $\mathcal C_\theta$ to $\mathcal D_\ve$ and then the kernel $K^{[t]}$ to $K_{\ve}^{[t]}$. Note that we can deform $\C_\theta$ to $\K(\ve)=\D_\ve\cup\A_\ve\cup (\C_\theta\backslash\C_{\theta,\ve})$ without changing the value of the determinant. Since $\mathcal C_\theta$ is a steep descent contour for $-\Re[h(z)]$, for each $\ve$ there exists a constant $b_1>0$ such that
 \be
 \Re[h(v)-h(\theta)]\geq b_1, \qquad v\in \C_\theta\backslash \C_{\theta,\ve}.
 \ee
 At the same time, by Proposition \ref{AcontProp}, for sufficiently small $\ve$ there exists a constant $b_2>0$ (depending on $\ve$) such that
 \be
 \Re[h(v)-h(\theta)]\geq b_2, \qquad v=\theta+\ve e^{\pm \phi\i}, \phi\in[\psi(\ve), 2\pi/3].
 \ee
Thus, for sufficiently small fixed $\ve$ there exists $b>0$ such that
\be
\Re[h(v)-h(\theta)]\geq b, \qquad v\in \A_\ve\cup(\C_\theta\backslash\C_{\theta,\ve})=\K(\ve)\backslash\D_\ve.
\ee
Hence we can apply Lemma \ref{Kvboundlem}, obtaining a constant $\Lambda_1$ such that
\begin{equation}
\label{cuttingFredcontourBound}
K^{[t]}(v,v')<\Lambda_1 e^{-bt/2}, \qquad v\in \K(\ve) \backslash \mathcal D_\ve, v'\in \K(\ve).
\end{equation}

Note that the difference of Fredholm determinants $\det(I+K^{[t]})_{L^2(\mathcal C_\theta)}-\det(I+K^{[t]})_{L^2(\mathcal D_\ve)}$ is equal to $\sum_{l\geq 0}\frac{B_l}{l!}$, where  $B_l$ are integrals of $\det[K(v_i,v_j)]_{i,j=1}^l$ taken over the difference  $\K(\ve)^l\backslash \mathcal D^l_\ve$. Since each point in this difference has at least one of $v_i$ not in $\mathcal D_\ve$, by \eqref{cuttingFredcontourBound}, Lemma \ref{Kvboundlem} and Hadamard's bound we get
\be
\left|\idotsint\limits_{\mathcal {K}(\ve)^l\backslash \mathcal D_\ve^l}\det[K(v_i,v_j)]_{i,j=1}^l\frac{dv_1}{2\pi\i}\dots \frac{dv_l}{2\pi\i}\right|\leq \Lambda_1 e^{-bt/2} t^{1/3}\Lambda_2^{l} l^{l/2+1}.
\ee
Summing over $l$, we obtain 
\be
|\det(I+K^{[t]})_{L^2(\mathcal C_\theta)}-\det(I+K^{[t]})_{L^2(\mathcal D_\ve)}|\leq \Lambda_1 e^{-bt/2} \sum_{l\geq 1}\frac{l^{l/2+1}}{l!}\Lambda_2^lt^{l/3} <\Lambda_3 t e^{-bt/2+2\Lambda_2^2t^{2/3}}\to 0.
\ee
where $\Lambda_3$ is a constant depending only on $\Lambda_1,\Lambda_2$, and to get the last inequality we use $\frac{l^{l/2+1}}{l!}\leq \frac{2^{(l/2+1)}l}{\floor{l/2}!}$ and apply the Taylor expansion of the exponent separately to odd and even values of $l$.

The argument replacing $K^{[t]}$ by $K_{\ve}^{[t]}$ is identical and relies on Lemma \ref{Kreducelem}. We have
\begin{multline*}
\left|\det(I+K^{[t]})_{L^2(\mathcal D_\ve)}-\det(I+K^{[t]}_\ve)_{L^2(\mathcal D_\ve)}\right|\\
\leq \sum_{l\geq 0}\frac{1}{l!}\left|\idotsint\limits_{\mathcal {D}_\ve^l}\det[K(v_i,v_j)]_{i,j=1}^l-\det[K_\ve(v_i,v_j)]_{i,j=1}^l\frac{dv_1}{2\pi\i}\dots \frac{dv_l}{2\pi\i}\right|,
\end{multline*}
and using multi-linearity of determinants and Lemmas \ref{Kreducelem},\ref{Kvboundlem} we have the following upper-bound for some constants $\Lambda_3,a>0$
\be
|\det[K^{[t]}(v_i,v_j)]_{i,j=1}^l-\det[K^{[t]}_\ve(v_i,v_j)]_{i,j=1}^l|<\Lambda_3e^{-at}{\Lambda_2}^ll^{l/2+1}t^{l/3},
\ee
which is enough to give an absolutely convergent upper-bound.

So, we have reduced Theorem \ref{mainBetaResult} to proving
\be
\lim_{t\to \infty} \det(I+K_{\ve}^{[t]})_{L^2(\mathcal D_\ve)}=F_{GUE}(r).
\ee

{\bfseries Step 2.} After getting rid of the global parts of the integration contours, in this step we take the needed limit by scaling the neighborhood of $\theta$ by $t^{1/3}$.

 For now, fix $\ve>0$ which is sufficiently small so that the argument in the previous step works, but otherwise arbitrary. We would like to perform the following change of coordinates in the Fredholm determinant $\det(I+K^{[t]}_{\ve})_{L^2(\mathcal D_\ve)}$ and the kernel $K_{\ve}^{[t]}$:
\be
z=\theta+t^{-1/3}\widetilde{z}, \qquad v=\theta+t^{-1/3}\widetilde{v}.
\ee
This leads to 
\be
\det(I+K^{[t]}_{\ve})_{L^2(\mathcal D_\ve)}=\det(I+\widetilde{K}_{\ve}^{[t]})_{L^2(\widetilde{\mathcal D})},
\ee
with the contour $\widetilde{{\mathcal D}}=\{|R|e^{\sign(R)2\pi\i/3}\mid R\in [-\infty, \infty]\}$ and the kernel
\be
\widetilde{K}_{\ve}^{[t]}(\widetilde{v}, \widetilde{v}')=t^{-1/3} \1_{|\widetilde{v}|,|\widetilde{v}'|<\ve t^{1/3}} K^{[t]}_{\ve}(\theta+t^{-1/3}\widetilde{v}, \theta+t^{-1/3}\widetilde{v}').
\ee
To take the limit as $t\to \infty$ we start with the point-wise convergence of the kernel $\widetilde{K}^{[t]}_{\ve}(\widetilde{v}, \widetilde{v}')$. Let $\widetilde{\mathcal L}^{[t]}_{\ve}$ be the contour consisting of vertical lines $\{\pm Y\i\mid Y\in [1, \ve t^{1/3}]\}$ and a semicircle $\{e^{\phi\i}\mid\phi\in[-\pi,\pi]\}$, oriented so that the imaginary part increases. Also, let $\widetilde{\mathcal L}$ denote the limiting contour with the vertical lines $\{\pm Y\i\mid Y\in [1, \infty]\}$ and the same semicircle. Then
\be
\widetilde{K}^{[t]}_{\ve}(\widetilde{v}, \widetilde{v}')=\frac{\1_{|\widetilde{v}|,|\widetilde{v}'|<\ve t^{1/3}}}{2\pi\i} \int_{\widetilde{\mathcal L}^{[t]}_{\ve}}\frac{-\pi t^{-1/3}}{\sin \pi t^{-1/3}(\widetilde{z}-\widetilde{v})}\frac{e^{t\(h^{[t]}(\theta+t^{-1/3}\widetilde{z})-h^{[t]}(\theta)\) -cr\widetilde{z}}}{e^{t\(h^{[t]}(\theta+t^{-1/3}\widetilde{v})-h^{[t]}(\theta)\) -cr\widetilde{v}}}\frac{d\widetilde{z}}{\widetilde{z}-\widetilde{v}'}.
\ee
Clearly
\be
\frac{-\pi t^{-1/3}}{\sin \pi t^{-1/3}(\widetilde{z}-\widetilde{v})}\to\frac{1}{\widetilde{v}-\widetilde{z}},
\ee
\be
\lim_{t\to \infty}t\(h^{[t]}(\theta+t^{-1/3}\widetilde{z})-h^{[t]}(\theta)\)=\lim_{t\to\infty} t\(h(\theta+t^{-1/3}\widetilde{z})-h(\theta)\)=\frac{c^3}{3}\widetilde{z}^3,
\ee
where for the two equalities we have sequentially used Lemmas \ref{nbhlemma} and \ref{htsmallbound}. So the integrand converges to
\be
\frac{1}{2\pi\i}\frac{e^{c^3\widetilde{z}^3/3 -cs\widetilde{z}}}{e^{c^3\widetilde{v}^3/3-cs\widetilde{v}}}\frac{d\widetilde{z}}{(\widetilde{v}-\widetilde{z})(\widetilde{z}-\widetilde{v}')}.
\ee

Now we need to give integrable upper-bounds. Note that $\Re\left[t\(h^{[t]}(\theta+t^{-1/3}\widetilde{z})-h^{[t]}(\theta)\)\right]$ is unifromly bounded for $z\in\widetilde{\L}$: by Lemma \ref{htbound} it is enough to consider $\Re\left[t\(h(\theta+t^{-1/3}\widetilde{z})-h(\theta)\)\right]$ and for $|\widetilde{z}|>1$ we can use the steep descent of $\L_\theta$, while for $\widetilde{z}$ on the semicircle this follows from the Taylor expansion of $h(z)$ from Lemma \ref{nbhlemma}. So for some constant $\Lambda_1>0$ we have: 
\begin{equation}
\label{tmpboundStep2}
\left|\frac{\pi t^{-1/3}}{\sin \pi t^{-1/3}(\widetilde{z}-\widetilde{v})}\frac{e^{t\(h^{[t]}(\theta+t^{-1/3}\widetilde{z})-h^{[t]}(\theta)\) -cs\widetilde{z}}}{\widetilde{z}-\widetilde{v}'}\right|<\frac{\Lambda_1}{|\tz-\tv|\,|\tz-\tv'|}, \qquad\tz\in\widetilde{\L},\ \tv,\tv'\in\widetilde{\D},\ |\tv|<\ve t^{1/3}
\end{equation}
Due to the quadratic decay we can apply dominated convergence, obtaining
\begin{equation}
\label{Kinfty}
\lim_{t\to\infty} \widetilde{K}^{[t]}_{\ve}(\widetilde{v}, \widetilde{v}')=\widetilde{K}_{\infty}(\widetilde{v},\widetilde{v}'):=\frac{1}{2\pi\i}\int_{\widetilde{\mathcal L}}\frac{e^{c^3\widetilde{z}^3/3 -cr\widetilde{z}}}{e^{c^3\widetilde{v}^3/3-cr\widetilde{v}}}\frac{d\widetilde{z}}{(\widetilde{v}-\widetilde{z})(\widetilde{z}-\widetilde{v}')}.
\end{equation}

To finish we need to justify commutation of the Fredholm determinant and the limit above. To do it, we again use dominated convergence with an absolutely summable upper-bound on the Fredholm determinant. Note that for $\tv\in\widetilde{\mathcal D}$ such that $|\tv|<\ve t^{1/3}$ we have
\be
\Re[t(h^{[t]}(\theta+t^{-1/3}\tv)-h^{[t]}(\theta))-cr\tv]>\Re[t(h(\theta+t^{-1/3}\tv)-h(\theta))]-\ln \Lambda_2
\ee
for some $\Lambda_2>0$. By Taylor expansion, there exists an \emph{independent of $\ve$} constant $\Xi$ such that 
\be
|t(h(\theta+t^{-1/3}\tv)-h(\theta)) -c^3\tv^3/3|<\Xi t^{-1/3}|\tv|^4, \qquad |\tv|<\theta t^{1/3}/2.
\ee 
Then for any $\ve<1$ we get
\be
\left|\frac{1}{e^{t(h^{[t]}(\theta+t^{-1/3}\tv)-h^{[t]}(\theta))-cs\tv}}\right|<\Lambda_2e^{-(\frac{c^3}{3}-\ve\Xi)\tv^3}, \qquad \tv\in\widetilde{\D},\ |\tv|<\ve t^{1/3}.
\ee
Recall that for the whole argument $\ve$ was an arbitrary sufficiently small number, so we can reduce it even further to get $\frac{c^3}{3}-\ve\Xi=b>0$. Combining the bound above with \eqref{tmpboundStep2} we get
\be
|\widetilde{K}^{[t]}_{\ve}(\widetilde{v}, \widetilde{v}')|<\Lambda_3 e^{-b\tv^{3}}
\ee
for some $\Lambda_3>0$, which combined with Hadamard's bound is sufficient for an integrable and absolutely summable upper-bound on the kernel of the Fredholm determinant. Hence
\be
\lim_{t\to \infty} \det(I+K^{[t]}_{\ve})_{L^2(\mathcal D_\ve)}=\det(I+\widetilde{K}_\infty)_{L^2(\widetilde{\D})},
\ee
with the kernel $\widetilde{K}_\infty$ defined in \eqref{Kinfty}.

{\bfseries Step 3.} To finish the proof we need to show that the Fredholm determinant in the right-hand side is indeed equal to $F_{GUE}$, which can be done by standard manipulations. Rescaling $\tv,\tz$ by $c$ we can assume that 
\be
\widetilde{K}_{\infty}(\widetilde{v},\widetilde{v}')=\frac{1}{2\pi\i}\int_{\widetilde{\mathcal L}}\frac{e^{\widetilde{z}^3/3 -r\widetilde{z}}}{e^{\widetilde{v}^3/3-r\widetilde{v}}}\frac{d\widetilde{z}}{(\widetilde{v}-\widetilde{z})(\widetilde{z}-\widetilde{v}')}
\ee
Since the integrand has quadratic decay, we can deform the contour $\widetilde{\mathcal L}$ to a contour going from $\infty e^{-i\pi/3}$ to $\infty e^{i\pi/3}$ with $\Re[\tz]>0$.  Now we can apply the same argument as in \cite[Lemma 8.6]{BCF12}: For any $\tz$ such that $\Re[\tz]>0$ we  use
\be
\frac{1}{\tz-\tv}=\int_{\mathbb R_{\geq 0}}e^{-\lambda(\tz-\tv)}d\lambda
\ee
to write $\widetilde{K}_\infty=-AB$, where $A,B$ are operators with kernels
\be
A(v,\lambda)=e^{-\tv^3/3+\tv(r+\lambda)},\qquad B(\lambda,\tv')=\frac{1}{2\pi\i}\int_{e^{-\pi\i/3}\infty}^{e^{\pi\i/3}\infty}e^{\tz^3/3-\tz(r+\lambda)}\frac{d\tz}{\tz-\tv'}.
\ee
Then
\be
BA(\lambda,\mu)=\int_{e^{-2\pi\i/3}\infty}^{e^{2\pi\i/3}\infty}\frac{d\tv}{2\pi \i}\int_{e^{-\pi\i/3}\infty}^{e^{\pi\i/3}\infty}\frac{d\tz}{2\pi \i}\ \frac{1}{\tz-\tv}\frac{e^{\tz^3/3-\tz(r+\lambda)}}{e^{\tv^3/3-\tv(r+\mu)}}=K_{\Ai}(r+\lambda,r+\mu),
\ee
where $K_{\Ai}(\lambda,\mu)$ is the Airy kernel. Since $AB$ and $BA$ are trace-class operators we have $\det(I-AB)=\det(I-BA)$, \emph{cf.} \cite{Sim00}, so we finally get 
\be
\det(I+\widetilde{K}_\infty)_{L^2(\mathcal D)}=\det(I-K_{\Ai})_{L^2[r,\infty)}=F_{GUE}(r).
\ee
\qed

\begin{appendix}
\newpage
\section{Appendix: deformed Yang-Baxter equations}
\label{app:deformed}
The aim of this section is to provide proofs of \eqref{defWYB} and \eqref{defhsYB}. We start with the former.

\begin{prop}
\label{defWYBprop}
For any parameters $t_1, t_2, t_3,\eta$  and any boundary conditions the following equality of rational functions holds
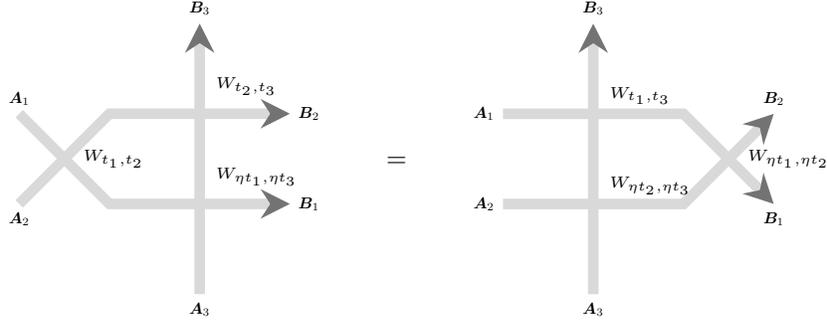
\be
\tikzbase{1.2}{-3}{
	\draw[fused]
	(-2,0.5) node[above,scale=0.6] {\color{black} $\bA_1$} -- (-1,-0.5) -- (1,-0.5) node[right,scale=0.6] {\color{black} $\bB_1$};
	\draw[fused] 
	(-2,-0.5) node[below,scale=0.6] {\color{black} $\bA_2$} -- (-1,0.5)  -- (1,0.5) node[right,scale=0.6] {\color{black} $\bB_2$};
	\draw[fused] 
	(0,-1.5) node[below,scale=0.6] {\color{black} $\bA_3$} -- (0,0) -- (0,1.5) node[above,scale=0.6] {\color{black} $\bB_3$};
	\node[above right] at (0,-0.4) {\tiny{ $W_{\eta t_1,\eta t_3}$}};
	\node[above right] at (0,0.6) {\tiny{ $W_{t_2,t_3}$}};
	\node[right] at (-1.5,0) {\ \tiny{$W_{t_1, t_2}$}};
}\qquad
=\qquad
\tikzbase{1.2}{-3}{
	\draw[fused] 
	(-1,1) node[left,scale=0.6] {\color{black} $\bA_1$} -- (1,1) -- (2,0) node[below,scale=0.6] {\color{black} $\bB_1$};
	\draw[fused] 
	(-1,0) node[left,scale=0.6] {\color{black} $\bA_2$} -- (1,0) -- (2,1) node[above,scale=0.6] {\color{black} $\bB_2$};
	\draw[fused] 
	(0,-1) node[below,scale=0.6] {\color{black} $\bA_3$} -- (0,0.5) -- (0,2) node[above,scale=0.6] {\color{black} $\bB_3$};
	\node[above right] at (0,0) {\tiny{ $W_{\eta t_2,\eta t_3}$}};
	\node[above right] at (0,1) {\tiny{ $W_{t_1,t_3}$}};
	\node[right] at (1.5,0.5) {\ \tiny{$W_{\eta t_1,\eta t_2}$}};
}
\ee
\end{prop}
\begin{proof}
The identity follows at once after noticing that
\be
W_{\eta t,\eta s}(\bA,\bB;\bC,\bD)=\frac{(\eta^2t^2;q)_{|\bD|}}{(\eta^2s^2;q)_{|\bA|}}\frac{(s^2;q)_{|\bA|}}{(t^2;q)_{|\bD|}}W_{t,s}(\bA,\bB;\bC,\bD)
\ee
and applying the Yang-Baxter equation \eqref{WYB}.
\end{proof}

\begin{prop}\label{defhsYBprop}
For any $x,s,t,\eta$ and any boundary conditions the following equality of rational functions holds
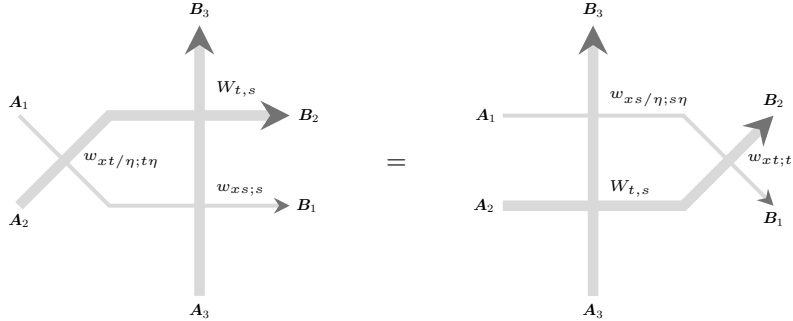
\be
\tikzbase{1.2}{-3}{
	\draw[unfused]
	(-2,0.5) node[above,scale=0.6] {\color{black} $\bA_1$} -- (-1,-0.5) -- (1,-0.5) node[right,scale=0.6] {\color{black} $\bB_1$};
	\draw[fused] 
	(-2,-0.5) node[below,scale=0.6] {\color{black} $\bA_2$} -- (-1,0.5)  -- (1,0.5) node[right,scale=0.6] {\color{black} $\bB_2$};
	\draw[fused] 
	(0,-1.5) node[below,scale=0.6] {\color{black} $\bA_3$} -- (0,0) -- (0,1.5) node[above,scale=0.6] {\color{black} $\bB_3$};
	\node[above right] at (0,-0.5) {\tiny{ $w_{xs;s}$}};
	\node[above right] at (0,0.6) {\tiny{ $W_{t,s}$}};
	\node[right] at (-1.5,0) {\ \tiny{$w_{xt/\eta;t\eta}$}};
}\qquad
=\qquad
\tikzbase{1.2}{-3}{
	\draw[unfused] 
	(-1,1) node[left,scale=0.6] {\color{black} $\bA_1$} -- (1,1) -- (2,0) node[below,scale=0.6] {\color{black} $\bB_1$};
	\draw[fused] 
	(-1,0) node[left,scale=0.6] {\color{black} $\bA_2$} -- (1,0) -- (2,1) node[above,scale=0.6] {\color{black} $\bB_2$};
	\draw[fused] 
	(0,-1) node[below,scale=0.6] {\color{black} $\bA_3$} -- (0,0.5) -- (0,2) node[above,scale=0.6] {\color{black} $\bB_3$};
	\node[above right] at (0,0) {\tiny{ $W_{t,s}$}};
	\node[above right] at (0,1) {\tiny{ $w_{xs/\eta;s\eta}$}};
	\node[right] at (1.5,0.5) {\ \tiny{$w_{xt;t}$}};
}
\ee
\end{prop}

\begin{proof}
The proof is based on an analytic continuation of certain compositions appearing in the non-deformed version \eqref{hsYB}, followed by a non-integral shift.

First we introduce a family of vertex weights denoted by
\be
\tikz{1}{
	\draw[unfused] (-1,0) -- (1,0);
	\draw[cont] (0,-1) -- (0,1);
	\node[left] at (-1,0) {\tiny $j$};\node[right] at (1,0) {\tiny $l$};
	\node[below] at (0,-1) {\tiny $\bm{\A}$};\node[above] at (0,1) {\tiny $\bm{\B}$};
	\node[above right] at (0,0) {\tiny $\widehat{w}_{z;s}$};
}
=\widehat{w}_{z;s}(\bm{\A},j,\bm{\B},l),
\ee
where $\bm{\A}=(\A_1, \dots, \A_n)\in{\mathbb C}^n$ and $\bm{\B}=(\B_1, \dots, \B_n)\in\mathbb C^n$ denote \emph{analytically continued} compositions. The explicit values of these weights are summarized in a table similar to \eqref{hsdef}:
\begin{align}
\label{hscontdef}
\begin{tabular}{|c|c|c|}
\hline
\quad
\tikz{0.7}{
	\draw[unfused] (-1,0) -- (1,0);
	\draw[cont] (0,-1) -- (0,1);
	\node[left] at (-1,0) {\tiny $0$};\node[right] at (1,0) {\tiny $0$};
	\node[below] at (0,-1) {\tiny $\bm{\A}$};\node[above] at (0,1) {\tiny $\bm{\A}$};
}
\quad
&
\quad
\tikz{0.7}{
	\draw[unfused] (-1,0) -- (1,0);
	\draw[cont] (0,-1) -- (0,1);
	\node[left] at (-1,0) {\tiny $i$};\node[right] at (1,0) {\tiny $i$};
	\node[below] at (0,-1) {\tiny $\bm{\A}$};\node[above] at (0,1) {\tiny $\bm{\A}$};
}
\quad
&
\quad
\tikz{0.7}{
	\draw[unfused] (-1,0) -- (1,0);
	\draw[cont] (0,-1) -- (0,1);
	\node[left] at (-1,0) {\tiny $0$};\node[right] at (1,0) {\tiny $i$};
	\node[below] at (0,-1) {\tiny $\bm{\A}$};\node[above] at (0,1) {\tiny $T^{-1}_{q;i}\bm{\A}$};
}
\quad
\\[1.3cm]
\quad
$\dfrac{1-s z\A_{[1;n]}}{1-sz}$
\quad
& 
\quad
$\dfrac{(s^2\A_i-sz) \A_{[i+1;n]}}{1-sz}$
\quad
& 
\quad
$\dfrac{sz(\A_i-1) \A_{[i+1;n]}}{1-sz}$
\quad
\\[0.7cm]
\hline
\quad
\tikz{0.7}{
	\draw[unfused] (-1,0) -- (1,0);
	\draw[cont] (0,-1) -- (0,1);
	\node[left] at (-1,0) {\tiny $i$};\node[right] at (1,0) {\tiny $0$};
	\node[below] at (0,-1) {\tiny $\bm{\A}$};\node[above] at (0,1) {\tiny $T_{q;i}\bm{\A}$};
}
\quad
&
\quad
\tikz{0.7}{
	\draw[unfused] (-1,0) -- (1,0);
	\draw[cont] (0,-1) -- (0,1);
	\node[left] at (-1,0) {\tiny $i$};\node[right] at (1,0) {\tiny $j$};
	\node[below] at (0,-1) {\tiny $\bm{\A}$};\node[above] at (0,1) 
	{\tiny $T_{q;i}T_{q;j}^{-1}\bm{\A}$};
}
\quad
&
\quad
\tikz{0.7}{
	\draw[unfused] (-1,0) -- (1,0);
	\draw[cont] (0,-1) -- (0,1);
	\node[left] at (-1,0) {\tiny $j$};\node[right] at (1,0) {\tiny $i$};
	\node[below] at (0,-1) {\tiny $\bm{\A}$};\node[above] at (0,1) {\tiny $T_{q;i}^{-1}T_{q;j}\bm{\A}$};
}
\quad
\\[1.3cm] 
\quad
$\dfrac{1-s^2 \A_{[1;n]}}{1-sz}$
\quad
& 
\quad
$\dfrac{sz(\A_j-1) \A_{[j+1;n]}}{1-sz}$
\quad
&
\quad
$\dfrac{s^2(\A_i-1)\A_{[i+1;n]}}{1-sz}$
\quad
\\[0.7cm]
\hline
\end{tabular} 
\end{align}
Here $1\leq i<j\leq n$, for the analytically continued composition $\bm{\A}=(\A_1, \dots, \A_n)\in\mathbb{C}^n$ we set $\A_{[a,b]}:=\prod_{i=a}^b\A_i$ and $T_{q;i}$ (resp. $T^{-1}_{q;i}$) denotes the operator multiplying the $i$th component by $q$ (resp. by $q^{-1}$). For the configurations not listed above we assume that the weight $\widehat{w}_{z;s}(\bm{\A}, j;\bm{\B}, l)$ vanishes, so that the conservation law holds:
\be
T_{q;j}\bm{\A}=T_{q;l}\bm{\B},
\ee
where we define $T_{q;0}$ as the identity operator. 

Comparing \eqref{hsdef} and \eqref{hscontdef} we get the following relation for any $0\leq j,l\leq n$ and compositions $\bI,\bK$:
\begin{equation}
\label{contSpec}
\tikz{1}{
	\draw[unfused] (-1,0) -- (1,0);
	\draw[cont] (0,-1) -- (0,1);
	\node[left] at (-1,0) {\tiny $j$};\node[right] at (1,0) {\tiny $l$};
	\node[below] at (0,-1) {\tiny $\restr{\(T_{\eta^2;j}{\bm{\A}}\)}{\substack{\A_c=q^{I_c},\\ c=1,\dots,n}}$};\node[above] at (0,1) {\tiny $\restr{\(T_{\eta^2; j}{\bm{\B}}\)}{\substack{\B_c=q^{K_c}\\c=1,\dots,n}}$};
	\node[above right] at (0,0) {\tiny $\widehat{w}_{z;s}$};
}=
\tikz{1}{
	\draw[unfused] (-1,0) -- (1,0);
	\draw[fused] (0,-1) -- (0,1);
	\node[left] at (-1,0) {\tiny ${j}$};\node[right] at (1,0) {\tiny $l$};
	\node[below] at (0,-1) {\tiny ${\bm I}$};\node[above] at (0,1) {\tiny ${\bm K}$};
	\node[above right] at (0,0) {\tiny $w_{z/\eta; \eta s}$};
},
\end{equation}
where in the left hand-side we specialize $\bm{\A}=(q^{I_1}, \dots, q^{I_n}), \bm{\B}=(q^{K_1}, \dots, q^{K_n})$ and, as before, $T_{\eta^2;j}$ denotes the operator multiplying the $j$th position by $\eta^2$, acting by identity if $j=0$. For $\eta=1$ the relation \eqref{contSpec} explains the interpretation of the weights $\widehat{w}_{z;s}$ as an analytic continuation of the weights $w_{z;s}$.

The next step is an analytic continuation of the weights $W_{t,s}(\bA,\bB;\bC,\bD)$. Here we make use of the fact that these weights actually do not depend on $\bB$ and $\bC$ as long as the conservation law holds, so we define
\be
\tikz{1}{
	\draw[fused] (0,-1) -- (0,0) -- (1,0);
	\draw[cont] (-1,0) -- (0,0) -- (0,1);
	\node[left] at (-1,0) {\tiny ${\bm{\B}}$};\node[right] at (1,0) {\tiny $\bD$};
	\node[below] at (0,-1) {\tiny ${\bA}$};\node[above] at (0,1) {\tiny ${\bm{\C}}$};
	\node[above right] at (0,0.1) {\tiny $\widehat{W}_{t,s}$};
}=\1_{q^{\bA}\bm{\B}=q^{\bD}\bm{\C}}\  (s^2/t^2)^{|\bD|}\frac{(s^2/t^2;q)_{|\bA|-|\bD|}(t^2;q)_{|\bD|}}{(s^2;q)_{|\bA|}}q^{\sum_{i<j}D_i(A_j-D_j)}\prod_{i=1}^n\binom{A_i}{D_i}_q,
\ee
where $\bm{\B}=(\B_1, \dots, \B_n)$ and $\bm{\C}=(\C_1, \dots, \C_n)$ are analytically continued compositions, and the conservation law $q^{\bA}\bm{\B}=q^{\bD}\bm{\C}$ reads as $q^{A_i}\B_i=q^{D_i}\C_i$ for every $i$.

Using the analytically continued weights $\widehat{w}_{z;s}$ and $\widehat{W}_{s,t}$, the Yang-Baxter equation \eqref{hsYB} can be written as 
\begin{equation}
\label{contYB}
\tikzbase{1.2}{-3}{
	\draw[unfused]
	(-2,0.5) node[above,scale=0.6] {\color{black} $a_1$} -- (-1,-0.5) -- (1,-0.5) node[right,scale=0.6] {\color{black} $b_1$};
	\draw[fused] 
	(0,-1.5) node[below,scale=0.6] {\color{black} $\bA_3$} -- (0,0) -- (0,0.5) -- (1,0.5) node[right,scale=0.6] {\color{black} $\bB_2$};
	\draw[cont] 
	(-2,-0.5) node[below,scale=0.6] {\color{black} $\bm{\A}$} -- (-1,0.5)  -- (0,0.5) -- (0,1.5) node[above,scale=0.6] {\color{black} $\bm{\B}$};
	\node[above right] at (0,-0.5) {\tiny{ $w_{xs;s}$}};
	\node[above right] at (0,0.6) {\tiny{ $\widehat{W}_{t,s}$}};
	\node[right] at (-1.5,0) {\ \tiny{$\widehat{w}_{xt;t}$}};
}\qquad
=\qquad
\tikzbase{1.2}{-3}{
	\draw[unfused] 
	(-1,1) node[left,scale=0.6] {\color{black} $a_1$} -- (1,1) -- (2,0) node[below,scale=0.6] {\color{black} $b_1$};
	\draw[fused] 
	(0,-1) node[below,scale=0.6] {\color{black} $\bA_3$} -- (0,0) -- (1,0) -- (2,1) node[above,scale=0.6] {\color{black} $\bB_2$} ;
	\draw[cont] 
	(-1,0) node[left,scale=0.6] {\color{black} $\bm{\A}$} -- (0,0) -- (0,2) node[above,scale=0.6] {\color{black} $\bm{\B}$};
	\node[above right] at (0,0) {\tiny{ $\widehat{W}_{t,s}$}};
	\node[above right] at (0,1) {\tiny{ $\widehat{w}_{xs;s}$}};
	\node[right] at (1.5,0.5) {\ \tiny{$w_{xt;t}$}};
}
\end{equation}
where $\bm{\A}=q^{\bm{X}}:=(q^{X_1}, \dots, q^{X_n})$ for sufficiently large $\bm{X}=(X_1, \dots, X_n)\in\mathbb Z_{\geq 0}$, and $\bm{\B}$ satisfies
\be
T_{q;a_1}q^{\bA_3}\bm{\A}=T_{q;b_1}q^{\bB_2}\bm{\B}.
\ee
As before, the diagrams in \eqref{contYB} denote sums over all configurations not violating the conservation law for the weights $w,\widehat{w}$ and $\widehat{W}$. For fixed boundary conditions there are at most $n+1$ such configurations for each side of \eqref{contYB}: each allowed configuration is uniquely determined by the label of the internal thin edge, which has $n+1$ possibilities. Since all weights depend rationally on $\bm{\A}=(\A_1, \dots, \A_n)$, both sides of \eqref{contYB} are rational functions in $\bm{\A}$ coinciding for $\bm{\A}=(q^{X_1}, \dots, q^{X_n})$, which implies that they are equal. So the analytically continued equation \eqref{contYB} holds for any $\bm{\A}$. 

The proof is finished by applying $T_{\eta^2;a_1}$ to $\bm{\A}$ and $\bm{\B}$ and then setting back $\bm{\A}=q^{\bA_2}$ and $\bm{\B}=q^{\bB_3}$. By \eqref{contSpec} the resulting equation coincides with the claim.
\end{proof}


\section{Appendix: Computations for Theorem \ref{mainBetaResult}}\label{app:num} Here we provide computational arguments for several statements used in the proof of Theorem \ref{mainBetaResult}. Recall that for fixed real parameters $\{\sigma_1 ,\dots, \sigma_k\}$, $\{\rho_1 ,\dots, \rho_k\}$, $\{\omega_1 ,\dots, \omega_k\}$, $\{\alpha_1 ,\dots, \alpha_k\}$, $\{\beta_1 ,\dots, \beta_k\}$, $\{\gamma_1 ,\dots, \gamma_k\}$, $\theta$ satisfying
\be
\sum_{i=1}^k\alpha_i=\sum_{i=1}^k\beta_i=\sum_{i=1}^k\gamma_i=1,\qquad \alpha_i,\beta_i,\gamma_i\geq 0,
\ee
\be
\omega_d<\rho_j<\sigma_i<\theta\qquad \text{for\ any\ } i,j,d,
\ee
we have defined
\be
x=\sum_i\beta_i\Psi_1(\theta-\rho_i)-\sum_i\gamma_i\Psi_1(\theta-\omega_i),\qquad y=\sum_i\alpha_i\Psi_1(\theta-\sigma_i)-\sum_i\gamma_i\Psi_1(\theta-\omega_i),
\ee
\be
I=x\sum_i\alpha_i\Psi(\theta-\sigma_i)-y\sum_i\beta_i\Psi(\theta-\rho_i)+ (y-x)\sum_i\gamma_i\Psi(\theta-\omega_i),
\ee
\be
h(z)=Iz-\sum_ix\alpha_i\ln\Gamma(z-\sigma_i)+\sum_iy\beta_i\ln\Gamma(z-\rho_i)-\sum_i(y-x)\gamma_i\ln\Gamma(z-\omega_i),
\ee
where 
\be
\Psi(z)=\frac{d}{dz}\ln\Gamma(z),\qquad \Psi_k(z)=\frac{d^{k+1}}{dz^{k+1}}\ln \Gamma(z).
\ee
For the computational purposes it is convenient to plug the expressions for $x,y,I$ into $h(z)$ to obtain
\begin{multline*}
h(z)=-\(\sum_i\beta_i\Psi_1(\theta-\rho_i)-\sum_i\gamma_i\Psi_1(\theta-\omega_i)\)\sum_i\alpha_i\(\ln\Gamma(z-\sigma_i)-z\Psi(\theta-\sigma_i)\)\\
+\(\sum_i\alpha_i\Psi_1(\theta-\sigma_i)-\sum_i\gamma_i\Psi_1(\theta-\omega_i)\)\sum_i\beta_i\(\ln\Gamma(z-\rho_i)-z\Psi(\theta-\rho_i)\)\\
-\(\sum_i\alpha_i\Psi_1(\theta-\sigma_i)-\sum_i\beta_i\Psi_1(\theta-\rho_i)\)\sum_i\gamma_i\(\ln\Gamma(z-\omega_i)-z\Psi(\theta-\omega_i)\).
\end{multline*}
The last expression can be rewritten as a convex sum of simpler functions, corresponding to $k=1$ case:
\begin{equation}\label{convex}
h(z)=\sum_{i=1}^k\sum_{j=1}^k\sum_{d=1}^k\alpha_i\beta_j\gamma_dh_{\sigma_i,\rho_j,\omega_d}(z),
\end{equation}
\begin{multline*}
h_{\sigma,\rho,\omega}(z)=-\(\Psi_1(\theta-\rho)-\Psi_1(\theta-\omega)\)\(\ln\Gamma(z-\sigma)-z\Psi(\theta-\sigma)\)\\
+\(\Psi_1(\theta-\sigma)-\Psi_1(\theta-\omega)\)\(\ln\Gamma(z-\rho)-z\Psi(\theta-\rho)\)\\
-\(\Psi_1(\theta-\sigma)-\Psi_1(\theta-\rho)\)\(\ln\Gamma(z-\omega)-z\Psi(\theta-\omega)\).
\end{multline*}
We will mostly work with the latter functions $h_{\sigma,\rho,\omega}(z)$, which for further convenience we rewrite as 
\begin{multline*}
h_{\sigma,\rho,\omega}(z)=\(\Psi_1(\theta-\rho)-\Psi_1(\theta-\omega)\)\(\ln\Gamma(z-\rho)-\ln\Gamma(z-\sigma)-z\Psi(\theta-\rho)+z\Psi(\theta-\sigma)\)\\
+\(\Psi_1(\theta-\sigma)-\Psi_1(\theta-\rho)\)\(\ln\Gamma(z-\rho)-\ln\Gamma(z-\omega)-z\Psi(\theta-\rho)+z\Psi(\theta-\omega)\).
\end{multline*}

In the following computations we frequently use the following series representations:
\begin{equation}
\label{Psiseries}
\Psi(z)-\Psi(\theta)=\sum_{n\geq 0}-\frac{1}{n+z}+\frac{1}{n+\theta}=\sum_{n\geq 0}\frac{z-\theta}{(n+\theta)(n+z)},
\end{equation}
\begin{equation}
\label{Psi1series}
\Psi_k(z)=\sum_{n\geq 0}\frac{(-1)^{k+1}k!}{(n+z)^{k+1}}.
\end{equation}
In particular, the functions $(-1)^{k}\Psi_k(z)$ are increasing functions from $(0,\infty)$ to $(-\infty,0)$, and 
\begin{equation}
\label{Psirecursion}
\Psi_k(z+1)=\Psi_k(z)-\frac{(-1)^{k+1}k!}{z^{k+1}}.
\end{equation}

\subsection{Derivatives of $h(z)$} This section is devoted to the signs of $h^{(d)}(\theta)$.

\begin{lem} \label{derivatives}
For any $d\geq 3$ we have $(-1)^{d+1}h^{(d)}(\theta)>0$.
\end{lem}
\begin{proof}
By \eqref{convex} it is enough to prove that $(-1)^{d+1}h_{\sigma,\rho,\omega}^{(d)}(\theta)>0$ for $\omega<\rho<\sigma$. By a slight abuse of notation let $h$ denote $h_{\sigma,\rho,\omega}$ during the proof.

Our approach follows \cite[Lemma 5.3]{BC15b}. For $d\geq 3$ we have
\begin{multline*}
h^{(d)}(\theta)=\(\Psi_1(\theta-\rho)-\Psi_1(\theta-\omega)\)\(\Psi_{d-1}(\theta-\rho)-\Psi_{d-1}(\theta-\sigma)\)\\
+\(\Psi_1(\theta-\sigma)-\Psi_1(\theta-\rho)\)\(\Psi_{d-1}(\theta-\rho)-\Psi_{d-1}(\theta-\omega)\).
\end{multline*}
Thus, we need to prove 
\be
\frac{(-1)^{d-1}\Psi_{d-1}(\theta-\rho)-(-1)^{d-1}\Psi_{d-1}(\theta-\omega)}{\Psi_1(\theta-\rho)-\Psi_1(\theta-\omega)}>\frac{(-1)^{d-1}\Psi_{d-1}(\theta-\sigma)-(-1)^{d-1}\Psi_{d-1}(\theta-\rho)}{\Psi_1(\theta-\sigma)-\Psi_1(\theta-\rho)}.
\ee
By the Cauchy's mean value theorem this inequality is equivalent to
\be
\frac{(-1)^{d-1}\Psi_{d}(\zeta)}{\Psi_2(\zeta)}>\frac{(-1)^{d-1}\Psi_{d}(\xi)}{\Psi_2(\xi)}
\ee
for some $\xi\in(\theta-\sigma,\theta-\rho)$ and $\zeta\in(\theta-\rho,\theta-\omega)$. Since $\xi<\zeta$, it is enough to prove that the function $\frac{(-1)^{d-1}\Psi_{d}(z)}{\Psi_2(z)}$ is increasing on $\mathbb R_{>0}$. Taking derivative, it is enough to prove that
\be
(-1)^{d-1}\(\Psi_{d+1}(z)\Psi_{2}(z)-\Psi_{d}(z)\Psi_{3}(z)\)>0,\qquad z>0.
\ee
Using series representation \eqref{Psi1series}, the left hand side can be rewritten as 
\begin{multline*}
(-1)^{d-1}\(\Psi_{d+1}(z)\Psi_{2}(z)-\Psi_{d}(z)\Psi_{3}(z)\)=\sum_{n,m\geq 0}\frac{2(d+1)!}{(n+z)^{d+2}(m+z)^3}-\frac{6\ d!}{(n+z)^{d+1}(m+z)^4}\\
=2d! \sum_{n,m\geq 0}\frac{(d+1)(m+z)^{d-1}-3(n+z)(m+z)^{d-2}}{(n+z)^{d+2}(m+z)^{d+2}}\geq 6 d! \sum_{n,m\geq 0}\frac{(m-n)(m+z)^{d-2}}{(n+z)^{d+2}(m+z)^{d+2}},
\end{multline*}
where we have used $d\geq 3$ to establish inequality. The proof is finished by symmetrizing the last summation:
\be
2\sum_{n,m\geq 0}\frac{(m-n)(m+z)^{d-2}}{(n+z)^{d+2}(m+z)^{d+2}}=\sum_{n,m\geq 0}\frac{(m-n)\((m+z)^{d-2}-(n+z)^{d-2}\)}{(n+z)^{d+2}(m+z)^{d+2}}>0,
\ee
where for the last inequality we have used that the function $z^{d-2}$ is increasing for $d\geq 3$.
\end{proof}

\subsection{Proof of Proposition \ref{goodParam}} \label{app:goodparam} It turns out that the two parts Proposition \ref{goodParam} are closely related. In the first part we want to prove that $x(\theta)/y(\theta)$ is increasing as a function of $\theta\in (\max_i\sigma_i, \infty)$, where
\be
x=\sum_i\beta_i\Psi_1(\theta-\rho_i)-\sum_i\gamma_i\Psi_1(\theta-\omega_i),\qquad y=\sum_i\alpha_i\Psi_1(\theta-\sigma_i)-\sum_i\gamma_i\Psi_1(\theta-\omega_i).
\ee
Taking derivative, we need to prove that
\be
x'(\theta)y(\theta)-x(\theta)y'(\theta)>0.
\ee
After algebraic manipulations the left-hand side of the inequality above is equal to
\be
-x(\theta)\sum_i\alpha_i\Psi_2(\theta-\sigma_i)+y(\theta)\sum_i\beta_i\Psi_2(\theta-\rho_i) -(y(\theta)-x(\theta))\sum_i\gamma_i\Psi_2(\theta-\omega_i),
\ee
which is exactly how we have defined $2c(\theta)^3$. So, the monotonicity from the first part boils down to showing that $c(\theta)^3>0$, which would immediately imply that $c(\theta)$ can be chosen to be positive. But from the expression above it is clear that $\frac{c(\theta)^3}{3}=h'''(\theta)>0$, with the inequality following from Lemma~\ref{derivatives}.

The limits of $x(\theta)/y(\theta)$ as $\theta\to \max_i\sigma_i,\infty$ are readily obtained from the definitions and the limit behavior of the trigamma function:
\be
\lim_{z\to 0}\Psi_1(z)=\infty,\qquad \Psi_1(z)=\frac{1}{z}+\frac{1}{2z^2}+O(z^{-3}).
\ee
\qed

\subsection{Proof of Proposition \ref{LcontProp}} \label{app:Lcont} Here we give a slightly modified version of the proof of the corresponding fact in \cite{BC15b}. Recall that we want to show that $\Re[h(\theta+b\i)]$ decreases for positive $b$ and increases for negative $b$. Since $h(z)$ is a convex combination of $h_{\sigma,\rho,\omega}(z)$, it is enough to consider the latter functions. By an abuse of notation we again denote the function $h_{\sigma,\rho,\omega}(z)$ simply by $h(z)$. By symmetry, it is also enough to consider the case $\Im(z)>0$, so we want to prove that
\be
\Re[\i h'(\theta+b\i)]<0, \qquad b>0.
\ee
We have
\begin{multline*}
\Re[\i h'(\theta+b\i)]=-\Im[(\Psi_1(\theta-\rho)-\Psi_1(\theta-\omega))(\Psi(\theta+b\i-\rho)-\Psi(\theta+b\i-\sigma))\\
+(\Psi_1(\theta-\sigma)-\Psi_1(\theta-\rho))(\Psi(\theta+b\i-\rho)-\Psi(\theta+b\i-\omega))].
\end{multline*}
From \eqref{Psiseries} it follows that for any real $A_1,A_2>0$
\be
\Im[\Psi(A_1+b\i)-\Psi(A_2+b\i)]=\sum_{n\geq 0}\frac{b}{(n+A_1)^2+b^2}-\frac{b}{(n+A_2)^2+b^2}.
\ee
Thus, using that $b>0$ and setting 
\be
\Phi(A)=\sum_{n\geq 0}\frac{1}{(n+A)^2+b^2},
\ee
the claim is reduced to showing
\be
-(\Psi_1(\theta-\rho)-\Psi_1(\theta-\omega))(\Phi(\theta-\rho)-\Phi(\theta-\sigma))
-(\Psi_1(\theta-\sigma)-\Psi_1(\theta-\rho))(\Phi(\theta-\rho)-\Phi(\theta-\omega))<0.
\ee
This inequality can be rewritten in more convenient form 
\be
\frac{\Psi_1(\theta-\rho)-\Psi_1(\theta-\sigma)}{\Phi(\theta-\rho)-\Phi(\theta-\sigma)}>\frac{\Psi_1(\theta-\omega)-\Psi_1(\theta-\rho)}{\Phi(\theta-\omega)-\Phi(\theta-\rho)},
\ee
which by Cauchy's mean value theorem is equivalent to
\be
\frac{\Psi_2(\xi)}{\Phi'(\xi)}>\frac{\Psi_2(\zeta)}{\Phi'(\zeta)}
\ee
for some $\xi\in(\theta-\sigma, \theta-\rho)$ and $\zeta\in (\theta-\rho, \theta-\omega)$. Since $0<\xi<\zeta$, it is enough to show that the function $\frac{\Psi_2(z)}{\Phi'(z)}$ is decreasing for $z>0$, and by taking derivative this is equivalent to
\be
\Psi_2(z)\Phi''(z)-\Psi_3(z)\Phi'(z)>0,\qquad z>0.
\ee
Using the definition of $\Phi(x)$ and \eqref{Psi1series} we obtain
\begin{multline*}
\Psi_2(z)\Phi''(z)-\Psi_3(z)\Phi'(z)=\sum_{n,m\geq 0}-\frac{2}{(n+z)^3}\frac{6(z+m)^2-2b^2}{((m+z)^2+b^2)^3}+\frac{6}{(n+z)^4}\frac{2(m+z)}{((m+z)^2+b^2)^2}\\
=12\sum_{n,m\geq 0}\frac{(m+z)^3+b^2(m+z)-(n+z)(m+z)^2+\frac{b^2}{3}(n+z)}{(n+z)^4((m+z)^2+b^2)^3}.
\end{multline*}
Clearly the terms in the numerator containing $b^2$ are positive, while for the remaining terms we have
\be
\sum_{n,m\geq 0}\frac{(m+z)^3-(n+z)(m+z)^2}{(n+z)^4((m+z)^2+b^2)^3}=\frac{1}{2}\sum_{n,m\geq 0}\frac{(m-n)(m+z)^2}{(n+z)^4((m+z)^2+b^2)^3}+\frac{(n-m)(n+z)^2}{(m+z)^4((n+z)^2+b^2)^3},
\ee
where the equality follows by symmetrization. To finish the proof it is enough to show that each summand in the last expression is nonnegative for any $n,m,z\geq 0$, which is true:
\begin{multline*}
\frac{(m-n)(m+z)^2}{(n+z)^4((m+z)^2+b^2)^3}+\frac{(n-m)(n+z)^2}{(m+z)^4((n+z)^2+b^2)^3}\geq0\\
\Leftrightarrow (m-n)\left[\(\frac{(m+z)^2}{((m+z)^2+b^2)}\)^3-\(\frac{(n+z)^2}{((n+z)^2+b^2)}\)^3\right]\geq0\\
\Leftrightarrow (m-n)\left[\frac{1}{\(1+\frac{b^2}{(m+z)^2}\)^3}-\frac{1}{\(1+\frac{b^2}{(n+z)^2}\)^3}\right]\geq0.
\end{multline*}
The last inequality holds since $\(1+\frac{b^2}{(z+m)^2}\)^{-3}$ is increasing in $m$.\qed

\subsection{Proof of Proposition \ref{CcontProp}}\label{app:Ccont} From now on we are working under Assumption \ref{assumptions}, so $\theta\in \(0,\frac{1}{2}\)$ and the function $h(z)$ is a convex combination of the functions $h_{0,-1,\omega}(z)$ for $\omega<-1$. The latter functions are simpler, and for further convenience we provide an explicit expression for them:
\be
h_{0,-1,\omega}(z)=\(\Psi_1(\theta+1)-\Psi_1(\theta-\omega)\)\(\ln(z)-\frac{z}{\theta}\)-\frac{1}{\theta^2}\(\ln\Gamma(z-\omega)-\ln\Gamma(z+1)-z\Psi(\theta-\omega)+z\Psi(\theta+1)\).
\ee
Note that to get this expression we use \eqref{Psirecursion} in the form $\Psi_1(\theta+1)-\Psi_1(\theta)=-\frac{1}{\theta^2}$.

Recall that we want to show that $\Re[h(\theta e^{\phi \i})]$ is increasing for $\phi>0$ and decreasing for $\phi<0$. Due to assumptions it is enough to consider the functions $h_{0,-1,\omega}(z)$, so by an abuse of notation we let $h$ denote a function $h_{0,-1,\omega}(z)$ for the duration of the proof. We consider only the case $\phi>0$, the other case follows by symmetry.

We need to prove that
\be
\Re[\i \theta e^{\phi \i}h'(\theta e^{\phi \i})]>0, \qquad \phi\in(0,\pi).
\ee
Computing the derivative and applying \eqref{Psiseries},\eqref{Psi1series} we obtain
\begin{multline*}
h'(z)=(\Psi_1(\theta+1)-\Psi_1(\theta-\omega))\(\frac{1}{z}-\frac{1}{\theta}\)-\frac{1}{\theta^2}(\Psi(z-\omega)-\Psi(z+1)-\Psi(\theta-\omega)+\Psi(\theta+1))\\
=\sum_{n\geq 0}\left[\frac{\theta-z}{z\theta}\(\frac{1}{(n+\theta+1)^2}-\frac{1}{(n+\theta-\omega)^2}\)-\frac{1}{\theta^2}\(\frac{z-\theta}{(n+\theta-\omega)(n+z-\omega)}-\frac{z-\theta}{(n+\theta+1)(n+z+1)}\)\right]\\
=\frac{(z-\theta)^2}{z\theta^2}\sum_{n\geq 0}\left[\frac{n+1}{(n+\theta+1)^2(n+z+1)}-\frac{n-\omega}{(n+\theta-\omega)^2(n+z-\omega)}\right].
\end{multline*}
Note that for any real $\tau$
\be
\Re\left [\i\theta e^{\phi \i} \frac{(e^{\phi \i}-1)^2}{\theta e^{\phi \i}}\frac{1}{\theta e^{\phi \i}+\tau}\right ]=\left |1-e^{\phi \i}\right|^2 \Re\left[-\i \frac{e^{\phi \i}}{\theta e^{\phi \i}+\tau}\right]=\left|1-e^{\phi \i}\right|^2\frac{\tau\sin\phi}{\tau^2+\theta^2+2\theta \tau \cos\phi},
\ee
so plugging $z=\theta e^{\phi\i}$ in the expression for $h'(z)$ and factoring out positive terms  we need to prove that
\be
\sum_{n\geq 0}\left [\frac{(n+1)^2}{((n+1)^2+\theta^2+2\theta (n+1) \cos\phi)(n+1+\theta)^2}-\frac{(n-\omega)^2}{((n-\omega)^2+\theta^2+2\theta (n-\omega) \cos\phi)(n-\omega+\theta)^2}\right ]>0.
\ee
It is enough to show that each summand above is positive, which can be rewritten as
\be
\frac{X_1^2}{(X_1^2+\theta^2+2\theta X_1 \cos\phi)(X_1+\theta)^2}>\frac{X_2^2}{(X_2^2+\theta^2+2\theta X_2 \cos\phi)(X_2+\theta)^2}
\ee
for $X_1=n+1$ and $X_2=n-\omega$. Since $\omega<-1$ and both sides of the last inequality are clearly positive, we can equivalently show that
\be
f(X)=(X^2+\theta^2+2\theta X \cos\phi)\(1+\frac{\theta}{X}\)^2
\ee
is increasing in $X\geq 1$ for any fixed $\phi$. Taking derivative and using $\theta\in\(0,\frac{1}{2}\)$ we get
\begin{multline*}
f'(X)=2(X+\theta\cos\phi)\(1+\frac{\theta}{X}\)^2-\frac{2\theta}{X^2}(X^2+\theta^2+2\theta X \cos\phi)\(1+\frac{\theta}{X}\)\\
=2\(1-\frac{\theta^2}{X^2}\) \(X+\theta(1+\cos\phi)+\frac{\theta^2}{X}\)>0.
\end{multline*}
\qed


\subsection{Proof of Proposition \ref{AcontProp}}\label{app:Acont} For convenience, we repeat the statement: 

\begin{prop}Suppose that Assumption \ref{assumptions} is satisfied. Let $\psi(\ve)=\arg(v_{\ve}-\theta)$, where $v_{\ve}$ is the intersection point of $\mathcal C_\theta$ with the circle of radius $\ve>0$ around $\theta$ satisfying $\Im[v_{\ve}]>0$. Then, for sufficiently small fixed $\ve>0$ there exists a depending on $\ve$ constant $b>0$  such that
\be
\Re\left[h(\theta+\ve e^{\pm \phi \i})-h(\theta)\right]>b, \qquad \phi\in[\psi(\ve),2\pi/3].
\ee
\end{prop}
\begin{proof}
Recall that $\theta\in \(0,\frac{1}{2}\)$. Due to the symmetry, it is enough to consider  $\Re[h(\theta+\ve e^{\phi \i})-h(\theta)]$. Since $h$ is a convex sum of finitely many $h_{0,-1,\omega}$, it is enough to consider $h(z)\equiv h_{0,-1,\omega}(z)$.

We consider only $\ve<\theta/2$. By Taylor expansion there exists a constant $\Xi>0$ such that 
\be
\left|h(\theta+\ve e^{\phi\i})-h(\theta) - \frac{h'''(\theta)}{6}\ve^3e^{3\phi\i}-\frac{h^{(4)}(\theta)}{24}\ve^4e^{4\phi\i}\right|<\Xi\ve^5
\ee
Taking real part and using trigonometric equalities $\cos(3\phi)=\sin(3(\phi-\pi/2))$ and $\cos(4\phi)=\cos(4(\phi-\pi/2))$, we get
\be
\Re\left[h(\theta+\ve e^{\phi\i})-h(\theta)\right]\geq \frac{h'''(\theta)}{6}\ve^3\sin(3(\phi-\pi/2))+\frac{h^{(4)}(\theta)}{24}\ve^4\cos(4(\phi-\pi/2))-\Xi\ve^5
\ee
Since $h'''(\theta)>0$ by Lemma \ref{derivatives}, for any $\phi\in[5\pi/8, 2\pi/3]$ we clearly have
\be
\Re\left[h(\theta+\ve e^{\phi\i})-h(\theta)\right]\geq \frac{h'''(\theta)}{6}\ve^3\sin(3\pi/8)-\Lambda_1\ve^4>\Lambda_2\ve^3
\ee
for some constants $\Lambda_1,\Lambda_2>0$ and sufficiently small $\ve$. So it is enough to consider the claim for $\phi\in [\psi(\ve), 5\pi/8]$. Moreover, note that for $\phi\in[\pi/2, 5\pi/8]$ and fixed $\ve$ the function 
\begin{equation}
\label{studiedfunction}
\frac{h'''(\theta)}{6}\ve^3\sin(3(\phi-\pi/2))+\frac{h^{(4)}(\theta)}{24}\ve^4\cos(4(\phi-\pi/2))-\Xi\ve^5
\end{equation}
is increasing as a function of $\phi$, where we have again used Lemma \ref{derivatives} to get $h^{(4)}(\theta)<0$. Hence the claim reduces to showing that the function above is positive at $\phi=\psi(\ve)$.

Recall that $\psi(\ve)$ was defined using $v_\ve$, which can be explicitly described:
\be
v_\ve=\frac{2\theta^2-\ve^2}{2\theta}+\ve\sqrt{1-\frac{\ve^2}{4\theta^2}}\i.
\ee
On the other hand, $v_\ve=\theta+\ve e^{\psi(\ve)\i}$, so we obtain
\be
\frac{v_\ve-\theta}{\ve}=-\frac{\ve}{2\theta}+\sqrt{1-\frac{\ve^2}{4\theta^2}}\i=\cos(\psi(\ve))+\sin(\psi(\ve))\i.
\ee
This implies that $\psi(\ve)=\frac{\pi}{2}+\frac{\ve}{2\theta}+O(\ve^2)$ as $\ve\to 0$. Plugging it into the studied function \eqref{studiedfunction} we get
\be
\frac{h'''(\theta)}{6}\ve^3\sin(3(\psi(\ve)-\pi/2))+\frac{h^{(4)}(\theta)}{24}\ve^4\cos(4(\psi(\ve)-\pi/2))-\Xi\ve^5>\ve^4\(\frac{h'''(\theta)}{4\theta}+\frac{h^{(4)}(\theta)}{24}\)-\Lambda_3\ve^5
\ee
for some constant $\Lambda_3>0$. So the claim follows as long as
\be
\frac{h'''(\theta)}{\theta}+\frac{h^{(4)}(\theta)}{6}>0.
\ee

We prove the last inequality by direct computation. Recall that
\be
h(z)=(\Psi_1(\theta+1)-\Psi_1(\theta-\omega))\(\ln(z)-\frac{z}{\theta}\)-\frac{1}{\theta^2}(\ln\Gamma(z-\omega)-\ln\Gamma(z+1)-z\Psi(\theta-\omega)+z\Psi(\theta+1)),
\ee
where $\omega<-1$. The derivatives are given by
\be
h'''(\theta)=\frac{2}{\theta^3}(\Psi_1(\theta+1)-\Psi_1(\theta-\omega))-\frac{1}{\theta^2}(\Psi_2(\theta-\omega)-\Psi_2(\theta+1)),
\ee
\be
h^{(4)}(\theta)=-\frac{6}{\theta^4}(\Psi_1(\theta+1)-\Psi_1(\theta-\omega))-\frac{1}{\theta^2}(\Psi_3(\theta-\omega)-\Psi_3(\theta+1)).
\ee
Rearranging terms we obtain
\be
\frac{h'''(\theta)}{\theta}+\frac{h^{(4)}(\theta)}{6}=f(-\omega)-f(1)
\ee
where 
\be
f(\tau)=-\frac{1}{\theta^4}\Psi_1(\theta+\tau)-\frac{1}{\theta^3}\Psi_2(\theta+\tau)-\frac{1}{6\theta^2}\Psi_3(\theta+\tau).
\ee
So, it is enough to prove that $f(\tau)$ is increasing in $\tau\geq 1$. Using series representations \eqref{Psi1series} write
\be
f(\tau)=\frac{1}{\theta^4}\sum_{n\geq 0}-\frac{1}{(n+\theta+\tau)^2}+\frac{2\theta}{(n+\theta+\tau)^3}-\frac{\theta^2}{(n+\theta+\tau)^4}=\sum_{n\geq 0}\frac{-(n+\tau)^2}{\theta^4(n+\theta+\tau)^4}
\ee
The claim follows since $-\frac{X^2}{(X+\theta)^4}$ is increasing in $X\geq 1$:
\be
\(-\frac{X^2}{(X+\theta)^4}\)'=\frac{-2X(X+\theta)+4X^2}{(X+\theta)^5}=\frac{2X(X-\theta)}{(X+\theta)^5}>0.
\ee
\end{proof}

\end{appendix}

\end{document}